\documentclass[11pt]{article}
\usepackage[utf8]{inputenc}

\usepackage{
amsmath,amssymb,amsfonts,amsthm}
\usepackage{bm}
\usepackage{colonequals}
\usepackage{comment}
\usepackage{epsfig} \usepackage{latexsym,nicefrac,bbm}
\usepackage
{xspace}
\usepackage{color,fancybox,graphicx,subfigure,fullpage}
\usepackage[top=1.25in, bottom=1.25in, left=1in, right=1in]{geometry}
\usepackage{tabularx}
\usepackage{hyperref}
\usepackage{cleveref}
\usepackage{pdfsync}
\usepackage[boxruled,linesnumbered]{algorithm2e}
\usepackage{siunitx}

\usepackage{multicol}
\usepackage{enumitem}
\usepackage{float}
\usepackage{tikz}
\usepackage{framed}

\makeatletter
\newcommand{\vast}{\bBigg@{4}}
\newcommand{\Vast}{\bBigg@{5}}
\makeatother

\newtheorem{theorem}{Theorem}[section]

\newtheorem{conjecture}[theorem]{Conjecture}
\newtheorem{definition}[theorem]{Definition}
\newtheorem{lemma}[theorem]{Lemma}
\newtheorem{remark}[theorem]{Remark}
\newtheorem{proposition}[theorem]{Proposition}
\newtheorem{corollary}[theorem]{Corollary}

\newcommand{\CC}{\mathbb{C}}

\newcommand{\EE}{\mathbb{E}}

\newcommand{\NN}{\mathbb{N}}
\newcommand{\PP}{\mathbb{P}}

\newcommand{\RR}{\mathbb{R}}

\DeclareSymbolFont{bbold}{U}{bbold}{m}{n}
\DeclareSymbolFontAlphabet{\mathbbold}{bbold}
\newcommand{\One}{\mathbbold{1}}
\newcommand{\Var}{\mathrm{Var}}
\newcommand{\Ham}{\mathrm{H}}
\newcommand{\KT}{\mathrm{KT}}

\renewcommand{\emptyset}{\varnothing}

\DeclareMathOperator*{\argmax}{arg\,max}

\renewcommand{\epsilon}{\varepsilon}
\renewcommand{\Tilde}{\widetilde}
\renewcommand{\hat}{\widehat}
\newcommand{\Unif}{\mathrm{Unif}}
\newcommand{\aln}{\mathrm{align}}
\newcommand{\Ex}{\mathop{\mathbb{E}}}
\newcommand{\Px}{\mathop{\mathbb{P}}}
\newcommand{\cc}{\mathrm{cc}}
\newcommand{\dbar}{\,\|\,}
\newcommand{\Adv}{\mathrm{Adv}}
\newcommand{\Corr}{\mathrm{Corr}}
\newcommand{\MMSE}{\mathrm{MMSE}}

\newcommand{\sQ}{\mathcal{Q}}
\newcommand{\sP}{\mathcal{P}}
\newcommand{\eye}{\bm i}
\newcommand{\Sym}{\mathrm{Sym}}
\newcommand{\inv}{\mathrm{inv}}

\newcommand{\Erdos}{Erd\H{o}s}
\newcommand{\Renyi}{R\'{e}nyi}

\newcommand{\iid}{\mathrm{iid}}

\newcommand{\Rad}{\mathrm{Rad}}
\newcommand{\SparseRad}{\mathrm{SparseRad}}

\newcommand\numberthis{\addtocounter{equation}{1}\tag{\theequation}}

\title{Statistical inference of a ranked community in a directed graph}

\usepackage{authblk}

\author[1]{Dmitriy Kunisky\thanks{Email: \texttt{kunisky@jhu.edu}. Supported in part by ONR Award N00014-20-1-2335 and a Simons Investigator Award to Daniel Spielman.}}

\author[2]{Daniel A.\ Spielman\thanks{Email: \texttt{daniel.spielman@yale.edu}.
Supported in part by ONR Awards N00014-20-1-2335 and N00014-24-1-2611, and a Simons Investigator Award to Daniel Spielman.}}

\author[3]{Alexander S.\ Wein\thanks{Email: \texttt{aswein@ucdavis.edu}. Supported in part by an Alfred P.\ Sloan Research Fellowship and NSF CAREER Award CCF-2338091.}}

\author[2]{Xifan Yu\thanks{Email: \texttt{xifan.yu@yale.edu}.
Supported in part by ONR Award N00014-24-1-2611 and a Simons Investigator Award to Daniel Spielman. Part of this work was done while XY visited ASW at University of California, Davis.}}

\affil[1]{Department of Applied Mathematics \& Statistics, Johns Hopkins University}
\affil[2]{Department of Computer Science, Yale University}
\affil[3]{Department of Mathematics, UC Davis}

\date{}

\begin{document}

\maketitle

\thispagestyle{empty}

\begin{abstract}
    We study the problem of detecting or recovering a \emph{planted ranked subgraph} from a directed graph, an analog for directed graphs of the well-studied planted dense subgraph model.
    We suppose that, among a set of $n$ items, there is a subset $S$ of $k$ items having a latent ranking in the form of a permutation $\pi$ of $S$, and that we observe a fraction $p$ of pairwise orderings between elements of $\{1, \dots, n\}$ which agree with $\pi$ with probability $\frac{1}{2} + q$ between elements of $S$ and otherwise are uniformly random.
    Unlike in the planted dense subgraph and planted clique problems where the community $S$ is distinguished by its unusual \emph{density} of edges, here the community is only distinguished by the unusual \emph{consistency} of its pairwise orderings.
    We establish computational and statistical thresholds for both detecting and recovering such a ranked community.
    In the \emph{log-density} setting where $k$, $p$, and $q$ all scale as powers of $n$, we establish the exact thresholds in the associated exponents at which detection and recovery become statistically and computationally feasible.
    These regimes include a rich variety of behaviors, exhibiting both statistical-computational and detection-recovery gaps.
    We also give finer-grained results for two extreme cases: (1) $p = 1$, $k = n$, and $q$ small, where a full tournament is observed that is weakly correlated with a global ranking, and (2) $p = 1$, $q = \frac{1}{2}$, and $k$ small, where a small ``ordered clique'' (totally ordered directed subgraph) is planted in a random tournament.
\end{abstract}

\clearpage

\pagestyle{empty}

\tableofcontents

\clearpage

\setcounter{page}{1}
\pagestyle{plain}

\section{Introduction}

We study several statistical tasks associated to random directed graphs\footnote{We always refer to simple directed graphs where there is at most $1$ directed edge between any pair of vertices.} $G$ on $n$ vertices.
Taken together, we call the two distributions of $G$ we study the \emph{planted ranked subgraph (PRS) model}.

The aim of the PRS model is to describe situations of the following kind: we observe directed social interactions among a collection of individuals, like the giving of gifts.
Some subset of these individuals form a small community having a strict hierarchy, causing those lower in this hierarchy to more often give gifts to those higher (or vice-versa).
Yet, the frequency of gift-giving in the community overall is the same as in the population at large.
Can we detect or identify this \emph{ranked community}, purely from the effect of its hierarchy on the \emph{direction} in which gifts are given, not the \emph{frequency} with which they are given?\footnote{One could also ask about inference of a ranked community given a combination of information of both kinds, where the community has both an unusual density of edges and an unusual order compatibility of edges---we leave this interesting generalization to future work.}
See, e.g., \cite{gupte2011finding,de2018physical,redhead2022social,wapman2022quantifying} for a small selection of work discussing hierarchy in network data appearing in various social sciences.

We formalize this question into two distributions of $G$.
Under the \emph{null model}, denoted $\mathcal{Q}$, each edge of $G$ is present with probability $\frac{1}{2}p$ in either the forwards or backwards direction, for a total probability $p$ of being present at all, for a parameter $p \in [0, 1]$.
Under the \emph{planted} or \emph{alternative model}, denoted $\mathcal{P}$, we insert a ranked community into $G$.
This structure depends on $p$ and also on further parameters $1 \leq k \leq n$ and $q \in [0, \frac{1}{2}]$.
We then sample $G$ by the following procedure:
\begin{enumerate}
    \item First, each vertex $i\in [n]$ is included in the ranked community, a subset $S \subseteq [n]$, independently with probability $k / n$.
    \item Next, we choose a permutation $\pi \in \Sym(S)$ of the set $S$ uniformly at random.
    When we want to emphasize that this permutation acts only on $S$, we write $\pi = \pi_S$.
    \item Finally, for every $i,j \in S$ with $\pi(i) < \pi(j)$, we add the directed edge $(i,j)$ to $G$ with probability $p(\frac{1}{2} + q)$, add the directed edge $(j,i)$ with probability $p\left(\frac{1}{2} - q\right)$, and add no edge between $i$ and $j$ with the remaining probability $1 - p$.
    For all other pairs $i, j \in [n]$ (where at most one of $i$ and $j$ belongs to $S$), we add a directed edge between $i$ and $j$ with probability $\frac{1}{2}p$ in either direction, for a total probability $p$ of an edge being present at all.
\end{enumerate}
We note that, under both $\sQ$ and $\sP$, the undirected graph $\widetilde{G}$ formed by ``forgetting'' the direction of each edge is merely an \Erdos-\Renyi\ random graph with edge probability $p$.
All the extra structure of $\sP$ therefore lies in the directions of the edges between members of $S$, as proposed above.

We consider two statistical problems.
First, when is it possible to \emph{detect} that a planted ranked subgraph is present in $G$, i.e., to \emph{hypothesis test} between $\sQ$ and $\sP$?
And second, when is it possible given $G \sim \sP$ to \emph{recover} or \emph{estimate} $S$ and $\pi$ accurately from this observation?
We also consider two variations of each question.
First, when is each task achievable \emph{statistically} or \emph{information-theoretically}, that is, with computations of arbitrary runtime permitted?
And second, when is each task achievable \emph{computationally} by a polynomial-time algorithm?

It has been known for some time that statistical and computational hardness can be different: there can be regimes of problems such as the one we propose where it is possible to solve the problem, but only at prohibitive computational cost (e.g., \cite{BPW-2018-GapsNotes,BB-2020-ReducibilityStatCompGaps,wu2021statistical}).
On the other hand, for many problems, it was previously observed that thresholds for the feasibility of detection and recovery coincide; for instance, \cite{Abbe-2017-SBMReview} discusses this point concerning the stochastic block model and its variations.
More recently, natural examples of problems were found where this does not occur, for instance for the planted dense subgraph \cite{CX-2016-ThresholdsPlantedClustersGrowing,schramm2022computational,BJ-2023-DetectionRecoveryGapPDS} and planted dense cycle \cite{MWZ-2023-DetectionRecoveryPlantedCycles} problems.
We will see that different regimes of the PRS model exhibit \emph{both} detection-recovery and statistical-computational gaps of this kind, and we hope that this model will be a valuable example for understanding the interplay of these behaviors.

Our results concern two regimes of the parameters $p = p(n)$, $k = k(n)$, $q = q(n)$.
First, by analogy with the well-studied \emph{planted dense subgraph (PDS)} model of undirected graphs \cite{mcsherry2001spectral,bhaskara2010detecting,hajek2015computational}, we consider $p$, $k$, and $q$ scaling polynomially with $n$, called the \emph{log-density} setting.
Then, we give some finer-grained results about the special case $p = 1$, in which case we observe a complete directed graph, also called a \emph{tournament}.
Within this case, we consider the two extremes of the remaining parameters $k$ and $q$: when $k = n$ and $q$ is small, then we observe a tournament weakly correlated with a global ranking (as we will see, this may be viewed as a digraph-valued version of a spiked matrix model), while when $q = \frac{1}{2}$ and $k$ is small, then we observe a tournament with a small subgraph on which the tournament induces a total ordering (which may be viewed as a digraph version of the planted clique model \cite{Jerrum-1992-LargeCliques,alon1998finding,FK-2000-PlantedClique,sos-clique}).

As a final remark, we will very often work with the \emph{adjacency matrix} of the directed graph $G$.
Unlike the symmetric adjacency matrix of undirected graphs, we take this to be a skew-symmetric matrix $Y \in \{0, \pm 1\}^{n \times n}$ (i.e., having $Y = -Y^{\top}$).
We set $Y_{ij} = 1$ if there is a directed edge from $i$ to $j$, $Y_{ij} = -1$ if there is a directed edge from $j$ to $i$, and $Y_{ij} = 0$ otherwise.
We will view $G$ and $Y$ as interchangeable, and will write $G \sim \sQ$ or $\sP$ and $Y \sim \sQ$ or $\sP$ as equivalent notations.
Several of the algorithms we propose will have straightforward algebraic or spectral interpretations in terms of operations on $Y$.

Before proceeding to the statements of the main results in these various settings, let us define what precisely we mean by detection and recovery in the PRS model.

\subsection{Detection and Recovery}

We will consider the following two standard notions of what it means for an algorithm to achieve detection between $\mathcal{P}$ and $\mathcal{Q}$.
\begin{definition}[Strong and weak detection]
    Consider a sequence of functions $A = A_n$ that take as input a directed graph $G$ on $n$ vertices (or equivalently its adjacency matrix $Y$) and output an element of $\{0,1\}$.
    We say that, as $n \to \infty$:
    \begin{itemize}
        \item $A$ achieves \emph{strong detection} between $\sP$ and $\sQ$ if
        \begin{align*}
        \lim_{n \to \infty} \left(\Px_{G \sim \mathcal{Q}}[A(G) = 1] + \Px_{G \sim \mathcal{P}}[A(G) = 0]\right) = 0;
    \end{align*}
    \item $A$ achieves \emph{weak detection} between $\sP$ and $\sQ$ if, for some $\delta > 0$,
    \begin{align*}
        \limsup_{n \to \infty} \left(\Px_{G \sim \mathcal{Q}}[A(G) = 1] + \Px_{G \sim \mathcal{P}}[A(G) = 0]\right) \le 1 - \delta.
    \end{align*}
    \end{itemize}
    We say that either of these notions is \emph{statistically} possible if any $A$ achieves it, and that it is \emph{computationally} possible if some $A$ computable in polynomial time in $n$ achieves it.
\end{definition}
\noindent
The two error terms in each line above are the \emph{Type I} and \emph{Type II} error probabilities respectively, or the respective probabilities of incorrectly refuting or incorrectly failing to refute the null hypothesis.
The second definition is reasonable since a total error probability of 1 is achieved by the trivial algorithm $A$ that always outputs either 0 or 1.

To formally define recovery under the planted model $\sP$, we must fix metrics by which we will measure the amount of error that an algorithm makes.
This is a little bit subtle, because the planted structure in $\sP$ consists of the two objects $S \subseteq [n]$ and $\pi$ a permutation of $S$.
We define metrics for both objects individually.
\begin{definition}[Hamming distance]
    The \emph{Hamming distance} between $S, T \subseteq [n]$ is $d_{\Ham}(S, T) = |S \triangle T|$, where $\triangle$ denotes the symmetric difference.
\end{definition}

\begin{definition}[Kendall tau distance]
    \label{def:kt}
    The \emph{Kendall tau distance} between $\sigma, \tau$ permutations of two, possibly different, subsets $S, T \subseteq [n]$, respectively, is $d_{\KT}(\sigma, \tau) = \#\{\{i, j\} \in \binom{S \cap T}{2}: i <_{\sigma} j, i>_{\tau} j\}$.
\end{definition}

\begin{definition}[Exact, strong, and weak recovery]
    \label{def:recovery}
    Consider a sequence of functions $A = A_n$ that take as input a directed graph $G$ on $n$ vertices (or equivalently its adjacency matrix $Y$) and output $(\hat{S}, \hat{\pi})$ with $\hat{S} \subseteq [n]$ and $\hat{\pi}$ a permutation on $S$.
    We say that, when $G \sim \sP$ and $n \to \infty$:
    \begin{itemize}
        \item $A$ achieves \emph{exact recovery} if $A(G) = (S, \pi)$ with high probability;
        \item $A$ achieves \emph{strong recovery} if
        \begin{align*}
        \limsup_{n \to \infty} \frac{\EE d_{\Ham}(S, \hat{S})}{k} &= 0, \\
        \limsup_{n \to \infty} \frac{\EE d_{\KT}(\pi, \hat{\pi})}{\binom{k}{2}}  &= 0;
         \end{align*}
        \item $A$ achieves \emph{weak support recovery} if, for some $\delta > 0$,
        \[ \limsup_{n \to \infty} \frac{\EE d_{\Ham}(S, \hat{S})}{k} \leq 1 - \delta. \]
    \end{itemize}
    We say that any of these notions is \emph{statistically} possible if any $A$ achieves it, and that it is \emph{computationally} possible if some $A$ computable in polynomial time in $n$ achieves it.
\end{definition}
\noindent
The idea of exact recovery should be clear.
A sequence of estimators achieves strong recovery if it nearly perfectly recovers $S$, up to a $o(1)$ fraction of vertices erroneously either included or excluded, and also nearly perfect recovers the latent ranking on $S$, up to a $o(1)$ fraction of total pairs $\binom{k}{2}$ having an incorrect pairwise ordering.
Weak support recovery only pertains to the estimate $\hat{S}$ of the community itself (thus for an algorithm aspiring to weak support recovery $\hat{\pi}$ may be arbitrary or just omitted from the setup entirely), and is sensible only when $k = o(n)$, in which case it demands that an algorithm correctly identifies any constant fraction of members of $S$.

\subsection{Main Results: Log-Densities}

\begin{figure}
    \begin{center}
        \includegraphics[scale=0.39]{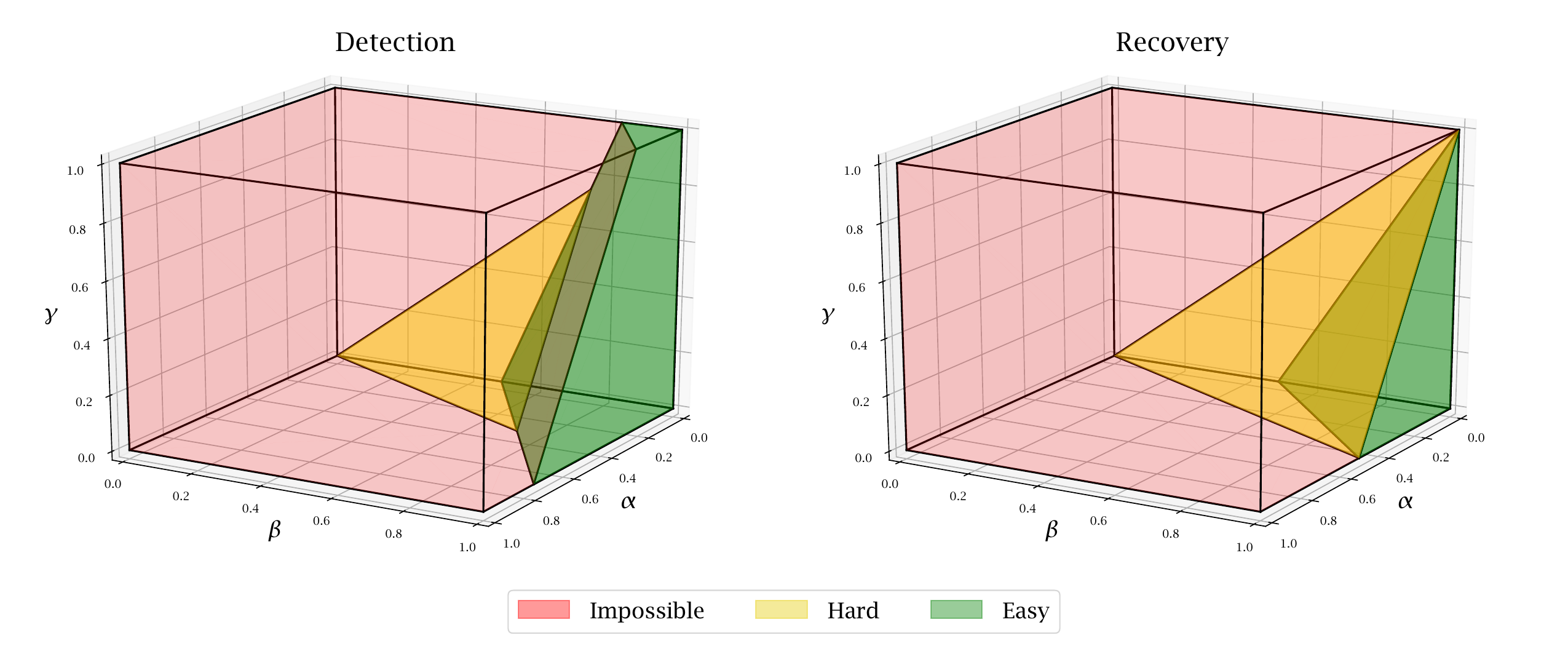}
    \end{center}
    \vspace{-1.5em}
    \caption{Computational and statistical thresholds for detection and recovery in the planted ranked subgraph model in the log-density setting. The green, yellow, and red regions indicate where each problem is computationally tractable, computationally hard but statistically tractable, and statistically impossible, respectively.}
    \label{fig:regimes}
\end{figure}

We now proceed to the first collection of our main results.
By the \emph{log-density setting} we mean a setting of the parameters of the PRS model as follows:
\begin{align*}
    q = q(n) &\colonequals n^{-\alpha}, \\
    k = k(n) &\colonequals n^{\beta}, \\
    p = p(n) &\colonequals n^{-\gamma},
\end{align*}
for some further parameters $\alpha, \beta, \gamma \in (0, 1)$.
Our results on the log-density setting completely characterize the feasibility of statistical and computational detection and recovery in the PRS model for any such choices.
We leave informal for now the precise meaning of our computational lower bounds.
These are carried out in the framework of analysis of \emph{low-degree polynomial algorithms}, which we describe in detail in Section~\ref{sec:low-deg}.
Modulo the details of what those lower bounds mean, our results explicitly decompose the three-dimensional cube of log-density parameters $(\alpha, \beta, \gamma) \in (0, 1)^3$ into regions where each problem is computationally easy, computationally hard but statistically possible, and statistically impossible.
These turn out to be straightforward polyhedral decompositions of the cube, which we illustrate in Figure~\ref{fig:regimes}.

There are eight questions we must answer to establish this decomposition: for each of computational and statistical detection and recovery, we must prove upper and lower bounds describing when it is possible or impossible.
These are addressed in the four theorems below, which each give a pair of upper and lower bounds.

\begin{theorem}[Computational detection in log-density setting]
    \label{thm:log-density-comp-det}
    The following hold:
    \begin{itemize}
        \item If $\beta > \frac{2}{3}\alpha + \frac{1}{3}\gamma + \frac{1}{2}$, then strong detection is computationally possible. It is achieved in this case by computing and thresholding a polynomial of degree 2 in the entries of $Y$.
        \item (Informal) If $\beta < \frac{2}{3}\alpha + \frac{1}{3}\gamma + \frac{1}{2}$, then no sequence of polynomials of degree bounded by $O((\log n)^{2 - \epsilon})$ achieves weak detection.
    \end{itemize}
\end{theorem}

\begin{theorem}[Statistical detection in log-density setting]
    \label{thm:log-density-stat-det}
    The following hold:
    \begin{itemize}
        \item If $\beta > \min\{2\alpha + \gamma, \frac{2}{3}\alpha + \frac{1}{3}\gamma + \frac{1}{2}\}$, then strong detection is statistically possible.
        \item If $\beta < \min\{2\alpha + \gamma, \frac{2}{3}\alpha + \frac{1}{3}\gamma + \frac{1}{2}\}$, then weak detection (and therefore also strong detection) is statistically impossible.
    \end{itemize}
\end{theorem}

\begin{theorem}[Computational recovery in log-density setting]
    \label{thm:log-density-comp-rec}
    The following hold:
    \begin{itemize}
        \item If $\beta > \alpha + \frac{1}{2}\gamma + \frac{1}{2}$, then strong recovery is computationally possible. It is achieved in this case by a spectral algorithm using the Hermitian complex-valued adjacency matrix $\eye Y$.
        This result holds not only under the log-density setting with the condition above, but also under the less stringent assumptions that $q = \omega(\frac{\sqrt{n}}{k\sqrt{p}})$, $p = \Omega(\frac{\log n}{n})$, and $k = \omega(1)$.
        \item (Informal) If $\beta < \alpha + \frac{1}{2}\gamma + \frac{1}{2}$, then no sequence of polynomials of degree bounded by $n^{o(1)}$ achieves weak support recovery.
    \end{itemize}
\end{theorem}

\begin{theorem}[Statistical recovery in log-density setting]
    \label{thm:log-density-stat-rec}
    The following hold:
    \begin{itemize}
        \item If $\beta > 2\alpha + \gamma$, then strong recovery is statistically possible. It is achieved in this case by computing a minor variant of a maximum likelihood estimator.
        \item If $\beta < 2\alpha + \gamma$, then strong recovery is statistically impossible.
    \end{itemize}
\end{theorem}

\subsection{Main Results: Extreme Parameter Scalings}

Finally, we examine two special cases that are not covered by the log-density setting, where some parameters are taken to their extreme values.
In both cases, we fix $p = 1$, in which case we observe \emph{all} edges of $G$.
Such a choice of directions for the complete graph is also called a \emph{tournament}.

\subsubsection{Planted Global Ranking Observed Through a Tournament}

We first consider the special case where we fix $k = n$, $p = 1$, and let $q$ vary.
This is the case of the PRS model where the ranked subgraph is actually the \emph{entire} graph, so we observe a full set of pairwise comparisons that are weakly correlated with $\pi \in \Sym([n])$.

\paragraph{Detection} For detection, the same threshold obtained by plugging $\alpha = 0$, $\beta = 1$ into the results of the log-density framework holds (though it does not follow entirely from our analysis of the log-density setting), which we may further sharpen in this setting as follows.
\begin{theorem}[Detection of global ranking through tournament]\label{thm:detection-th}
    The following hold in the PRS model with $p = 1$, $k = n$, and $q = q(n) \in [0, \frac{1}{2}]$:
    \begin{itemize}
    \item If $q = \omega(n^{-3/4})$, then there exists a polynomial-time algorithm that achieves strong detection.

    \item If $q = O(n^{-3/4})$, then strong detection is statistically impossible.

    \item There exists a constant $c > 0$ such that, if $q \ge c \cdot n^{-3/4}$, then there exists a polynomial-time algorithm that achieves weak detection.

    \item If $q = o(n^{-3/4})$, then weak detection is statistically impossible.
    \end{itemize}
    The polynomial-time algorithms above are the same as those referenced in Theorem~\ref{thm:log-density-comp-det}.
\end{theorem}

\paragraph{Suboptimality of Spectral Algorithm for Detection} As an additional point of comparison, we give the following analysis of a natural spectral algorithm for detection.
If $Y$ is an adjacency matrix as we have defined, then, for $\eye$ the imaginary unit, $\eye Y$ is a Hermitian matrix, and thus this latter matrix has real eigenvalues, whose absolute values are also the singular values of $Y$.
We consider the performance of an algorithm thresholding the largest eigenvalue of this matrix (equivalently, the largest singular value of $Y$), and find that its performance is inferior by a polynomial factor in $n$ in the required signal strength $q$ compared to our algorithm based on a simple low-degree polynomial.
\begin{theorem}[Spectral detection thresholds]
    \label{thm:spectral}
    Suppose $q = c\cdot n^{-1/2}$.
    Then, the following hold:
    \begin{itemize}
        \item If $Y \sim \sQ$, then $\frac{1}{\sqrt{n}}\lambda_{\max}(\eye Y) \to 2$ in probability.
        \item If $Y \sim \sP$ and $c \leq \pi / 4$, then $\frac{1}{\sqrt{n}}\lambda_{\max}(\eye Y) \to 2$ in probability.
        \item If $Y \sim \sP$ and $c > \pi / 4$, then $\frac{1}{\sqrt{n}}\lambda_{\max}(\eye Y) \geq 2 + f(c)$ for some $f(c) > 0$ with high probability.
    \end{itemize}
\end{theorem}
\noindent
In words, the result says that the success of a detection algorithm computing and thresholding the leading order term of $\lambda_{\max}(\eye Y)$ undergoes a transition around the critical value $q = \frac{\pi}{4}n^{-1/2}$, much greater than the scale $q \sim n^{-3/4}$ for which the success of a simpler algorithm computing a low-degree polynomial undergoes the same transition (per Theorem~\ref{thm:detection-th}).
The proof of this result relates $\eye T$ to a complex-valued \emph{spiked matrix model}, a low-rank additive perturbation of a Hermitian matrix of i.i.d.\ noise.

\begin{remark}
    Technically, Theorem~\ref{thm:spectral} does not rule out the existence of a spectral algorithm that successfully distinguishes $\mathcal{P}$ and $\mathcal{Q}$ by thresholding $\lambda_{\max}(\eye Y)$ when $q$ is below $\frac{\pi}{4}n^{-1/2}$, since we only focused on the behavior of $\lambda_{\max}(\eye Y)$ to leading order and ignored the smaller $o(\sqrt{n})$ fluctuations. For some $q < \frac{\pi}{4}n^{-1/2}$, potentially there could exist $\epsilon = \epsilon(n, q) > 0$ such that $\lambda_{\max}(\eye Y) > (2 + \epsilon)\sqrt{n}$ w.h.p.~for $Y \sim \mathcal{P}$ but $\lambda_{\max}(\eye Y) < (2 + \epsilon)\sqrt{n}$ w.h.p.~for $Y \sim \mathcal{Q}$, thus leaving open the possibility of the success of the spectral algorithm below the threshold mentioned in Theorem~\ref{thm:spectral}; however, our results imply that such $\epsilon$ would need to have $\epsilon = o(1)$ as $n \to \infty$. We believe this is an interesting issue to address in future work.
\end{remark}

\paragraph{Recovery} We introduce an extra notion of weak recovery for this setting, which is clearer to define here versus in the general PRS model, since we may compare the performance of a given estimate of the permutation $\pi$ with a random guess.
\begin{definition}[Recovery of global rankings]
    \label{def:strong-recovery-global}
    Consider a sequence of functions $A = A_n$ that take as input a directed graph $G$ on $n$ vertices (or equivalently its adjacency matrix $Y$) and output $\hat{\pi} \in \Sym([n])$.
    We say that, when $G \sim \sP$ with $k = n$ and $n \to \infty$:
    \begin{itemize}
        \item $A$ achieves \emph{strong recovery} if
        \begin{align*}
        \lim_{n \to \infty} \frac{\Ex[d_{\KT}(\pi, \hat{\pi})]}{\binom{n}{2} } = 0,
        \end{align*}
        the same as Definition~\ref{def:recovery} if we view the algorithm as automatically outputting $\hat{S} = [n]$;
        \item $A$ achieves \emph{weak recovery} if there exists $\delta > 0$ such that
    \begin{align*}
        \limsup_{n \to \infty} \frac{\Ex[d_{\KT}(\pi, \hat{\pi})]}{\binom{n}{2} } \le \frac{1}{2} - \delta.
    \end{align*}
    \end{itemize}
    As before, we say that either of these notions is \emph{statistically} possible if any $A$ achieves it, and that it is \emph{computationally} possible if some $A$ computable in polynomial time in $n$ achieves it.
\end{definition}

\begin{theorem}[Recovery thresholds] \label{thm:recovery-th}
    Suppose $0\le q = q(n) \le 1/4$.
    The following hold:
    \begin{itemize}
    \item If $q = \omega(n^{-1/2})$, then a polynomial-time algorithm achieves strong recovery.

    \item If $q = \Theta(n^{-1/2})$, then strong recovery is statistically impossible.

    \item If $q = \Theta(n^{-1/2})$, then a polynomial-time algorithm achieves weak recovery.

    \item If $q = o(n^{-1/2})$, then weak recovery is statistically impossible.
    \end{itemize}
\end{theorem}
\noindent
In fact, while the definitions above only concern expectations of $d_{\KT}(\pi, \hat{\pi})$, we also give high-probability results when we present our proofs in Section~\ref{sec:pf:thm:recovery-th}.

Unlike the more complicated spectral recovery algorithm in the log-density setting when the ranking is planted on only a subset of vertices, here our results in fact show that the following simpler recovery algorithm is optimal.
\begin{definition}[Ranking by wins]
    \label{def:rank-by-wins}
    The \emph{Ranking By Wins} algorithm takes as input a directed adjacency matrix $Y$ of a tournament and outputs a permutation $\hat{\pi} \in \Sym([n])$ in the following way:
    \begin{enumerate}
    \item For each $i\in [n]$, compute a score $s_i = \sum_{k\in [n]}Y_{i,k}$,
    \item Rank the elements $i\in [n]$ according to the scores $s_i$ from the highest to the lowest, under an arbitrary tie-breaking rule (say, ranking $i$ below $j$ if $i < j$ when $s_i = s_j$).
    \end{enumerate}
\end{definition}

\begin{remark}
    In the proof of Theorem~\ref{thm:recovery-th}, we actually obtain a quantitative bound for the recovery error.
    Namely, we show that the Ranking By Wins algorithm outputs a permutation $\hat{\pi}$ that achieves
    \begin{align}
        \frac{\Ex[d_\KT(\pi, \hat{\pi})]}{\binom{n}{2}} \le \frac{C}{q\sqrt{n}} \cdot \exp\left(-q^2 n\right)
    \end{align}
    for a constant $C > 0$. We also establish the following lower bound on the expected error achievable by any algorithm $A$:
    \begin{align}
        \frac{\Ex[d_\KT(\pi, A(G))]}{\binom{n}{2}} \ge \frac{1}{2}\max\left\{1 - \frac{4q\sqrt{n}}{\sqrt{\frac{1}{4} - q^2}}, \frac{1}{2} \exp\left(-\frac{8 q^2 n}{\frac{1}{4} - q^2}\right)\right\}.
    \end{align}
\end{remark}

\paragraph{Alignment Maximization}
Finally, we state an ancillary result on finding a permutation that is maximally \emph{aligned} with an observed tournament, maximizing the objective function:
\begin{equation}
    \aln(\hat{\pi}, G) = \sum_{(i,j)\in E(G)} \big(\One\{\hat{\pi}(i) < \hat{\pi}(j)\} - \One\{\hat{\pi}(i) > \hat{\pi}(j)\}\big).
\end{equation}

Let us first draw a connection to maximum likelihood estimation to explain why the alignment objective is an interesting one.
Let $G \sim \sP$.
The likelihood function $\mathcal{L}(\hat{\pi} \mid G)$ in this case can be expressed as
\begin{align*}
    \mathcal{L}(\hat{\pi} \mid G) &= \Px_{G \sim \mathcal{P}}[G \mid \hat{\pi}]\\
    &= \prod_{(i,j)\in E(G)}\left(\frac{1}{2} + q\right)^{\One\{\hat{\pi}(i) < \hat{\pi}(j)\}} \left(\frac{1}{2} - q\right)^{\One\{\hat{\pi}(i) > \hat{\pi}(j)\}}\\
    &= \left(\frac{1}{2} + q\right)^{\sum_{(i,j)\in E(G)} \One\{\hat{\pi}(i) < \hat{\pi}(j)\} }\left(\frac{1}{2} - q\right)^{\sum_{(i,j)\in E(G)} \One\{\hat{\pi}(i) > \hat{\pi}(j)\} }\\
    &= \left(\frac{1}{2} - q\right)^{\frac{1}{2}\binom{n}{2}} \left(\frac{1}{2} + q\right)^{\frac{1}{2}\binom{n}{2}} \left(\frac{\frac{1}{2} + q}{\frac{1}{2} - q}\right)^{\frac{1}{2}\cdot \aln(\hat{\pi}, G)},
\end{align*}
where the last line follows from $\sum_{(i,j)\in E(G)} \One\{\hat{\pi}(i) < \hat{\pi}(j)\} + \sum_{(i,j)\in E(G)} \One\{\hat{\pi}(i) > \hat{\pi}(j)\} = \binom{n}{2}$. Thus, the maximizer of the alignment objective has the pleasant statistical interpretation of being the maximum likelihood estimator of the hidden permutation under the planted distribution $\mathcal{P}$, given the observation $G$.
Unfortunately, computing the maximum likelihood estimator or (equivalently) optimizing the alignment objective for a general worst-case input $G$ is NP-hard \cite{alon2006ranking}.

Nevertheless, as our results below will show, when we consider draws from the planted model when strong recovery is information-theoretically possible, then the same simple Ranking By Wins algorithm nearly maximizes the likelihood.

\begin{theorem}[Alignment maximization]\label{thm:alignment-approx}
    Suppose $0\le q = q(n) \le 1/4$. For $q = \omega(n^{-1/2})$, there exists a polynomial-time algorithm which, given $G \sim \mathcal{P}$, outputs a permutation $\hat{\pi} \in \Sym([n])$ that with high probability satisfies
    \begin{equation}
        \aln(\hat{\pi}, G) \ge (1 - o(1) ) \cdot \max_{\Tilde{\pi} \in S_n} \, \aln(\Tilde{\pi}, G).
    \end{equation}
\end{theorem}

\begin{remark}
    \label{rem:q-1-4}
    We focus on the case $q \leq 1/4$ for technical reasons. In the case of $\gamma \ge 1/4$, the maximum likelihood estimator can be computed exactly in polynomial time with high probability~\cite{braverman2007noisy}, which equivalently exactly maximizes the alignment objective.
\end{remark}

While this algorithm will be the same Ranking By Wins algorithm as for estimating the hidden permutation under the Kendall tau distance, we emphasize that a good such estimator $\hat{\pi}$ does \emph{not} necessarily \emph{a priori} give a good approximate maximizer of the alignment objective.
Yet, it turns out that the Ranking By Wins estimator does have this property, which requires further analysis of its behavior.

\subsubsection{Planted Ordered Clique in a Tournament}

Finally, we consider the case of the PRS model with $p = 1$, $q = \frac{1}{2}$, and $k = k(n)$ varying.
This is a directed version of the planted clique model, where we observe a tournament $Y$ drawn from $\sP$ having a hidden subset $S \subseteq [n]$ and a latent permutation $\pi_S$ on $S$ such that all the directed edges between vertices in $S$ are oriented according to $\pi_S$.

While Theorem~\ref{thm:log-density-comp-rec} shows that a spectral algorithm successfully recovers $S$ and $\pi_S$ approximately once $k = \omega\left(\sqrt{n}\right)$, we can actually do better. Here we show, analogous to the results of \cite{alon1998finding} on the undirected planted clique model, that a slightly modified spectral algorithm works all the way down to $k = \Omega\left(\sqrt{n}\right)$ and achieves \textbf{exact} recovery of $S$ and $\pi_S$, rather than just strong recovery.

\begin{theorem}[Recovery of planted ordered clique] \label{thm:planted-ordered-clique-rec}
    Fix $p = 1$ and $q = \frac{1}{2}$. There exists a constant $C > 0$ such that if $k = k(n) \ge C \sqrt{n}$, then there is a polynomial-time algorithm that achieves exact recovery (in the sense of Definition~\ref{def:recovery}).
\end{theorem}
Adapting another idea of \cite{alon1998finding}, we may reduce the constant in front of $\sqrt{n}$ and further show the following.
\begin{corollary} \label{cor:planted-ordered-clique-rec}
    Fix $p = 1$ and $q = \frac{1}{2}$. For any constant $c > 0$, if $k = k(n) \ge c \sqrt{n}$, then there exists a polynomial time algorithm (with runtime depending on $c$) that achieves exact recovery.
\end{corollary}

\subsection{Related Work}

Numerous models for random digraphs, either fully or partially observed, having some hidden structure have been proposed in the literature.
One of the most popular for generating noisy pairwise comparisons between $n$ elements, the Bradley-Terry-Luce (BTL) model, was introduced in \cite{bradley1952rank, luce1959individual}. In the BTL model, there is a hidden \emph{preference vector} $w = (w_1, \dots, w_n) \in \mathbb{R}_{> 0}^n$, such that one observes a noisy label $T_{i, j}$ that takes value $+1$ with probability $w_i / (w_i + w_j)$, and $-1$ with probability $w_j / (w_i + w_j)$.
Such models are usually studied in terms of query complexity, with multiple independent queries of the same pair $(i, j)$ allowed.
There have been extensive studies of when one can approximate the preference vector $w$ (see e.g.~\cite{negahban2012iterative, rajkumar2014statistical, shah2018simple}, though \cite{shah2018simple} actually work under a substantially more general model than BTL) or recover the top $k$ elements (see e.g.~\cite{jang2016top, chen2015spectral, mohajer2017active}) in the BTL model.
But, even aside from not including community structure, the BTL model is quite different from ours, because the magnitudes of the $w_i$ can create a broader range of biases in the observations than our single parameter $q$.

The case $k = n$, $p = 1$ of our model, a global tournament weakly correlated with a hidden ranking, is often referred to as a \emph{noisy sorting} model.
For $q > 0$ a constant, the results in \cite{mao2018minimax,wang2022noisy}, further improved by \cite{gu2023optimal}, give tight bounds on the number of noisy comparisons needed to recover the hidden permutation.
In this same ``high signal'' setting, \cite{braverman2007noisy} proposed an efficient algorithm that with high probability \emph{exactly} computes the MLE of the hidden permutation, for the signal scaling $q = \Theta(1)$.\footnote{We use the term \emph{with high probability} for sequences of events occurring with probability converging to 1 as $n \to \infty$.}
Moreover, it is shown that the MLE is close to the hidden permutation.
A faster $O(n^2)$-time algorithm is given in \cite{klein2011tolerant} in the same setting as \cite{braverman2007noisy}, but that algorithm does not output the exact MLE and has a worse guarantee on the total ``dislocation distance.''
As our results show, the scaling $q = \Theta(1)$ is also far greater than the thresholds for efficiently recovering or detecting a hidden ranking with other algorithms.

Improving on this scaling, \cite{rubinstein2017sorting} gave an efficient algorithm that again with high probability exactly computes the MLE, now for $q = \Omega((\log \log n / \log n)^{1/6})$.
The sequence of works \cite{geissmann2017sorting, geissmann2018optimal} yielded an algorithm that achieves the same approximation guarantee as in \cite{braverman2007noisy} with an improved running time of $O(n \log n)$, but that again does not compute the exact MLE and operates under an even more stringent assumption that $q > 7/16$ is a sufficiently large constant.

The Ranking By Wins algorithm has appeared in various guises in the past.
It may be viewed as a relative of the \emph{Condorcet method} in the theory of elections and social choice \cite{fishburn1977condorcet}.
More recently, it has appeared in works including \cite{shah2018simple,shah2019feeling,chatterjee2019estimation}.
Some of these results are close to our analysis of the noisy sorting setting; e.g., \cite{chatterjee2019estimation} obtains a threshold for recovery of a certain signal matrix in that setting that is worse than our Theorem~\ref{thm:recovery-th} only by logarithmic factors.
None of these or the previously mentioned works consider ranking problems in the presence of community structure, however.

Spectral algorithms for sorting or ranking problems have appeared in the past such as in \cite{chatterjee2015matrix,shah2016stochastically}.
But, it appears that our work is the first to directly link such questions to the literature on fine-grained results on spiked matrix models, and also to observe that such an algorithm (at least in our noisy sorting model) is inferior to a seemingly more naive combinatorial one for the detection task. In the log-density setting with the presence of a hidden ranked community, we show that a spectral method not only recovers the hidden community but also the latent permutation down to the computational threshold evidenced by the so-called low-degree conjecture as stated in Conjecture~\ref{conj:low-degree}.

\section{Notations}

We write $\Sym(S)$ for the symmetric group on a set $S$.
If $\pi \in \Sym(S)$, we sometimes write $\pi_S$ to emphasize the set on which $\pi$ acts (especially when $S \subseteq [n]$ but $\pi \in \Sym(S)$ rather than $\Sym([n])$).

We use $\Rad(q)$ to denote the distribution of a skewed Rademacher random variable that takes value $1$ with probability $q$ and $-1$ with probability $1-q$. We use $\SparseRad(p, q)$ to denote the distribution of a random variable that takes value $0$ with probability $1-p$, and follows $\Rad(q)$ with probability $p$. We write $d_{\mathrm{TV}}(\cdot, \cdot)$ for the total variation distance between two probability measures, $d_{\mathrm{KL}}(\cdot, \cdot)$ for the Kullback-Leibler divergence, and $\chi^2(\cdot \dbar \cdot)$ for the $\chi^2$-divergence.

The directed graphs in this paper are always simple, in the sense that between every pair of vertices $i,j \in [n]$, there is at most one directed edge. The symbol $Y$ always denotes the skew-symmetric adjacency matrix of a directed graph, with entries in $\{-1, 0, +1\}$.
For a general $n \times n$ matrix $Z$ and $A \subseteq \binom{[n]}{2}$ a subset of edge indices, we write
\[ Z^A \colonequals \prod_{\{i,j\}\in A: i<j } Z_{i,j}. \]
We also write $Z^{\circ 2}$ for the entrywise square of $Z$.
Note that for $Y$ a directed adjacency matrix, $Y^{\circ 2}$ is an ordinary graph adjacency matrix, of the graph formed from forgetting the directions in the graph whose adjacency matrix $Y$ gave.

We write $\binom{X}{k}$ for the subsets of $X$ of size $k$.
Most often, we will run into $\binom{[n]}{2}$ in our arguments.
We use letters $A, B$ for subsets of $\binom{[n]}{2}$, which we also interpret as graphs on a set of vertices labelled by $[n]$.
In this situation, we write $V(A)$ for the vertex set of $A$, including only those vertices that are incident with some edge of $A$, and $\cc(A)$ for the number of connected components, likewise omitting isolated vertices.

For a permutation $\pi$ of a set $S \subseteq [n]$, we write $i >_{\pi} j$ if $\pi(i) > \pi(j)$, and write
\[ \pi(i, j) \colonequals (-1)^{\One\{i >_{\pi} j\}}. \]
The matrix of these values, with zeroes on the diagonal, gives the adjacency matrix of the directed graph associated to the total ordering $\pi$ gives to $S$.

The asymptotic notations $o(\cdot), O(\cdot), \Omega(\cdot), \omega(\cdot), \Theta(\cdot), \ll, \gg, \lesssim, \gtrsim$ have their usual definitions, always with respect to the limit $n \to \infty$.
Subscripts on these symbols refer to quantities the implicit constants depend on.

\section{Preliminaries}

\subsection{Low-Degree Polynomial Algorithms}
\label{sec:low-deg}

Our computational lower bounds will be in the framework of viewing polynomials as algorithms for statistical problems, with the polynomial degree as a measure of complexity.
This idea originates in the literature on sum-of-squares optimization, where it plays an important technical role in the lower bound technique of \emph{pseudocalibration}.
Since then, it has become an independent form of evidence of computational hardness of statistical problems \cite{sos-clique,sos-detect,HS-bayesian,hopkins-thesis,kunisky2019notes}.

Much of the early work \cite{sos-detect,HS-bayesian,hopkins-thesis} concerned simple hypothesis testing problems, like in our case the problem of trying to distinguish $\sQ$ from $\sP$, in particular when one distribution is a ``natural'' null distribution, like our $\sQ$.
Later extensions treated the complexity of other statistical tasks, including recovery (or \emph{estimation}, in statistical terminology) \cite{schramm2022computational}, which we will use for our results.

We first specify what it means for a polynomial to solve a detection task between two probability measures $\sQ$ and $\sP$.
\begin{definition}[Strong and weak separation]
    \label{def:separation}
    Let $\sQ = \sQ_n$ and $\sP = \sP_n$ be two sequences of probability measures over $\RR^{N}$ for some $N = N(n)$.
    We say that a sequence of polynomials $f(Y) = f_n(Y_1, \dots, Y_N)$ \emph{strongly separates} $\sQ$ from $\sP$ if
    \[ \EE_{\sP}[f(Y)]  - \EE_{\sQ}[f(Y)] = \omega(\max\{\sqrt{\Var_{\sQ}[f(Y)]}, \sqrt{\Var_{\sP}[f(Y)]}\}), \]
    and that it \emph{weakly separates} $\sQ$ from $\sP$ if
    \[ \EE_{\sP}[f(Y)]  - \EE_{\sQ}[f(Y)] = \Omega(\max\{\sqrt{\Var_{\sQ}[f(Y)]}, \sqrt{\Var_{\sP}[f(Y)]}\}), \]
    both requirements referring to the limit $n \to \infty$.
\end{definition}
\noindent
In words, strong separation means that Chebyshev's inequality implies that thresholding $f$ at, say, $(\EE_{\sP}[f(Y)]  + \EE_{\sQ}[f(Y)]) / 2$ distinguishes $\sQ$ and $\sP$ with high probability (i.e., achieves strong detection). Similarly, weak separation implies that thresholding $f$ achieves weak detection, but here the proof is more subtle and the threshold might not be the midpoint between the means. The claim for weak separation is stated in Proposition~\ref{prop:weak-sep} below, with the proof deferred to Appendix~\ref{app:weak-sep}. This is a strengthening of \cite[Proposition~6.1]{fp}, which shows that \emph{some} function of $f(Y)$ (not necessarily a threshold function) achieves weak detection.

\begin{proposition}\label{prop:weak-sep}
If $f = f_n$ weakly separates $\sQ$ from $\sP$ then there exists a choice of threshold $t = t_n$ such that the test $A(Y) = \One\{f(Y) \ge t\}$ achieves weak detection between $\sQ$ and $\sP$.
\end{proposition}

The following measurement of ``one-sided separation'' is a useful proxy for these notions.
\begin{definition}[Low-degree advantage]
    \label{def:adv}
    For $\sQ$ and $\sP$ as above, we define
    \begin{equation}
        \Adv_{\leq D}(\sQ, \sP) \colonequals \sup_{\substack{f \in \RR[Y]_{\leq D} \\ \EE_{Y \sim \sQ} f(Y)^2 \neq 0}} \frac{\EE_{Y \sim \sP} f(Y)}{\sqrt{\EE_{Y \sim \sQ} f(Y)^2}}.
    \end{equation}
\end{definition}
\noindent
In particular, bounding the advantage shows that separation is impossible in the following ways.
\begin{proposition}[{\cite[Lemma~7.3]{grp-test}}]
    \label{prop:adv-bounds}
    In the setting of Definition~\ref{def:separation}, if $\Adv_{\leq D}(\sQ, \sP) = O(1)$ for some choice of $D = D(n)$, then there exists no sequence of $f_n \in \RR[Y]$ with $\deg(f_n) \leq D(n)$ that strongly separates $\sQ$ from $\sP$.
    If $\Adv_{\leq D}(\sQ, \sP) = 1 + o(1)$, then there exists no such sequence that weakly separates $\sQ$ from $\sP$.
\end{proposition}

\begin{remark}
    The advantage \emph{diverging} only shows a part of the strong separation criterion, since we must also bound the variance of the polynomial involved under $\sP$.
    A number of recent examples show that the advantage may in fact diverge while still no low-degree polynomial achieves strong separation~\cite{fp,grp-test,subhypergraph,graph-matching}.
\end{remark}

For reconstruction tasks, success is naturally measured in terms of mean squared error.
We focus on the task of recovering just the \emph{support} of the ranked community with a low-degree polynomial, not the permutation itself---a kind of weak support recovery by low-degree polynomials.

\begin{definition}[Low-degree minimum mean squared error~\cite{schramm2022computational}] \label{def:MSE}
Under the distribution $\sP$ of the PRS model, write $\theta \in \{0, 1\}^n$ for the indicator vector of membership in the planted community $S$.
We then write
\[ \MMSE_{\leq D}(\sP) \colonequals \inf_{f \in \RR[Y]_{\leq D}^n} \Ex_{(\theta, Y) \sim \sP} \|f(Y) - \theta\|^2 = n \inf_{f \in \RR[Y]_{\leq D}} \Ex_{(\theta, Y) \sim \sP} (f(Y) - \theta_1)^2. \]
\end{definition}

 Following Fact~1.1 of~\cite{schramm2022computational}, this can be equivalently formulated in terms of a ``low-degree correlation'':
\begin{equation}
\MMSE_{\leq D}(\sP) = \EE\|\theta\|^2 - n\Corr_{\leq D}(\sP)^2 = k - n\Corr_{\leq D}(\sP)^2,
\end{equation}
where $\Corr_{\leq D}$ is defined below.

\begin{definition}[Low-degree correlation with $\theta_1$]
    For $\sP$ as above, viewed as a joint distribution over $(\theta, Y)$, we define
    \begin{equation}
        \Corr_{\leq D}(\sP) \colonequals \sup_{\substack{f \in \RR[Y]_{\leq D} \\ \EE_{Y \sim \sP} f(Y)^2 \neq 0}} \frac{\EE_{(\theta, Y) \sim \sP}\, \theta_1 f(Y)}{\sqrt{\EE_{Y \sim \sP} f(Y)^2}}.
    \end{equation}
\end{definition}

\noindent
We thus say that weak support recovery is hard for degree $D = D(n)$ polynomials if $\MMSE_{\leq D}(\sP) = k(1 - o(1))$, which by the above is equivalent to having $\Corr_{\leq D}(\sP)^2 \ll \frac{k}{n}$.

\subsection{Low-Degree Conjecture}

One reason why the class of low-degree polynomial algorithms is interesting is due to the following low-degree conjecture, which is an informal statement of~\cite[Conjecture~2.2.4]{hopkins-thesis}.
\begin{conjecture}[Informal] \label{conj:low-degree}
     For ``sufficiently nice'' $\sQ$ and $\sP$, if there exists $\varepsilon > 0$ and $D = D(n) \ge (\log n)^{1+\varepsilon}$ for which $\Adv_{\le D}(\sQ, \sP)$ remains bounded as $n \to \infty $, then there is no polynomial-time algorithm that achieves strong detection between $\sQ$ and $\sP$.
\end{conjecture}

\begin{remark}
    We remark that the original conjecture in \cite{hopkins-thesis} is stated in terms of the notion of coordinate degree rather polynomial degree, but it turns out that for spaces where each coordinate is supported on a constant-sized alphabet, the two notions of degree are equivalent up to a constant.
\end{remark}

Therefore, hardness results against the class of low-degree polynomial algorithms may on the one hand be viewed as unconditional lower bounds for a class of general algorithms in the sense stated in Proposition~\ref{prop:adv-bounds}, and on the other hand as evidence that no polynomial-time algorithm works for the detection task, provided that we believe Conjecture~\ref{conj:low-degree}.

\subsection{Low-Degree Analysis of Planted Ranked Subgraph Model}

We develop some tools for working with polynomials and their expectations under the PRS distributions.
The following gives some initial calculations of expectations of monomials.
\begin{proposition}[Planted expectations] \label{prop:fourier-expectation}
    Let $A \subseteq \binom{[n]}{2}$. Then,
    \begin{align*}
        \Ex_{Y \sim \mathcal{P} }[Y^A] &= \left(\frac{k}{n}\right)^{|V(A)|}(2pq)^{|A|} \Ex_{\pi \sim \mathrm{Unif}(\Sym([n]))}\left[(-1)^{\sum_{\{i,j\}\in A: i < j} \One\{\pi(i) > \pi(j)\} } \right].
    \end{align*}
\end{proposition}

\begin{proof}
    Recall that, to sample a directed graph $Y$ from $\mathcal{P}$, one may first sample a permutation $\pi \in \Sym([n])$ uniformly at random and a random subset $S \subseteq [n]$ that includes every vertex with probability $k/n$, and then generate $Y$ that correlates suitably with $\pi$ on $S$. For a fixed pair of $(S, \pi)$, let us denote by $\sP_{S,\pi}$ the distribution $\sP$ conditional on the ranked community being $S$ and the hidden permutation being $\pi$. In particular, notice that $\mathcal{P}_{S,\pi}$ is a product distribution, where each $Y_{i,j}$ is chosen independently between all pairs of $i,j$ (but with different distributions depending on $(S,\pi)$). Then, we have
    \begin{align*}
        \Ex_{Y \sim \mathcal{P} }[Y^A] &= \Ex_{S}\Ex_{\pi \sim \mathrm{Unif}(\Sym([n]))} \Ex_{Y \sim \mathcal{P}_{S,\pi}}[Y^S]\\
        &= \Ex_{S}\Ex_{\pi \sim \mathrm{Unif}(\Sym([n]))} \prod_{\{i,j\} \in A: i<j} \Ex_{Y \sim \mathcal{P}_{S,\pi}}[Y_{i,j}] \\
        &= \Ex_{S}\Ex_{\pi \sim \textrm{Unif}(\Sym([n]))} \prod_{\{i,j\} \in A: i<j} \left(\One\{i,j \in S\}(-1)^{\One\{\pi(i)>\pi(j)\}} (2pq)\right) \\
        &= \left(\frac{k}{n}\right)^{|V(A)|}(2pq)^{|A|} \Ex_{\pi \sim \mathrm{Unif}(\Sym([n]))}\left[(-1)^{\sum_{\{i,j\}\in A: i < j} \One\{\pi(i) > \pi(j)\} } \right],
    \end{align*}
    completing the proof.
\end{proof}

\begin{proposition}[Component-wise independence] \label{prop:component-independence}
    Let $A \subseteq \binom{[n]}{2}$ be $A = A_1 \sqcup A_2$ with two vertex-disjoint components $A_1$ and $A_2$. Then,
    \begin{align*}
        \Ex_{Y \sim \mathcal{P} }[Y^{A}] = \Ex_{Y \sim \mathcal{P} }[Y^{A_1}]\Ex_{Y \sim \mathcal{P} } [Y^{A_2}].
    \end{align*}
\end{proposition}

\begin{proof}
    Since $A_1, A_2$ are vertex disjoint, the distribution of $Y^{A_1}$ and $Y^{A_2}$ under $\mathcal{P}$ are independent, as we can independently sample a permutation $\pi_1$ on the vertex set of $A_1$ and a permutation $\pi_2$ on the vertex set of $A_2$, and then sample the directed edges used in $A_1$ and $A_2$ which correlate with $\pi_1$ and $\pi_2$ respectively. Thus, $\Ex_{\mathcal{P} }[Y^{A_1 \sqcup A_2}] = \Ex_{\mathcal{P} }[Y^{A_1}] \Ex_{\mathcal{P} }[Y^{A_2}]$.
\end{proof}

\begin{proposition}[Adjacency matrix monomial bounds]
    \label{prop:adj-moments}
    Let $A, B \subseteq \binom{[n]}{2}$ be edge-disjoint.
    Call $A$ \emph{even} if, when viewed as a graph, all of its connected components have an even number of edges.
    Then
    \begin{align*}
    \Ex_{Y \sim \mathcal{Q} }[Y^A]  &= \One\{A = \emptyset\}, \\
    \Ex_{Y \sim \mathcal{Q} } [(Y^{\circ 2})^A] &= p^{|A|},\\
    \Ex_{Y \sim \mathcal{Q} } [Y^A(Y^{\circ 2})^B]  &= p^{|B|} \One\{A = \emptyset\},\\
    \left| \Ex_{Y \sim \mathcal{P} }[Y^A]\right| &\le \left(\frac{k}{n}\right)^{|V(A)|} (2pq)^{|A|} \One\{A \text{ even}\} \numberthis \label{ineq:monomial-bd-1},\\
    \left|\Ex_{Y \sim \mathcal{P}}[Y^A(Y^{\circ 2})^B]\right| &\le \left(\frac{k}{n}\right)^{|V(A)|}p^{|B|} (2pq)^{|A|} \One\{A \text{ even}\}, \numberthis \label{ineq:monomial-bd-2}\\
    \Ex_{Y \sim \mathcal{P}}[Y_{i,j}Y_{i,k}] &= \frac{4}{3} \left(\frac{k}{n}\right)^{3}p^{2} q^{2}.
\end{align*}
\end{proposition}

\begin{proof}
    The first three identities are easy to verify, and the last identity can be computed using Proposition~\ref{prop:fourier-expectation}. We will mainly discuss how to derive \eqref{ineq:monomial-bd-1} and \eqref{ineq:monomial-bd-2}, and in particular, the no-odd-connected-component condition.

    Let us first consider \eqref{ineq:monomial-bd-1}. By Proposition~\ref{prop:fourier-expectation} and Proposition~\ref{prop:component-independence},
    \begin{align*}
        \left|\Ex_{\sP}[Y^{A}]\right| &= \prod_{\delta \in C(A)} \left|\Ex_{\sP}[Y^{\delta}]\right|
        \intertext{where $C(A)$ denotes the collection of connected components of $A$,}
        &= \prod_{\delta \in C(A)} \left|\left(\frac{k}{n}\right)^{|V(\delta)|}(2pq)^{|\delta|} \Ex_{\pi \sim \mathrm{Unif}(\Sym([n]))}\left[(-1)^{\sum_{\{i,j\}\in \delta: i < j} \One\{\pi(i) > \pi(j)\} } \right]\right|\\
        &= \left(\frac{k}{n}\right)^{|V(A)|}(2pq)^{|A|}\prod_{\delta \in C(A)} \left|\Ex_{\pi \sim \mathrm{Unif}(\Sym([n]))}\left[(-1)^{\sum_{\{i,j\}\in \delta: i < j} \One\{\pi(i) > \pi(j)\} } \right]\right|.
    \end{align*}
    Clearly, for any $\delta$, we have $\left|\Ex_{\pi \sim \mathrm{Unif}(\Sym([n]))}\left[(-1)^{\sum_{\{i,j\}\in \delta: i < j} \One\{\pi(i) > \pi(j)\} } \right]\right| \le 1$. We will argue that if $|\delta|$ is odd, then $\Ex_{\pi \sim \mathrm{Unif}(\Sym([n]))}\left[(-1)^{\sum_{\{i,j\}\in \delta: i < j} \One\{\pi(i) > \pi(j)\} } \right] = 0$. Let us denote
    \[ \mathrm{swaps}(\pi, \delta) \colonequals \sum_{\{i,j\}\in \delta: i < j} \One\{\pi(i) > \pi(j)\}. \]
    For any $\pi \in \Sym([n])$, we let $\text{rev}(\pi) \in \Sym([n])$ denote the reverse of $\pi$, given by $\text{rev}(\pi)(i) = n+1-\pi(i)$ for all $i\in [n]$. We may then pair up $\pi$ with $\text{rev}(\pi)$ to get
    \begin{align*}
        &\hspace{-1cm}\Ex_{\pi \sim \mathrm{Unif}(\Sym([n]))}\left[(-1)^{\sum_{\{i,j\}\in \delta: i < j} \One\{\pi(i) > \pi(j)\} } \right] \\
        &= \Ex_{\pi \sim \mathrm{Unif}(\Sym([n]))}\left[(-1)^{\mathrm{swaps}(\pi, \delta) } \right]\\
        &= \frac{1}{2} \cdot \Ex_{\pi \sim \textrm{Unif}(\Sym([n]))}\left[(-1)^{\mathrm{swaps}(\pi, \delta)} + (-1)^{\mathrm{swaps}(\mathrm{rev}(\pi, \delta))}\right] .
    \end{align*}
    For any fixed $\pi \in \Sym([n])$, we observe that
    \[ \mathrm{swaps}(\pi,\delta) + \mathrm{swaps}(\mathrm{rev}(\pi,\delta)) = |\delta|.\]
    Since $|\delta|$ is odd, for every $\pi \in \Sym([n])$, one of the quantities above is odd and the other is even. We thus find that $\Ex_{\pi \sim \mathrm{Unif}(\Sym([n]))}\left[(-1)^{\sum_{\{i,j\}\in \delta: i < j} \One\{\pi(i) > \pi(j)\} } \right] = 0$ if $|\delta|$ is odd. This concludes the proof that
    \begin{align*}
        \left|\Ex_{\sP}[Y^{A}]\right| &= \left(\frac{k}{n}\right)^{|V(A)|}(2pq)^{|A|}\prod_{\delta \in C(A)} \left|\Ex_{\pi \sim \mathrm{Unif}(\Sym([n]))}\left[(-1)^{\sum_{\{i,j\}\in \delta: i < j} \One\{\pi(i) > \pi(j)\} } \right]\right|\\
        &\le \left(\frac{k}{n}\right)^{|V(A)|}(2pq)^{|A|} \One\{A \text{ even}\}.
    \end{align*}

    The proof for \eqref{ineq:monomial-bd-2} is similar, as we can separate out the part $(Y^{\circ 2})^B$ from $Y^A$. Each $Y_{i,j}^2$ is distributed as $\text{Bern}(p)$ independent of $S$ and $\pi$, which leads to an additional $p^{|B|}$ term in the upper bound.
\end{proof}

Next, the following describes an orthonormal basis of polynomials for the null distribution $\sQ$ (really a product basis formed from an orthonormal basis for the one-dimensional sparse Rademacher distribution).
\begin{definition}
    \label{def:h-poly}
    For $A, B \subseteq \binom{[n]}{2}$ disjoint subsets, we define the polynomial
    \[ h_{A, B}(Y) \colonequals \frac{1}{p^{|A|/2}} Y^A \frac{1}{(p(1-p))^{|B|/2}} (Y^{\circ 2}-pJ)^B. \]
\end{definition}
\begin{proposition}
    \label{prop:h-orthonormal}
    The $h_{A, B}$ over all pair of disjoint $A, B \subseteq \binom{[n]}{2}$ form an orthonormal basis of polynomials for $\sQ$.
\end{proposition}
\begin{proof}
    For the first claim of orthonormality, first note that every polynomial in $Y$ in the support of $\sQ$, i.e.\ any adjacency matrix of a directed graph, has entries satisfying $Y_{i, j}^3 = Y_{i, j}$, and thus every polynomial in $Y$ is equivalent to one where each entry occurs in each monomial with degree at most~2.
    The dimension of the space of polynomials in $Y$ is then at most
    \[ \sum_{A \subseteq \binom{[n]}{2}} 2^{|A|} = \sum_{k = 0}^{\binom{n}{2}} \binom{\binom{n}{2}}{k} 2^k = 3^{\binom{n}{2}}. \]
    And, this is precisely the number of $A, B \subseteq \binom{[n]}{2}$ disjoint, which may be computed as
    \[ \sum_{A \subseteq \binom{[n]}{2}} 2^{\binom{n}{2} - |A|} = 2^{\binom{n}{2}} \sum_{A \subseteq \binom{[n]}{2}}2^{-|A|} = 2^{\binom{n}{2}} \left(\frac{3}{2}\right)^{\binom{n}{2}} = 3^{\binom{n}{2}}. \]
    Thus, it suffices to show that the $h_{A, B}$ are a set of orthonormal polynomials for $\sQ$.

    To do that, we compute:
    \begin{align*}
    \Ex_{\mathcal{Q}}[h_{A_1, B_1}(Y) h_{A_2, B_2}(Y)]
    &= \prod_{\substack{\{i,j\}\in A_1 \cap A_2 :\\ i<j}} \Ex_{\mathcal{Q}}\left[\frac{1}{p} Y_{i,j}^2\right] \prod_{\substack{\{i,j\}\in B_1 \cap B_2 :\\ i<j}} \Ex_{\mathcal{Q}}\left[\frac{1}{p(1-p)} (Y_{i,j}^2-p)^2\right]\\
    &\quad \prod_{\substack{\{i,j\}\in (A_1 \cap B_2) \cup (A_2 \cap B_1) :\\ i<j}} \Ex_{\mathcal{Q}}\left[\frac{1}{p\sqrt{1-p}} Y_{i,j}\left(Y_{i,j}^2-p\right)\right]\\
    &\quad \prod_{\substack{\{i,j\} \in (A_1 \setminus (A_2 \cup B_2)) \cup (A_2 \setminus (A_1 \cup B_1)): \\ i<j} } \Ex_{\mathcal{Q}}\left[\frac{1}{\sqrt{p}} Y_{i,j}\right]\\
    &\quad \prod_{\substack{\{i,j\} \in (B_1 \setminus (A_2 \cup B_2)) \cup (B_2 \setminus (A_1 \cup B_1)): \\ i<j} } \Ex_{\mathcal{Q}}\left[\frac{1}{\sqrt{p(1-p)}} \left(Y_{i,j}^2-p\right)\right].
\end{align*}
Here, the first two products are always 1, while any of the last three products is 0 if it is non-empty (and 1 otherwise).
Thus, the entire expression is 0 if $A_1 \neq A_2$ or $B_1 \neq B_2$, and 1 otherwise, completing the proof.
\end{proof}

Having an explicit orthonormal basis of polynomials is especially useful for carrying out low-degree calculations. Below is an alternative expression (c.f.~\cite[Proposition 2.8]{kunisky2019notes}) for the low-degree advantage defined in Definition~\ref{def:adv}.
\begin{proposition} \label{prop:adv-alternative}
    \begin{align}
        \Adv_{\le D}(\sQ, \sP)^2 = \sum_{\substack{A, B \subseteq \binom{[n]}{2} \text{ disjoint}:\\ |A| + 2|B| \le D}} \left(\Ex_{Y \sim \sP}\left[h_{A,B}(Y)\right]\right)^2 .
    \end{align}
\end{proposition}

\begin{proof}
    For any polynomial $f \in \RR[Y]_{\le D}$, we may expand it using the basis of polynomials $h_{A,B}$ as \[f(Y) = \sum_{A,B} \hat{f}_{A,B} \cdot h_{A,B}(Y).\]
    Note $\deg(h_{A,B}) = |A| + 2|B|$. Since $\deg(f) \le D$, the coefficients satisfy $\hat{f}_{A,B} = 0$ for any pair of $A,B \in \binom{[n]}{2}$ such that $|A| + 2|B| > D$. Then, we may rewrite
    \begin{align*}
        \Adv_{\le D}(\sQ, \sP)^2 &= \inf_{\substack{f \in \RR[Y]_{\le D}\\ \Ex_{\sQ} f(Y)^2 \ne 0}} \frac{\left(\Ex_{\sP} f(Y)\right)^2}{\Ex_{\sQ} f(Y)^2}\\
        &= \inf_{\substack{\hat{f} = \{\hat{f}_{A,B}\} \ne 0\\ f= \sum_{A,B} \hat{f}_{A,B} \cdot h_{A,B}\\ \deg(f) \le D }} \frac{\left(\Ex_{\sP} f(Y)\right)^2}{\Ex_{\sQ} f(Y)^2}\\
        &= \inf_{\substack{\hat{f} = \{\hat{f}_{A,B}\} \ne 0\\ f= \sum_{A,B} \hat{f}_{A,B} \cdot h_{A,B}\\ \deg(f) \le D }} \frac{\left(\sum_{A,B} \hat{f}_{A,B} \cdot\Ex_{\sP} \left[h_{A,B}(Y)\right]\right)^2}{\sum_{A,B, A', B'} \hat{f}_{A,B} \hat{f}_{A', B'}\cdot \Ex_{\sQ} \left[h_{A,B}(Y)h_{A', B'}(Y)\right]}\\
        &= \inf_{\substack{\hat{f} = \{\hat{f}_{A,B}\} \ne 0\\ f= \sum_{A,B} \hat{f}_{A,B} \cdot h_{A,B}\\ \deg(f) \le D }} \frac{\left(\sum_{A,B} \hat{f}_{A,B} \cdot\Ex_{\sP} \left[h_{A,B}(Y)\right]\right)^2}{\sum_{A,B} \left(\hat{f}_{A,B}\right)^2 }
        \intertext{by orthonormality of $h_{A,B}$ as stated in Proposition~\ref{prop:h-orthonormal},}
        &= \sum_{\substack{A, B \subseteq \binom{[n]}{2} \text{ disjoint}:\\ |A| + 2|B| \le D}} \left(\Ex_{Y \sim \sP}\left[h_{A,B}(Y)\right]\right)^2,
    \end{align*}
    completing the proof.
\end{proof}

\subsection{Tools for Analysis of Ranking By Wins Algorithm} \label{sec:ranking-by-wins}

We also introduce some tools that will be useful in analyzing the Ranking By Wins algorithm (Definition~\ref{def:rank-by-wins}).
Its analysis will boil down to estimating the expected error or value achieved by the algorithm as well as controlling the fluctuations of this quantity.

To bound the fluctuation of solution output by the Ranking By Wins algorithm, we will use the following results on tail bounds for weakly dependent random variables.

\begin{definition}[Read-$k$ families \cite{gavinsky2015tail}]
    Let $X_1, \dots, X_m$ be independent random variables. Let $Y_1, \dots, Y_n$ be Boolean random variables such that $Y_j = f_j((X_i)_{i \in P_j})$ for some Boolean functions $f_j$ and index sets $P_j \subseteq [m]$. If the index sets satisfy $|\{j: i\in P_j\}|\le k$ for every $i\in [n]$, we say that $\{Y_j\}_{j=1}^n$ forms a \emph{read-$k$ family}.
\end{definition}

\begin{theorem}[Tail bounds for read-$k$ families \cite{gavinsky2015tail}]\label{thm:tail-read-k}
    Let $Y_1, \dots, Y_r$ be a read-$k$ family of Boolean random variables.
    Write $\mu \colonequals \EE \sum_{i = 1}^r Y_i$.
    Then, for any $t\ge 0$,
    \begin{align*}
        \mathbb{P}\left[\sum_{i=1}^r Y_i \ge \mu + t \right] &\le \exp\left(-\frac{2t^2}{rk}\right),\\
        \mathbb{P}\left[\sum_{i=1}^r Y_i \le \mu - t \right] &\le \exp\left(-\frac{2t^2}{rk}\right).
    \end{align*}
\end{theorem}

To estimate the expectation of the error or alignment objective value achieved by the Ranking By Wins algorithm, we will use the following version of the Berry-Esseen quantitative central limit theorem.

\begin{theorem}[Berry-Esseen theorem for non-identically distributed summands \cite{berry1941accuracy}]\label{thm:berry-esseen}
    Let $X_1, \dots, X_n$ be independent random variables with $\mathbb{E}[X_i] = 0, \mathbb{E}[X_i^2] = \sigma_i^2$, and $\mathbb{E}[|X_i|^3] = \rho_i < \infty$. Let
    \[S_n = \frac{\sum_{i=1}^n X_i}{\sqrt{\sum_{i=1}^n \sigma_i^2}}.\]
    Then, there exists an absolute constant $C>0$ independent of $n$ such that for any $x \in \mathbb{R}$,
    \begin{align*}
        \left|\mathbb{P}\left[S_n \le x\right] - \Phi(x)\right| \le C\cdot \frac{\max_{1\le i\le n} \frac{\rho_i}{\sigma_i^2}}{\sqrt{\sum_{i=1}^n \sigma_i^2}},
    \end{align*}
    where $\Phi: \mathbb{R} \to [0,1]$ is the cumulative distribution function (cdf) of the standard normal distribution.
\end{theorem}

After applying the Berry-Esseen theorem above, naturally we need to deal with expressions involving $\Phi$, the cdf of the standard normal distribution. We state a useful lemma for bounding certain sums involving the function $\Phi$.

\begin{lemma} \label{lem:concavity-Phi-expr}
    Let $a, b \ge 0$. As a function of $y$,
    \begin{align*}
        (1 - y) \cdot \Phi(-ay-b)
    \end{align*}
    is concave for $y\in [0,1]$.
\end{lemma}

\begin{proof}[Proof of Lemma \ref{lem:concavity-Phi-expr}]
    We compute the first and the second derivative of $(1 - y)\Phi(-ay-b)$.
    \begin{align*}
    \frac{d}{dy} (1 - y) \Phi\left(-ay-b\right)
    &= \frac{d}{dy} (1 - y) \int_{-\infty}^{-ay-b} \frac{1}{\sqrt{2\pi}} e^{-\frac{1}{2}z^2}dz\\
    &= - \int_{-\infty}^{-ay-b} \frac{1}{\sqrt{2\pi}} e^{-\frac{1}{2}z^2}dz + (1-y)\frac{1}{\sqrt{2\pi}} e^{-\frac{(ay+b)^2}{2}},\\
    \frac{d^2}{dy^2} (1 - y) \Phi\left(- ay-b\right)
    &= \frac{d}{dy}\left[- \int_{-\infty}^{-ay-b} \frac{1}{\sqrt{2\pi}} e^{-\frac{1}{2}z^2}dz + (1-y)\frac{1}{\sqrt{2\pi}} e^{-\frac{(ay+b)^2}{2}}\right]\\
    &= -\frac{1}{\sqrt{2\pi}} e^{-\frac{(ay+b)^2}{2}} - \frac{1}{\sqrt{2\pi}} e^{-\frac{(ay+b)^2}{2}} + (1-y) \frac{1}{\sqrt{2\pi}}\left(-a(ay+b)\right)e^{-\frac{(ay+b)^2}{2}}\\
    &= \frac{1}{\sqrt{2\pi}}e^{-\frac{(ay+b)^2}{2}}\left(a(y-1)(ay+b) - 2\right).
\end{align*}
We observe that the second derivative is negative for $y \in [0,1]$. Thus, $(1 - y)\Phi(-ay-b)$ is concave on $[0,1]$.
\end{proof}

\section{Proofs for Log-Density Setting}

\subsection{Computational Detection: Proof of Theorem~\ref{thm:log-density-comp-det}}

\subsubsection{Upper Bound}
\label{sec:log-density-comp-det-upper}

Consider the polynomial
\begin{equation}
f(Y) = \sum_{i=1}^n \sum_{\substack{j,k\in [n]\setminus \{i\}\\ j<k}} Y_{i,j}Y_{i,k}.
\label{eq:deg-2-test}
\end{equation}
We will show that thresholding this polynomial achieves strong detection, by achieving strong separation of $\sQ$ and $\sP$ (see Definition~\ref{def:separation}), provided that
\begin{equation}
q = \omega\left(\frac{n^{3/4}}{k^{3/2}p^{1/2}}\right).
\label{eq:log-density-comp-det-condition}
\end{equation}

First, let us give some intuition about why $f(Y)$ is a reasonable test statistic.
If each row and column of $Y$ has approximately the same number of non-zero entries, then we expect
    \begin{equation}
        f(Y) \approx K + \frac{1}{2}\sum_{i = 1}^n \left(\sum_{j \in [n] \setminus \{i\}} Y_{i, j}\right)^2
    \end{equation}
    for some constant $K$.
    Therefore, up to translation and rescaling, $f(T)$ looks sample variance of the numbers of ``wins'' of various vertices (the number of other vertices they are ranked above, as also appears in the Ranking By Wins algorithm in Definition~\ref{def:rank-by-wins}).
    In other words, $f(T)$ will be larger when the distribution of win counts is more spread out, which we expect to occur under the planted model with sufficiently strong signal.

We will use the basic properties of monomials in the directed adjacency matrix $Y$ given in Proposition~\ref{prop:adj-moments}.
Using these, we compute that
{\allowdisplaybreaks
\begin{align*}
    \Ex_{\mathcal{Q}}[f(Y)] &= 0\\
    \Ex_{\mathcal{P}}[f(Y)] &= \frac{n(n-1)(n-2)}{2} \Ex_{\mathcal{P}}[Y_{1,2}Y_{1,3}]\\
    &= (1 + o(1))\frac{2}{3} k^3p^2q^2\\
    \Var_{\mathcal{Q}}(f(Y)) &=\Ex_{\mathcal{Q}}[f(Y)^2] \\
    &= \sum_{\substack{(i_1, j_1, k_1), (i_2, j_2, k_2):\\ A_t = \{\{i_t,j_t\}, \{i_t, k_t\}\}, t = 1,2}} \Ex_{\mathcal{Q}}[Y^{A_1} Y^{A_2}]\\
    &= \sum_{\substack{(i_1, j_1, k_1), (i_2, j_2, k_2):\\ A_t = \{\{i_t,j_t\}, \{i_t, k_t\}\}, t = 1,2}} p^2 \One\{A_1 = A_2\}\\
    &= (1 + o(1))\frac{n^3}{2}p^2\\
    \Var_{\mathcal{P}}(f(Y)) &= \Ex_{\mathcal{P}}[f(Y)^2] - \Ex_{\mathcal{P}}[f(Y)]^2\\
    &= \sum_{\substack{(i_1, j_1, k_1), (i_2, j_2, k_2):\\ A_t = \{\{i_t,j_t\}, \{i_t, k_t\}\}, t = 1,2}} \left(\Ex_{\mathcal{P}}[Y^{A_1} Y^{A_2}] - \Ex_{\mathcal{P}}[Y^{A_1}]\Ex_{\mathcal{P}}[ Y^{A_2}]\right)\\
    &= \sum_{\substack{(i_1, j_1, k_1), (i_2, j_2, k_2):\\ A_t = \{\{i_t,j_t\}, \{i_t, k_t\}\}, t = 1,2\\
    V(A_1) \cap V(A_2) \ne \emptyset}} \left(\Ex_{\mathcal{P}}[Y^{A_1} Y^{A_2}] - \Ex_{\mathcal{P}}[Y^{A_1}]\Ex_{\mathcal{P}}[ Y^{A_2}]\right)
    \intertext{since if $A_1, A_2$ are vertex disjoint, by Proposition~\ref{prop:component-independence}, the corresponding term is $0$,}
    &\le \sum_{\substack{(i_1, j_1, k_1), (i_2, j_2, k_2):\\ A_t = \{\{i_t,j_t\}, \{i_t, k_t\}\}, t = 1,2\\
    |V(A_1) \cap V(A_2)| = 3}} \Ex_{\mathcal{P}}[Y^{A_1} Y^{A_2}]  + \sum_{\substack{(i_1, j_1, k_1), (i_2, j_2, k_2):\\ A_t = \{\{i_t,j_t\}, \{i_t, k_t\}\}, t = 1,2\\
    |V(A_1) \cap V(A_2)| = 2}} \Ex_{\mathcal{P}}[Y^{A_1} Y^{A_2}]\\
    &\qquad + \sum_{\substack{(i_1, j_1, k_1), (i_2, j_2, k_2):\\ A_t = \{\{i_t,j_t\}, \{i_t, k_t\}\}, t = 1,2\\
    |V(A_1) \cap V(A_2)| = 1}} \Ex_{\mathcal{P}}[Y^{A_1} Y^{A_2}]\\
    &\le O\Bigg[ n^3 \left(p^2 + \left(\frac{k}{n}\right)^3p^3q^2 \right) + n^4\left(\left(\frac{k}{n}\right)^3 p^3 q^2 + 0 + \left(\frac{k}{n}\right)^4 p^4q^4 + \left(\frac{k}{n}\right)^4p^4q^4 \right)\\
    &\qquad + n^5 \left(\left(\frac{k}{n}\right)^5 p^4 q^4 + \left(\frac{k}{n}\right)^5 p^4 q^4 + \left(\frac{k}{n}\right)^5 p^4 q^4\right)\Bigg] \\
    &= O(n^3p^2 + k^3 n p^3 q^2 + k^5p^4 q^4),
\end{align*}}
where the second-to-last line follows from Proposition~\ref{prop:fourier-expectation} and the ways two length-$2$ paths can intersect as shown in Figure~\ref{fig:intersecting-pat}.
\begin{figure}
    \begin{center}
    \begin{tabular}{c|cccc}
        \hline
        Vertex Intersection & \multicolumn{4}{c}{Shapes} \\
        \hline \\[-0.8em]
        3
        & \begin{tikzpicture}[scale=0.2] \draw (0,0) -- (2,2); \draw (2,2) -- (4,0); \draw (0,0.5) -- (2,2.5); \draw (2,2.5) -- (4,0.5);
        \end{tikzpicture}
        & \begin{tikzpicture}[scale=0.2] \draw (0,0) -- (2,2); \draw (0,0) -- (4,0); \draw (0,0.5) -- (2,2.5); \draw (2,2.5) -- (4,0.5);
        \end{tikzpicture} & & \\[0.5em]
        2 & \begin{tikzpicture}[scale=0.2] \draw (0,0) -- (2,2); \draw (2,2) -- (4,0); \draw (2.5,2) -- (4.5,0); \draw (2.5,2) -- (6, 2); \end{tikzpicture} & \begin{tikzpicture}[scale=0.2] \draw (0,0) -- (2,2); \draw (2,2) -- (4,0); \draw (2.5,2) -- (4.5,0); \draw (4.5,0) -- (6, 2); \end{tikzpicture} & \begin{tikzpicture}[scale=0.2] \draw (0,0) -- (2,2); \draw (2,2) -- (4,0); \draw (2.5,2) -- (6,2); \draw (4.5,0) -- (6, 2); \end{tikzpicture} & \begin{tikzpicture}[scale=0.2] \draw (0,0) -- (2,2); \draw (0,0) -- (4,0); \draw (2.5,2) -- (6,2); \draw (4.5,0) -- (6, 2); \end{tikzpicture} \\[0.5em]
        1 & \begin{tikzpicture}[scale=0.2] \draw (0,0) -- (2,1); \draw (0,2) -- (2,1); \draw (2.5, 1) -- (4.5, 0); \draw (2.5, 1) -- (4.5, 2); \end{tikzpicture} & \begin{tikzpicture}[scale=0.2] \draw (0,0) -- (2,1); \draw (0,2) -- (2,1); \draw (2.5, 1) -- (4.5, 0); \draw (4.5, 0) -- (4.5, 2); \end{tikzpicture} & \begin{tikzpicture}[scale=0.2] \draw (0,0) -- (2,1); \draw (0,2) -- (0,0); \draw (2.5, 1) -- (4.5, 0); \draw (4.5, 0) -- (4.5, 2); \end{tikzpicture} & \\[0.5em]
        \hline
        \end{tabular}
    \end{center}
    \caption{All possible patterns in which two paths of length $2$ can intersect non-trivially.}
    \label{fig:intersecting-pat}
\end{figure}

The condition that $f$ strongly separates $\sP$ and $\sQ$ then translates to
\begin{align*}
    \sqrt{n^3p^2 + k^3 n p^3 q^2 + k^5p^4 q^4} = o(k^3p^2q^2),
\end{align*}
which holds when $q = \omega(\frac{n^{3/4}}{k^{3/2}p^{1/2}})$.
Thus, under this condition, thresholding $f(Y)$ at a suitable value achieves strong detection between $\sP$ and $\sQ$.

\subsubsection{Lower Bound}

We now show a lower bound against low-degree polynomials for detection, showing that it is impossible for polynomials of degree $D = O((\log n)^{2 - \epsilon})$ to weakly separate $\sQ$ from $\sP$ once $\alpha > \frac{3}{2}\beta - \frac{3}{4} - \frac{1}{2}\gamma$.
By Proposition~\ref{prop:adv-bounds}, it suffices to show that the advantage $\Adv_{\leq D}(\sQ, \sP)$ is bounded for such $D = D(n)$.

Recall from Definition~\ref{def:h-poly} and Proposition~\ref{prop:h-orthonormal} the basis of orthonormal polynomials $h_{A, B}(Y)$ for $\sQ$, taken over disjoint pairs $A, B \subseteq \binom{[n]}{2}$.
The following will be a useful preliminary computation.
Call $A \subseteq \binom{[n]}{2}$ \emph{even} if, when interpreted as a graph on $[n]$, every connected component of $A$ has an even number of edges.
\begin{proposition}
    \label{prop:EP-hST-bound}
    For any pair of disjoint $A, B$ as above, we have
    \[ |\EE_{\sP} [h_{A, B}(Y)]| \leq \One\{A \text{ even}\}\cdot\One\{B = \emptyset\} \cdot \left(\frac{k}{n}\right)^{|V(A)|} \left(2q\sqrt{p} \right)^{|A|}. \]
\end{proposition}
\begin{proof}
    That the expectation is zero for $B \neq \emptyset$ follows since
    \[ \Ex_{\mathcal{P}}[h_{A,B}]
    = \Ex_{\mathcal{P}}[h_{A,\emptyset}] \cdot \prod_{\substack{\{i,j\} \in B: \\ i<j} } \Ex_{\mathcal{P}}\left[\frac{1}{\sqrt{p(1-p)}} \left(Y_{i,j}^2-p\right)\right]. \]
    When $B = \emptyset$, only $\Ex_{\mathcal{P}}[h_{A,\emptyset}]$ remains and the bound follows from \eqref{ineq:monomial-bd-1} in Proposition~\ref{prop:adj-moments}.
\end{proof}

We use the $h_{A, B}$ together with the bound of Proposition~\ref{prop:EP-hST-bound} to compute the low-degree advantage.
We abbreviate $h_A \colonequals h_{A, \emptyset}$.
Then:
{\allowdisplaybreaks
\begin{align*}
    \Adv_{\le D}(\sQ, \sP)^2
    &= \sup_{\substack{f \in \RR[Y]_{\leq D} \\ \Ex_{\sQ} f(Y)^2 \neq 0}} \frac{\left(\Ex_{\sP} f(Y)\right)^2}{\Ex_{\sQ} f(Y)^2}\\
    &= \sup_{\substack{f \in \RR[Y]_{\leq D}\\ f= \sum_{A,B} \hat{f}_{A,B} \cdot h_{A,B}\\
    \hat{f}\ne 0 }} \frac{\left(\Ex_{\sP} f(Y)\right)^2}{\sqrt{\Ex_{\sQ} f(Y)^2}}\\
    &= \sum_{\substack{A, B \subseteq \binom{[n]}{2} \text{ disjoint}: \\ \deg(h_{A,B}) \le D}} \left(\Ex_{\sP}[h_{A, B}]\right)^2
    \intertext{since $h_{A,B}$ form an orthonormal basis for $\sQ$ by Proposition~\ref{prop:h-orthonormal},}
    &= \sum_{\substack{A \subseteq \binom{[n]}{2}:  A \text{ even}}} (\Ex_{\mathcal{P}}[h_A])^2\\
    &\le \sum_{\substack{A \subseteq \binom{[n]}{2}: A \text{ even}}} \left(\left(\frac{k}{n}\right)^{|V(A)|} \left(2q\sqrt{p}\right)^{|A|}\right)^2
    \intertext{by Proposition~\ref{prop:EP-hST-bound},}
    &\le 1 + \sum_{d=2}^D \sum_{v=2}^{n} \sum_{\substack{A \subseteq \binom{[n]}{2}:\\ |A| = d, |V(A)| = v,\\ A \text{ even}}}\left(\left(\frac{k}{n}\right)^{v} \left(2q\sqrt{p} \right)^{d}\right)^2\\
    &\le 1 + \sum_{d=2}^D \sum_{v = \left\lceil \sqrt{2d} \right\rceil}^{\min\{\frac{3}{2}d, n\}} \sum_{\substack{A \subseteq \binom{[n]}{2}:\\ |A| = d, |V(A)| = v,\\ A \text{ even}}} \left(\left(\frac{k}{n}\right)^{v} \left(2q\sqrt{p}\right)^{d}\right)^2
    \intertext{where the summation over $v$ is truncated since for any $A$ that has even number of edges in each component, we have $|V(A)| \le |E(A)| + \text{cc}(A) \le \frac{3}{2}|E(A)|$, and moreover any $A$ with $v$ vertices has at most $\binom{v}{2}$ edges, so we may further bound}
    &\le 1 + \sum_{d=2}^D \sum_{v = \left\lceil \sqrt{2d} \right\rceil}^{\min\{\frac{3}{2}d, n\}} \binom{n}{v} \binom{\binom{v}{2}}{d} \left(\left(\frac{k}{n}\right)^{v} \left(2q\sqrt{p}\right)^{d}\right)^2\\
    &\le 1 + \sum_{d=2}^D \left(\frac{2e pq^2}{d}\right)^{d} \cdot \sum_{v=\left\lceil \sqrt{2d} \right\rceil}^{\min\{\frac{3}{2}d, n\}} \left(\frac{k^2 e}{n v}\right)^v \cdot v^{2d}.
\end{align*}

Note that $$\frac{d}{dv} \log\left(\left(\frac{k^2 e}{n v}\right)^v \cdot v^{2d}\right) = \log \frac{k^2}{n} - \log v + \frac{2d}{v}.$$
We consider three cases, $\beta > 1/2$, $\beta = 1/2$, and $\beta < 1/2$ (recall that $k = n^{\beta}$), depending on the behavior of the first term here.

When $\beta > \frac{1}{2}$, we observe that
$$\frac{d}{dv} \log\left(\left(\frac{k^2 e}{n v}\right)^v \cdot v^{2d}\right) = \log \frac{k^2}{n} - \log v + \frac{2d}{v} = (2\beta - 1) \log n - \log v + \frac{2d}{v} \ge 0$$
for $v \le \frac{3}{2}d \le \frac{3}{2}D$. Thus, the maximum of the inner sum appears at the last term, and there are at most $\frac{3}{2}d$ terms in the sum. We may then bound the advantage by
\begin{align*}
    \text{Adv}_{\le D}(\sQ, \sP)^2 &\le 1 + \sum_{d=2}^D \left(\frac{2e pq^2}{d}\right)^{d} \cdot \sum_{v=\left\lceil \sqrt{2d} \right\rceil}^{\min\{\frac{3}{2}d, n\}} \left(\frac{k^2 e}{n v}\right)^v \cdot v^{2d}\\
    &\le 1 + \sum_{d=2}^D \left(\frac{2e pq^2}{d}\right)^{d} \cdot \frac{3}{2}d \cdot \left(\frac{k^2 e}{n \left({\frac{3}{2}d}\right)}\right)^{\frac{3}{2}d} \cdot \left({\frac{3}{2}d}\right)^{2d}\\
    &= 1 + \frac{3}{2}\sum_{d=2}^D d \left(\frac{6^{\frac{1}{2}}e}{d^{\frac{1}{2}}} \cdot \frac{pq^2 k^3}{n^{\frac{3}{2}}}\right)^{d},
\end{align*}
which is $1 + o(1)$ if $pq^2 k^3/n^{3/2} = o(1)$, which translates to the condition $\alpha > \frac{3}{2}\beta - \frac{3}{4} - \frac{1}{2}\gamma$ in terms of the exponents $\alpha, \beta, \gamma$.
Note that in this case the only condition on $D$ is that $\frac{3}{2}D \leq n$, or $D \leq \frac{2}{3}n$, which is satisfied under our assumptions.

Now let us consider the case when $\beta = \frac{1}{2}$.
We may then bound the low-degree advantage by
\begin{align*}
    \text{Adv}_{\le D}(\sQ, \sP)^2 &\le 1 + \sum_{d=2}^D \left(\frac{2e pq^2}{d}\right)^{d} \cdot \sum_{v=\left\lceil \sqrt{2d} \right\rceil}^{\min\{\frac{3}{2}d, n\}} \left(\frac{k^2 e}{n v}\right)^v \cdot v^{2d}\\
    &\le 1 + \sum_{d=2}^D \left(\frac{2e pq^2}{d}\right)^{d} \cdot \sum_{v=\left\lceil \sqrt{2d} \right\rceil}^{\min\{\frac{3}{2}d, n\}}  (ev)^{2d}\\
    &\le 1 + \sum_{d=2}^D \left(\frac{2e pq^2}{d}\right)^{d} \cdot \left(\frac{3}{2}d\right) \cdot  \left(e \frac{3}{2}d\right)^{2d}\\
    &\le 1 + \frac{3}{2}\sum_{d=2}^D d \left(\frac{9e^3}{2} \cdot pq^2d\right)^{d},
\end{align*}
which is $1 + o(1)$ for $D = O(\mathrm{polylog}(n))$ if $\alpha > 0$ or $\gamma > 0$, which is implied by $\alpha > \frac{3}{2}\beta - \frac{3}{4} - \frac{1}{2}\gamma$ when $\beta = \frac{1}{2}$.

If $\beta < \frac{1}{2}$, we observe that
$$\frac{d}{dv} \log\left(\left(\frac{k^2 e}{n v}\right)^v \cdot v^{2d}\right) = \log \frac{k^2}{n} - \log v + \frac{2d}{v} \le -(1-2\beta)\log n - \log v + \sqrt{2d} \le 0$$
for $v \ge \sqrt{2d}$ and $d \le D = O((\log n)^{2 - \epsilon})$. Thus, the maximum of the inner sum appears at the first term, and there are at most $\frac{3}{2}d$ terms in the sum. We may then upper bound the low degree advantage by
\begin{align*}
    \text{Adv}_{\le D}(\sQ, \sP)^2 &\le 1 + \sum_{d=2}^D \left(\frac{2e pq^2}{d}\right)^{d} \cdot \sum_{v=\left\lceil \sqrt{2d} \right\rceil}^{\min\{\frac{3}{2}d, n\}} \left(\frac{k^2 e}{n v}\right)^v \cdot v^{2d}\\
    &\le 1 + \sum_{d=2}^D \left(\frac{2e pq^2}{d}\right)^{d} \cdot \frac{3}{2}d \cdot \left(\frac{k^2 e}{n \sqrt{2d}}\right)^{\sqrt{2d}} \cdot {\left\lceil \sqrt{2d} \right\rceil}^{2d}
    \intertext{where we use that $\frac{k^2 e}{n \lceil \sqrt{2d}\rceil} \le \frac{k^2 e}{n \sqrt{2d}} \le 1$ for $d \le D = O((\log n)^{2 - \epsilon})$ when $\beta < \frac{1}{2}$ and $k = n^{\beta}$,}
    &\le 1 + \frac{3}{2}\sum_{d=2}^D d\left(\frac{2e pq^2}{d}\right)^{d} \cdot \left(\frac{k^2 e \left\lceil \sqrt{2d} \right\rceil^{\sqrt{2d}}}{n \sqrt{2d}}\right)^{\sqrt{2d}}.
\end{align*}
Notice that $$\left(\frac{k^2 e \left\lceil \sqrt{2d} \right\rceil^{\sqrt{2d}}}{n \sqrt{2d}}\right)^{\sqrt{2d}} = o(1)$$ for $d \le D = O((\log n)^{2 - \varepsilon})$ when $\beta < \frac{1}{2}$. Thus, the low degree advantage is $1 + o(1)$ when $\beta < \frac{1}{2}$.

We thus conclude that, in all cases, if $\alpha > \frac{3}{2}\beta - \frac{3}{4} - \frac{1}{2}\gamma$, then the low degree advantage is $1 + o(1)$ for $D = O((\log n)^{2 - \epsilon})$.
The result then follows by Proposition~\ref{prop:adv-bounds}, as outlined at the beginning of this proof.

\subsection{Statistical Detection: Proof of Theorem~\ref{thm:log-density-stat-det}}

\subsubsection{Upper Bound}

We will consider two algorithms for detection. One algorithm is the degree-$2$ polynomial considered in Section~\ref{sec:log-density-comp-det-upper} to show a computational upper bound for detection, which strongly separates (and thus achieves strong detection for) $\mathcal{P}$ and $\mathcal{Q}$ when $q \gg \frac{n^{3/4}}{k^{3/2}p^{1/2}}$, i.e., when $\alpha < \frac{3}{2}\beta - \frac{1}{2}\gamma - \frac{3}{4}$.

The other algorithm we consider is to threshold the following statistic, inspired by a maximum likelihood calculation:
\begin{align}
    f(Y) \colonequals \max_{\substack{\hat{S}, \pi_{\hat{S}}\\ \hat{S} \subseteq V, |\hat{S}| \le 2k}} \sum_{\substack{i,j \in \hat{S}:\\ i < j}} Y_{i,j}\cdot \pi_{\hat{S}}(i,j),
\end{align}
where the maximum is over permutations $\pi_{\hat{S}}$ of the set $\hat{S}$ (whose size varies in the maximization).
We will show that thresholding this statistic achieves strong detection if $q \gg \frac{1}{\sqrt{kp}} \sqrt{\log n}$, and therefore whenever $\beta > 2\alpha + \gamma$.

We will use $c, C$ for constants in this proof which may vary from line to line.
The results stated will all hold with $c$ chosen sufficiently small and $C$ chosen sufficiently large, independent of all other parameters.

    First, let us consider $f(Y)$ for $Y \sim \mathcal{P}$. Then, w.h.p.~we have
    \begin{align*}
        f(Y) \ge \sum_{\substack{i,j \in S:\\ i < j}} Y_{i,j}\cdot \pi(i,j),
    \end{align*}
    since $S$ has size at most $2k$ with probability $1 - \exp(-k/3)$ by Chernoff bound.

    From the definition of $\mathcal{P}$, we have
    \begin{align}
        \sum_{\substack{i,j \in S:\\ i < j}} Y_{i,j}\cdot \pi(i,j) &\stackrel{\text{(d)}}{=} \sum_{\substack{i,j\in S:\\ i < j}} Z_{i,j} \cdot P_{i,j},
    \end{align}
    where $Z_{i,j} \stackrel{\iid}{\sim} \text{Rad}\left(1/2 + q\right)$ and $P_{i,j} \stackrel{\iid}{\sim} \text{Bern}(p)$.

    Now let us also condition on the set $S$. We may compute
    \[ \mathbb{E}\left[\sum_{\substack{i,j \in S:\\ i < j}} Y_{i,j}\cdot \pi(i,j)\right] = p \cdot \binom{|S|}{2} \cdot (2q)
        = p q |S|(|S|-1). \]
    We may moreover invoke Bernstein inequality and get that for any $\lambda > 0$,
    \[ \mathbb{P}\left[\sum_{\substack{i,j \in S:\\ i < j}} Y_{i,j}\cdot \pi(i,j) \le pq|S|(|S|-1) - \lambda\right] \le \exp\left(- \frac{\lambda^2}{2\binom{|S|}{2}p + \frac{2}{3} \cdot 2\lambda}\right), \]
    since for each $i,j\in S$, $\Var(Y_{i,j}\cdot \pi(i,j)) \le p$ and $|Y_{i,j} - \mathbb{E}[Y_{i,j}]| \le 2$.

    Fix a small constant $c > 0$. Setting $\lambda = c\cdot pq k^2$, we get
    \begin{align*}
        \mathbb{P}\left[\sum_{\substack{i,j \in S:\\ i < j}} Y_{i,j}\cdot \pi(i,j) \le pq|S|(|S|-1) - c\cdot pq k^2 \right] &\le \exp\left(- \frac{c^2 \cdot p^2q^2 k^4}{2\binom{|S|}{2}p + \frac{4}{3}c \cdot pq k^2}\right) \\ &\le \exp\left(- C \cdot pq^2 k^2\right),
    \end{align*}
    w.h.p.~since $|S| = (1 + o(1))k$, which is $o(1)$ since $pq^2 k^2 \gg k$ when $q \gg \frac{1}{\sqrt{kp}} \sqrt{\log n}$. Thus, w.h.p., we have
    $$\sum_{\substack{i,j \in S:\\ i < j}} Y_{i,j}\cdot \pi(i,j) > pq|S|(|S| - 1) - c\cdot pq k^2 \ge (1-c - o(1)) \cdot pq k^2.$$

    On the other hand, let us consider $f(Y)$ for $Y \sim \mathcal{Q}$. For every pair of $\hat{S}, \pi_{\hat{S}}$ with $|\hat{S}| \le 2k$, from the definition of $\mathcal{Q}$, we have
    \begin{align}
        \sum_{\substack{i,j\in \hat{S}:\\ i < j }} Y_{i,j} \cdot \hat{\pi}(i,j) \stackrel{\text{(d)}}{=} \sum_{\substack{i,j\in \hat{S}:\\ i < j }} Z_{i,j}\cdot P_{i,j},
    \end{align}
    where $Z_{i,j} \stackrel{\iid}{\sim} \text{Rad}\left(1/2\right)$ and $P_{i,j} \stackrel{\iid}{\sim} \text{Bern}(p)$. In particular, all terms in the latter sum above have zero mean.

    Again, by Bernstein inequality, for any $\lambda > 0$,
    \[ \mathbb{P}\left[\sum_{\substack{i,j\in \hat{S}:\\ i < j }} Y_{i,j} \cdot \hat{\pi}(i,j) \ge \lambda\right]
        \le \exp\left(- \frac{\lambda^2}{2\binom{|S|}{2}p + \frac{2}{3} \cdot 2\lambda}\right). \]

    Fix a small constant $c > 0$. Setting $\lambda = c\cdot pq k^2$, we get
    \[ \mathbb{P}\left[\sum_{\substack{i,j \in S:\\ i < j}} Y_{i,j}\cdot \hat{\pi}(i,j) \ge c\cdot pq k^2 \right] \le \exp\left(- \frac{c^2 \cdot p^2q^2 k^4}{2\binom{|\hat{S}|}{2}p + \frac{4}{3}c \cdot pq k^2}\right) \le \exp\left(- C \cdot pq^2 k^2\right), \]
    since $|\hat{S}| \le 2k$. Note that when $q \gg \frac{1}{\sqrt{kp}} \sqrt{\log n}$, the above is $ \exp\left(-\omega\left(k \log n\right)\right) \le n^{-\omega(k)}$. Since the total number number of pairs of $\hat{S}, \pi_{\hat{S}}$ with $|\hat{S}| \le 2k$ is at most $\binom{n}{\le 2k} \cdot (2k)! \le n^{O(k)}$, we may apply a union bound over all possible $\hat{S}, \pi_{\hat{S}}$ to conclude that w.h.p.,
    \[ f(Y) = \max_{\substack{\hat{S}, \pi_{\hat{S}}  \hat{S} \subseteq V, |\hat{S}|\le 2k}} \sum_{\substack{i,j \in \hat{S}:\\ i < j}} Y_{i,j}\cdot \pi_{\hat{S}}(i,j) < c\cdot pq k^2. \]

    Thus, with the choice of a constant $c < \frac{1}{2}$, the calculations above show that thresholding $f(Y)$ at $\frac{1}{2}\cdot pq k^2$ successfully distinguishes $\mathcal{P}$ from $\mathcal{Q}$ w.h.p.\ when $q \gg \frac{1}{\sqrt{kp}}\sqrt{\log n}$.

Combining the two algorithms considered, we find that that strong detection is statistically possible when
\[ q \gg \min \left\{\frac{1}{\sqrt{kp}}\sqrt{\log n}, \frac{n^{3/4}}{k^{3/2}\sqrt{p}}\right\}. \]
In the log-density setting, this corresponds to $\beta > \min\left\{2\alpha + \gamma, \frac{2}{3}\alpha + \frac{1}{3}\gamma + \frac{1}{2}\right\}$, completing the proof.

\subsubsection{Lower Bound}
\label{sec:log-density-stat-det-lower}

We will prove our information-theoretic lower bounds for detection by bounding the total variation distance between $\mathcal{P}$ and $\mathcal{Q}$.
We will have to be careful about a technical detail involving conditioning this computation on the size of $S$; our approach to handling this is similar to that of \cite{hajek2015computational}.

Writing $\sP_{k}$ for the law $\sP$ conditional on the planted ranked subgraph having size $|S| = k$, we have
\begin{align*}
    d_{\text{TV}}(\mathcal{P}, \mathcal{Q} )
    &= d_{\text{TV}}(\mathbb{E}_{|S|}[\mathcal{P}_{  |S|}], \mathcal{Q} )\\
    &\le \mathbb{E}_{|S|} [d_{\text{TV}}(\mathcal{P}_{  |S|}, \mathcal{Q} )]
    \intertext{by Jensen's inequality and convexity of the TV distance, and further}
    &\le \exp(-k/3) + \sum_{1\le k' \le 2k} d_{\text{TV}}(\mathcal{P}_{k'}, \mathcal{Q} ) \cdot \mathcal{P}[|S| = k'],  \numberthis \label{ineq:TV-bound}
\end{align*}
where the last inequality follows since the total variation distance is at most $1$, and $|S| > 2k$ with probability at most $\exp(-k/3)$ by Chernoff bound.

Next, we further upper bound $d_{\text{TV}}(\mathcal{P}_{k'}, \mathcal{Q} )$ using the corresponding $\chi^2$-divergence.
Recall that, in general, for distributions $\sQ, \sP$ over a discrete domain, we have
\[ 1 + \chi^2(\sP \dbar \sQ) = \Ex_{Y \sim \sQ} \left(\frac{\sP[Y]}{\sQ[Y]}\right)^2 = \sum_Y \frac{(\sP[Y])^2}{\sQ[Y]}, \]
with the sum going over the support of $\sQ$.
Thus in our case above, further writing $\sP_{S, \pi_S}$ for the planted distribution conditioned on the choices of $S$ and $\pi_S$ the planted community and ranking, we have:
\begin{align*}
    &1 + \chi^2(\mathcal{P}_{k'} \dbar \mathcal{Q}) \\
    &= \sum_Y \frac{(\sP_{k'}[Y])^2}{\sQ[Y]} \\
    &= \Ex_{S, S' \sim \Unif(\binom{[n]}{k'})} \Ex_{\pi_S, \pi_{S'}} \sum_Y \frac{\sP_{S, \pi_S}[Y] \sP_{S', \pi_{S'}}[Y]}{\sQ[Y]}
    \intertext{For the sake of brevity, let us write $\pi = \pi_S$ and $\pi' = \pi_{S'}$, always pairing the set and its permutation in this way in the notation. Continuing, we may decompose $Y$ into the subset $A \subseteq \binom{[n]}{2}$ of edges that exist at all in $Y$, and into the assignment $T \in \{\pm 1\}^A$ of directions to the edges that do exist:}
    &= \Ex_{S, S'} \Ex_{\pi, \pi'} \sum_{A \subseteq \binom{[n]}{2}} \sum_{T \in \{\pm 1\}^{A}} \frac{1}{(1-p)^{\binom{n}{2} - |A|} p^{|A|}\left(\frac{1}{2}\right)^{|A| }}\\
    &\quad\quad (1-p)^{\binom{n}{2} - |A|} p^{|A|}\left(\frac{1}{2}\right)^{|A - \binom{S}{2}| }\prod_{\{i,j\}\in A \cap \binom{S}{2}} \left(\frac{1}{2} + q\right)^{\One\{\pi(i,j) = T_{i,j}\} } \left(\frac{1}{2} - q\right)^{\One\{\pi(i,j) = -T_{i,j}\} }\\
    &\quad\quad (1-p)^{\binom{n}{2} - |A|} p^{|A|}\left(\frac{1}{2}\right)^{|A - \binom{S'}{2}| }\prod_{\{i,j\}\in A \cap \binom{S'}{2}} \left(\frac{1}{2} + q\right)^{\One\{\pi'(i,j) = T_{i,j}\} } \left(\frac{1}{2} - q\right)^{\One\{\pi'(i,j) = -T_{i,j}\} }\\
    &= \Ex_{S, S'} \Ex_{\pi, \pi'} \Ex_A \prod_{\{i,j\} \in A\cap \binom{S}{2} \cap \binom{S'}{2} } \Bigg[ 2\sum_{t \in \{-1, 1\} }  \left(\frac{1}{2} + q\right)^{\One\{\pi(i,j) = t\} } \left(\frac{1}{2} - q\right)^{\One\{\pi(i,j) = -t\} }\\
    &\quad \hspace{5cm} \left(\frac{1}{2} + q\right)^{\One\{\pi'(i,j) = t\} } \left(\frac{1}{2} - q\right)^{ \One\{\pi'(i,j) = -t\} }\Bigg]\\
    &\quad\quad \prod_{\{i,j\} \in A \cap (\binom{S}{2} - \binom{S'}{2}) } \left[ 2\sum_{t \in \{-1, 1\} }  \left(\frac{1}{2}\right)\left(\frac{1}{2} + q\right)^{\One\{\pi(i,j) = t\} } \left(\frac{1}{2} - q\right)^{\One\{\pi(i,j) = -t\} } \right]\\
    &\quad\quad \prod_{\{i,j\} \in A \cap (\binom{S'}{2} - \binom{S}{2}) } \left[ 2\sum_{t \in \{-1, 1\} }  \left(\frac{1}{2}\right)\left(\frac{1}{2} + q\right)^{\One\{\pi'(i,j) = t\} } \left(\frac{1}{2} - q\right)^{\One\{\pi'(i,j) = -t\} } \right]\\
    &\quad\quad \prod_{\{i,j\} \in A - \binom{S}{2} - \binom{S'}{2} } \left[2\sum_{t \in \{-1, 1\} } \left(\frac{1}{2}\right)^2\right]
    \intertext{where $\mathbb{E}_A$ is over the distribution of random $A \subseteq \binom{S}{2}$ that includes each pair with probability $p$. We may simplify to:}
    &= \Ex_{S, S'} \Ex_{\pi, \pi'} \Ex_A \prod_{\{i,j\} \in A \cap \binom{S}{2} \cap \binom{S'}{2} } \left[(1 + 4q^2)^{\One\{\pi(i,j) = \pi'(i,j) \} } (1 - 4q^2)^{\One\{\pi(i,j) \ne \pi'(i,j) \} }\right] \\
    &= \Ex_{S, S'} \Ex_{\pi, \pi'} \prod_{\{i,j\} \in \binom{S}{2} \cap \binom{S'}{2} } \left(1+p\left[(1 + 4q^2)^{\One\{\pi(i,j) = \pi'(i,j) \} } (1 - 4q^2)^{\One\{\pi_S(i,j) \ne \pi'(i,j) \} } - 1\right]\right)
    \intertext{and, writing $\sigma, \sigma'$ for the permutations $\pi, \pi'$ each restricted to $S \cap S'$, we have}
    &= \Ex_{S, S'} \Ex_{\pi, \pi'} (1+4pq^2)^{\binom{|S \cap S'|}{2} - d_{\KT}(\sigma, \sigma') } (1 - 4pq^2)^{d_{\KT}(\sigma, \sigma')}
    \intertext{where $d_{\KT}$ is the Kendall  tau distance (Definition~\ref{def:kt}),}
    &= \Ex_{H = |S \cap S'|} \Ex_{\sigma, \sigma' \sim \text{Unif}(\Sym([H]))} (1+4pq^2)^{\binom{|H|}{2} - d_{\KT}(\sigma, \sigma') } (1 - 4pq^2)^{d_{\KT}(\sigma, \sigma')}
    \intertext{and noticing that $d_{KT}(\sigma, \sigma') = d_{\KT}(\sigma \circ \sigma'^{-1}, \mathrm{id})$, that the law of $\sigma \circ \sigma'^{-1}$ is again $\mathrm{Unif}(\Sym([H]))$, and that $d_{\KT}(\sigma, \mathrm{id}) = \inv(\sigma)$ is the number of pairs $\{i, j\} \in \binom{H}{2}$ that $\sigma$ inverts, we have}
    &= \Ex_{H = |S \cap S'|} \Ex_{\pi \sim \text{Unif}(\Sym([H]))} (1+4pq^2)^{\binom{|H|}{2} - \inv(\pi) } (1 - 4pq^2)^{\inv(\pi)} \\
    &\le \Ex_{H = |S \cap S'|} \Ex_{\pi \sim \text{Unif}(\Sym([H]))} (1 + 4pq^2)^{\binom{|H|}{2} - 2\inv(\pi) }. \numberthis \label{ineq:chi-square}
\end{align*}

The following result bounds the inner expectation, which is in essence a moment generating function of the number of inversions of a random permutation on a set of a given size.
\begin{proposition}[Moment generating function of inversions]
    \label{prop:mgf-inv}
    For any $h \geq 1$ and $x \geq 0$, we have
    \[ \Ex_{\pi \sim \Unif(\Sym([h]))} (1 + x)^{\binom{h}{2} - 2\inv(\pi)} \leq \exp\left(\frac{1}{2}x^2h^3\right)\left(1 + 2\sqrt{\frac{\pi}{2}x^2 h^3}\right). \]
\end{proposition}
\noindent
We give the proof in Appendix~\ref{app:pf:prop:mgf-inv}.

Plugging the bound above to \eqref{ineq:chi-square}, we get, using Cauchy-Schwarz to separate the two factors appearing,
\begin{align*}
    &\hspace{-1cm}\Ex_{H = |S \cap S'|} \left[\Ex_{\pi \sim \text{Unif}(S_H)} (1 + 4pq^2)^{\binom{H}{2} - 2\inv(\pi) }\right] \\
    &\leq \sqrt{\Ex_{H = |S \cap S'|} \left[\Ex_{\pi \sim \text{Unif}(S_H)} (1 + 4pq^2)^{\binom{H}{2} - 2\inv(\pi) }\right]^2} \\
    &\le \sqrt{\Ex_H \left[\exp(8p^2q^4 H^3) \left(1 + 2\sqrt{8\pi p^2q^4 H^3} \right)\right]^2} \\
    &\le \sqrt{\Ex_H \exp(16p^2q^4H^3)} \cdot \sqrt{\Ex_H\left(1 + 2\sqrt{8\pi p^2q^4 H^3} \right)^2} \\
    &\le \sqrt{\Ex_M \exp(16p^2q^4M^3)} \cdot\sqrt{\Ex_M\left(1 + 2\sqrt{8\pi p^2q^4 M^3} \right)^2},
\end{align*}
where $H = |S \cap S'| \stackrel{\text{(d)}}{=} \text{Hypergeometric}(n, k', k')$, and $M \sim \text{Binom}(k', \frac{k'}{n-k'})$ which stochastically dominates $H$.

Assume $k \le \frac{n}{4}$. We have $\mathbb{E}M^3 = O\left(\frac{k'^2}{n} + \frac{k'^6}{n^3}\right)$, and thus
\begin{align*}
    &\hspace{-1cm}\mathbb{E}\left[\left(1 + 2\sqrt{8\pi p^2q^4 M^3} \right)^2\right]\\
    &= 1 + 4\sqrt{8\pi p^2 q^4}\EE M^{3/2} + 32\pi p^2q^4 \mathbb{E}M^3 \\
    &\le 1 + 4\sqrt{8\pi p^2q^4}\sqrt{\mathbb{E} M^3} + 32\pi p^2q^4 \EE M^3 \\
    &\le 1 + C\cdot \left( \sqrt{p^2q^4 \max\left\{\frac{k'^2}{n}, \frac{k'^6}{n^3}\right\}} + p^2q^4\max\left\{\frac{k'^2}{n}, \frac{k'^6}{n^3}\right\}\right), \label{bd:term2}
\end{align*}
for some absolute constant $C > 0$.
This is $1 + o(1)$ when $p^2q^4 \max\left\{\frac{k'^2}{n}, \frac{k'^6}{n^3}\right\} = o(1)$, i.e., when $\beta < \frac{2}{3}\alpha + \frac{1}{3}\gamma + \frac{1}{2}$.

The other term requires a more involved calculation, which we summarize in the following.
As with our other omitted calculation, this is just a moment generating function calculation, in this case of a binomial random variable cubed under certain assumptions.
\begin{proposition}
    \label{prop:binom3}
    Suppose $0 \leq k = k(n) \leq n / 4$ and $M \sim \mathrm{Binom}(k, \frac{k}{n - k})$.
    Suppose also $x = x(n)$ satisfies $xk = o(1)$ and $x^2 k^6 / n^3 = o(1)$.
    Then,
    \[ \EE \exp(x^2M^3) = 1 + o(1). \]
\end{proposition}
\noindent
We give the proof in Appendix~\ref{app:pf:prop:binom3}.
In our case, $k = k'$ and $x = 4pq^2$, so the assumptions correspond to the region of log-density parameters $\beta < \min\left\{2\alpha + \gamma, \frac{2}{3}\alpha + \frac{1}{3}\gamma + \frac{1}{2}\right\}$.

Collecting our results, we find that, if $\beta < \min\left\{2\alpha + \gamma, \frac{2}{3}\alpha + \frac{1}{3}\gamma + \frac{1}{2}\right\}$, then $\chi^2(\sP_{k'} \dbar \sQ) = o(1)$ for all $1 \leq k' \leq 2k$.
By Pinsker's inequality, we then have
\[d_{\text{TV}}\left(\mathcal{P}_{k'},  \mathcal{Q}\right) \le \sqrt{\frac{\log\left(\chi^2(\mathcal{P}_{k'} \dbar \mathcal{Q}) + 1\right)}{2}} = o(1).\]
Plugging this bound into \eqref{ineq:TV-bound}, we conclude that $d_{\text{TV}}(\mathcal{P}, \mathcal{Q}) = o(1)$ and thus that weak detection is impossible, provided that $\beta < \min\left\{2\alpha + \gamma, \frac{2}{3}\alpha + \frac{1}{3}\gamma + \frac{1}{2}\right\}$.

\subsection{Computational Recovery: Proof of Theorem~\ref{thm:log-density-comp-rec}}
\label{sec:log-density-comp-rec}

\subsubsection{Upper Bound}

We will describe a simple spectral algorithm that achieves strong recovery when $q \gg \frac{\sqrt{n}}{k\sqrt{p}}$, as long as $p$ is not too small, namely as long as $p = \Omega(\log n / n)$.
This latter condition is automatically satisfied in the log-density setting, so this will show that strong recovery can be achieved in polynomial time provided that $\beta > \alpha + \frac{1}{2}\gamma + \frac{1}{2}$.

The analysis is based on the following well-known eigenvector perturbation bound.
\begin{theorem}[Davis-Kahan \cite{DK-1970-EigenvectorsPerturbation}] \label{thm:Davis-Kahan}
    Let $A, \Tilde{A} \in \mathbb{C}^{n \times n}$ be Hermitian. Let $v, \Tilde{v} \in \mathbb{C}^n$ with unit norms be eigenvectors associated to the largest eigenvalues of $A$ and $\Tilde{A}$ respectively. Define $\delta \colonequals \lambda_1(A) - \lambda_2(A)$ to be the spectral gap between the largest and the second largest eigenvalues of $A$. Then, if $\delta > 0$,
    \begin{align*}
        \|vv^\dagger - \Tilde{v}\Tilde{v}^\dagger\|_{F} \le \sqrt{2}\cdot \frac{\|A - \Tilde{A}\|}{\delta}.
    \end{align*}
\end{theorem}

Recall that $Y \in \{0, \pm 1\}^{n \times n}$ is the skew-symmetric adjacency matrix of a directed graph (in this proof always drawn from $\sP$).
Since the Davis-Kahan bound applies only to Hermitian matrices, we will work with $\eye Y$.
Since $Y = -Y^{\top}$, $\eye Y$ is Hermitian.

As before, for a fixed pair of $(S, \sigma_S)$, we define $\sP_{S, \sigma_S}$ to be the planted model $\sP$ conditioned on the planted community being $S$ and the planted permutation on $S$ being $\sigma_S$.
In this proof, we always use $\sigma$ for permutations, to avoid confusion with the numerical constant $\pi$ that will also appear.
Note that $\sP = \mathbb{E}_{S, \sigma_S}[\sP_{S, \sigma_S}]$ where the expectation is w.r.t.~the uniform distributions according to which $S$ and $\sigma_S$ are drawn.

To apply the Davis-Kahan bound, our strategy is to decompose
\[ \eye Y = \eye \EE Y + \eye(Y - \EE Y). \]
To be more precise, we will work with the conditional planted model $\sP_{S, \sigma_S}$, and show that for $Y \sim \sP_{S, \sigma_S}$, the spectral algorithm approximately recovers $S$ and $\sigma_S$ w.h.p., as long as $|S| \approx k$. Note that under $\sP_{S, \sigma_S}$, we can explicitly write down $\mathbb{E}[Y]$ as follows.
\begin{proposition}
    \label{prop:EY}
    Let $Y \sim \sP_{S, \pi_S}$. Then
    \begin{align*}
        \mathbb{E}[Y]_{i,j} = \begin{cases}
        2pq\cdot \pi(i,j) & \quad \text{ if } i,j \in S, i\ne j,\\
        0 & \quad \text{ otherwise}.
    \end{cases}
    \end{align*}
\end{proposition}
In particular, as we will see, the eigenvalues and eigenvectors of $\eye \mathbb{E}[Y]$ have nice explicit descriptions and the top eigenvector will be useful for recovering the ranked community $S$ and the planted permutation $\pi_S$.

Define $A_{\ell} \in \mathbb{C}^{\ell \times \ell}$ as
    \begin{align*}
        (A_{\ell})_{i,j} = \begin{cases}
            \eye & \quad \text{ if } i < j,\\
            -\eye & \quad \text{ if } i > j,\\
            0 & \quad \text{ if } i = j.
        \end{cases}
    \end{align*}

Proposition~\ref{prop:EY} implies that that $\eye \mathbb{E}[Y]$ is supported on the rows and columns indexed by $S$, and this principal submatrix $\eye \mathbb{E}[Y]_{S,S}$ is essentially a rescaled version of $A_{|S|}$, up to conjugation by a permutation matrix that permutes the rows and columns of $A_{|S|}$ according to $\sigma_S$. To be clear, define $P_{\sigma_S} \in \mathbb{C}^{S \times |S|}$ such that
    \begin{align*}
        (P_{\sigma_S})_{i,j} = \begin{cases}
            1 & \quad \text{ if } \sigma_S(i)=j,\\
            0 & \quad \text{ otherwise},
        \end{cases}
    \end{align*}
    and one can verify that for $Y \sim \sP_{S, \sigma_S}$, we have $\eye \mathbb{E}[Y]_{S,S} = 2pq \cdot P_{\sigma_S} A_{|S|} P_{\sigma_S}^\top$. As a result, to understand the eigenvalues and eigenvectors of $\eye \mathbb{E}[Y]$, it is enough to understand those of $A_{|S|}$.

    \begin{proposition}
        \label{prop:eig-A}
        The eigenvalues of $A = A_{\ell} \in \mathbb{C}^{\ell \times \ell}$ are given by
        \begin{align*}
            \lambda_i(A) = \frac{1}{\tan\left(\frac{2i-1}{2l}\pi\right)} \,\,\text{ for }\,\, 1\le i\le \ell
        \end{align*}
        with the corresponding eigenvectors $v^{(i)} \in \CC^{\ell}$ whose entries are given by
        \begin{align*}
            v^{(i)}_j = \exp\left(-\eye \pi \frac{(2i-1)j}{\ell}\right) \,\,\text{ for }\,\, 1\le j\le \ell.
        \end{align*}
    \end{proposition}
\noindent
We give the proof in Appendix~\ref{app:pf:prop:eig-A}.

Now that we have explicit descriptions of eigenvalues and eigenvectors of $A_{\ell}$, we immediately get the eigenvalues and eigenvectors of $\eye \mathbb{E}[Y]$ as discussed above.

\begin{corollary}\label{cor:expectation-spectral}
    Let $Y \sim \sP_{S, \sigma_S}$ and $\ell = |S|$. The eigenvalues of $\eye \mathbb{E}[Y]$ are \[\lambda_i = 2pq \cdot \frac{1}{\tan\left(\frac{2i-1}{2\ell}\pi\right)} \,\,\text{ for }\,\, 1\le i\le \ell,\]
    together with $n - \ell$ eigenvalues of $0$.

    For $1\le i \le \ell$, the eigenvectors $v^{(i)} \in \CC^n$ corresponding to $\lambda_i$ are given by
    \begin{align*}
        v^{(i)}_j = \One\{j \in S\}\cdot  \exp\left(-\eye \pi \frac{(2i-1)\sigma_S(j)}{\ell}\right) \,\,\text{ for }\,\, 1\le j\le n.
    \end{align*}
\end{corollary}

Let us first give the high-level idea of why a spectral algorithm should be effective.
By the Corollary above, when $q$ is large enough, we expect that $\eye Y = \eye \mathbb{E}[Y] + \eye \left(Y - \mathbb{E}[Y]\right)$ is ``dominated'' by the expectation part $\eye \mathbb{E}[Y]$, and that the top eigenvector of $Y$ should be ``close'' to the top eigenvector of $\eye \mathbb{E}[Y]$, as quantified by the Davis-Kahan theorem. The top eigenvector $v \colonequals v^{(1)}$ of $\eye \mathbb{E}[Y]$ has the nice structure that $v_i = 0$ for $i \not\in S$ and that $v_i$ for $i\in S$ are unit complex numbers which are evenly spaced out radially, with their angles ordered according to $\sigma_S$. Ideally, just by examining the entries of the top eigenvector $\Tilde{v}$ of $\eye Y$, which is close to $v$, we hope to approximately recover both $S$ from the entries of $\Tilde{v}$ with large magnitude, and $\pi_S$ from the way those entries of $\Tilde{v}$ with large magnitude are spread out.

However, since we are working with eigenvectors in $\mathbb{C}^n$, the eigenvectors are only determined up to multiplication by a unit complex number (i.e., in the complex plane, up to a global rotation of the coordinates), we cannot \emph{a priori} claim that $\Tilde{v}$ is close to $v$. Reflecting this limitation, the Davis-Kahan theorem is stated in terms of the difference between the two projection matrices to the span of $v$ and that of $\Tilde{v}$. Fortunately, this technicality does not create too much trouble for us and we can use a simple trick to take advantage of the previous geometric intuition, as we will see in the description of the spectral algorithm and the proof below.

To be concrete, our spectral algorithm performs the following steps as described in Figure~\ref{fig:spectral-log-density-rec}, taking as input the skew-symmetric directed adjacency matrix $Y \sim \sP$:
\begin{figure}[H]
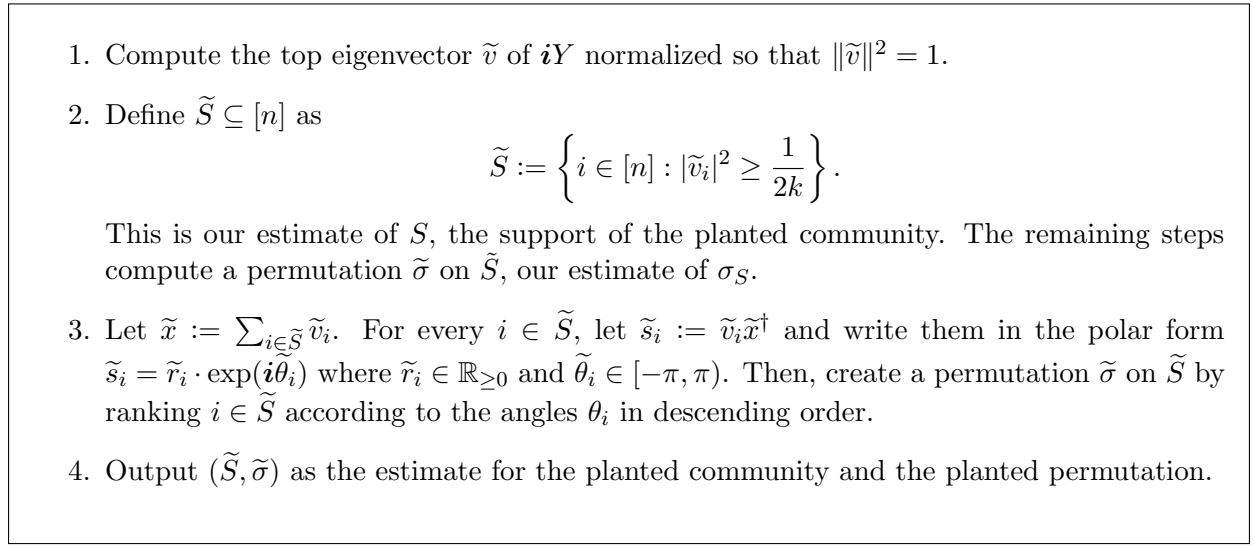

    \centering
    \begin{framed}
        \begin{enumerate}
            \item Compute the top eigenvector $\Tilde{v}$ of $\eye Y$ normalized so that $\|\Tilde{v}\|^2 = 1$.
            \item Define $\Tilde{S} \subseteq [n]$ as
            \[\Tilde{S} \colonequals \left\{i \in [n]: |\Tilde{v}_i|^2 \ge \frac{1}{2k}\right\}.\]
            This is our estimate of $S$, the support of the planted community.
            The remaining steps compute a permutation $\Tilde{\sigma}$ on $\tilde{S}$, our estimate of $\sigma_S$.
            \item Let $\Tilde{x} \colonequals \sum_{i \in \Tilde{S}} \Tilde{v}_i$. For every $i \in \Tilde{S}$, let $\Tilde{s}_i \colonequals \Tilde{v}_i \Tilde{x}^\dagger$ and write them in the polar form $\Tilde{s}_i = \Tilde{r}_i \cdot \exp(\eye \Tilde{\theta}_i)$ where $\Tilde{r}_i \in \RR_{\ge 0}$ and $\Tilde{\theta}_i \in [-\pi, \pi)$. Then, create a permutation $\Tilde{\sigma}$ on $\Tilde{S}$ by ranking $i \in \Tilde{S}$ according to the angles $\theta_i$ in  descending order.

            \item Output $(\Tilde{S}, \Tilde{\sigma})$ as the estimate for the planted community and the planted permutation.
        \end{enumerate}
    \end{framed}
    \vspace{-1em}
    \caption{A description of the spectral algorithm analyzed in Theorem~\ref{thm:log-density-comp-rec}.}
    \label{fig:spectral-log-density-rec}
\end{figure}

First, let us show that under our assumptions, the error in the Davis-Kahan bound is small, formalized as follows.
\begin{proposition} \label{prop:eigen-perturb}
    Suppose $q = \omega(\frac{\sqrt{n}}{k\sqrt{p}})$, $p = \Omega(\frac{\log n}{n})$, and $k = \omega(1)$. Let $Y \sim \sP_{S, \pi_S}$. Let $\Tilde{v}$ and $v$ be the normalized top eigenvectors of $\eye\cdot Y$ and $\eye\cdot \mathbb{E}[Y]$ respectively, with $\|\Tilde{v}\|^2 = \|v\|^2 = 1$.
    If $|S| \ge (1-o(1))k$, then w.h.p.~$\|\Tilde{v}\Tilde{v}^\dagger - vv^\dagger\|_{F} = o(1)$.
\end{proposition}

\begin{proof}
    Let $k' = |S|$.
    By Corollary~\ref{cor:expectation-spectral}, the spectral gap $\delta$ between the largest and the second largest eigenvalues of $\eye \mathbb{E}[Y]$ is
    \begin{align*}
        \delta &= 2pq \cdot \left(\frac{1}{\tan\left(\frac{1}{2k'}\pi\right)} - \frac{1}{\tan\left(\frac{3}{2k'}\pi\right)}\right)\\
        &\sim 2pq\cdot \left(\frac{2k'}{\pi} - \frac{2k'}{3 \pi}\right)\\
        &= \frac{8}{3} p q k'\\
        &\ge \left(\frac{8}{3} - o(1)\right)\cdot p q k.
    \end{align*}

    To apply the Davis-Kahan theorem, we further need to bound the spectral norm of $\Delta \colonequals \eye \left(Y - \mathbb{E}[Y]\right)$. Note that $\Delta$ is Hermitian, and for $1\le i < j \le n$, $\mathbb{E}[[\Delta_{i,j}] = 0$, $\mathbb{E}[|\Delta_{i,j}|^2] \le p$, and $|\Delta_{i,j}| \le 2$.
    By \cite[Theorem 2.7]{benaych2020spectral}, w.h.p.~$\|\Delta\| = O(\sqrt{np})$ provided that $p = \Omega(\frac{\log n}{n})$. Then, by Theorem~\ref{thm:Davis-Kahan}, w.h.p.~we have
    \begin{align*}
        \| \Tilde{v}\Tilde{v}^\dagger - vv^\dagger\|_F &\le \sqrt{2}\cdot \frac{\|\eye \left(Y - \mathbb{E}[Y]\right)\|}{\delta}\\
        &\le O\left(\frac{\sqrt{np}}{pq k}\right)\\
        &\le o(1),
    \end{align*}
    since $q = \omega(\frac{\sqrt{n}}{k\sqrt{p}})$.
\end{proof}

We now show our main result, that our estimate $(\widetilde{S}, \widetilde{\sigma})$ achieves strong recovery in the sense that $d_{\Ham}(S, \widetilde{S}) = o(k)$ and $d_{\KT}(\sigma, \Tilde{\sigma}) = o(k^2)$.

    By Chernoff bound, w.h.p.~$|S| = (1 + o(1))k$. Thus, by Proposition~\ref{prop:eigen-perturb}, w.h.p., for $Y \sim \sP_{S,\sigma}$, $\|\Tilde{v}\Tilde{v}^\dagger - vv^\dagger\|_F \le o(1)$ where $\Tilde{v}$ and $v$ are the normalized top eigenvectors of $\eye Y$ and $\eye \mathbb{E}[Y]$ respectively. By Corollary~\ref{cor:expectation-spectral},
    \begin{align}
        v_i = \frac{\One\{i\in S\}}{\sqrt{|S|}} \cdot \exp\left(-\eye \pi \frac{\sigma_S(i)}{|S|}\right).
    \end{align}
    We also have
    \[ o(1) \ge \|\Tilde{v}\Tilde{v}^\dagger - vv^\dagger\|_F^2 = \sum_{i=1}^n \sum_{j=1}^n |\Tilde{v}_i \Tilde{v}_j^* - v_iv_j^*|^2. \]

    Let us define two sets $A \colonequals \{i \in S: |\Tilde{v}_i|^2 < \frac{1}{2k}\}$ and $B \colonequals \{i \not\in S: |\Tilde{v}_i|^2 \ge \frac{1}{2k}\}$. Then,
    \begin{align}
        o(1) &\ge \sum_{i=1}^n \sum_{j=1}^n |\Tilde{v}_i \Tilde{v}_j^* - v_iv_j^*|^2\\
        &\ge \sum_{i,j \in A} \left(|v_i v_j^*| - |\Tilde{v}_i\Tilde{v}_j^*|\right)^2 + \sum_{i,j \in B} \left(|\Tilde{v}_i\Tilde{v}_j^*|- |v_i v_j^*|\right)^2\\
        &\ge \sum_{i,j \in A} \left(\frac{1}{|S|} - \frac{1}{2k}\right)^2 + \sum_{i,j \in B} \left(\frac{1}{2k} - 0\right)^2\\
        &\ge |A|^2\cdot (1 + o(1))\frac{1}{4k^2} + |B|^2 \cdot \frac{1}{4k^2},
    \end{align}
    where we use that $|S| = (1+o(1))k$ in the last step. Thus, we conclude that $|A| = o(k)$ and $|B| = o(k)$.

    Recall that $\Tilde{S} \subseteq [n]$ is defined by
    \[\Tilde{S} \colonequals \left\{i \in [n]: |\Tilde{v}_i|^2 \ge \frac{1}{2k}\right\}.\]
    By the definitions of $A$ and $B$, we have $\Tilde{S} = S - A + B$. As a result, w.h.p., $d_{\Ham}(S, \Tilde{S}) = |A| + |B| = o(k)$ as desired.

    Next, we want to show that the permutation $\Tilde{\sigma}$ constructed by the spectral algorithm is w.h.p.~close to $\sigma$. Recall that we measure the error in Kendall tau distance, which is
    \[d_{\KT}(\sigma, \Tilde{\sigma}) \colonequals \sum_{\{i,j\} \in \binom{S \cap \Tilde{S}}{2}} \One\{\sigma(i,j) \ne \Tilde{\sigma}(i,j)\}.\]
    By Corollary~\ref{prop:eigen-perturb},
    \begin{align*}
        o(1) &\ge \|\Tilde{v}\Tilde{v}^\dagger - vv^\dagger\|\\
        &= \sum_{i=1}^n \sum_{j=1}^n |\Tilde{v}_i \Tilde{v}_j^* - v_iv_j^*|^2\\
        &\ge \sum_{i\in \Tilde{S}} \sum_{j\in \Tilde{S}} |\Tilde{v}_i \Tilde{v}_j^* - v_iv_j^*|^2\\
        &\ge \frac{1}{|\Tilde{S}|}\sum_{i\in \Tilde{S}} \left|\Tilde{v}_i \left(\sum_{j\in \Tilde{S}} \Tilde{v}_j\right)^* - v_i \left(\sum_{j\in \Tilde{S}} v_j\right)^*\right|^2
        \intertext{by Cauchy-Schwarz inequality,}
        &= \frac{1}{|\Tilde{S}|}\sum_{i\in \Tilde{S}} \left|\Tilde{v}_i \left(\sum_{j\in \Tilde{S}} \Tilde{v}_j\right)^* - v_i \left(\sum_{j\in \Tilde{S} \cap S} v_j\right)^*\right|^2
        \intertext{since $v_i = 0$ for $i \not\in S$,}
        &= \frac{1}{|\Tilde{S}|}\sum_{i\in \Tilde{S}} \left|\Tilde{v}_i \left(\sum_{j\in \Tilde{S}} \Tilde{v}_j\right)^* - v_i \left(\sum_{j\in  S} v_j\right)^* + v_i \left(\sum_{j\in  S \setminus \Tilde{S}} v_j\right)^*\right|^2\\
        &\ge \frac{1}{|\Tilde{S}|}\sum_{i\in \Tilde{S}} \left(\frac{1}{2}\left|\Tilde{v}_i \left(\sum_{j\in \Tilde{S}} \Tilde{v}_j\right)^* - v_i \left(\sum_{j\in  S} v_j\right)^*\right|^2 - \left| v_i \left(\sum_{j\in  S \setminus \Tilde{S}} v_j\right)^*\right|^2\right),
    \end{align*}
    where the last inequality follows from $|a+b|^2 \le 2|a|^2 + 2|b|^2$ for any $a,b \in \CC$. Rearranging the inequality above, we get
    \begin{align*}
        \frac{1}{|\Tilde{S}|} \sum_{i \in \Tilde{S}} \left|\Tilde{v}_i \left(\sum_{j \in \Tilde{S} }\Tilde{v}_j\right)^* - v_i\left(\sum_{j\in S}v_j\right)^*\right|^2
        &\le \frac{2}{|\Tilde{S}|} \sum_{i \in \Tilde{S}}\left|v_i \left(\sum_{j \in S \setminus \Tilde{S}} v_j\right)^*\right|^2 + o(1)\\
        &\le \frac{2}{|\Tilde{S}|} \sum_{i \in \Tilde{S}}|v_i|^2 \left(\sum_{j \in S \setminus \Tilde{S}} |v_j|\right)^2 + o(1)\\
        &\le \frac{2}{|\Tilde{S}|} \cdot |\Tilde{S}| \cdot \frac{1}{|S|} \cdot \left(|S \setminus \Tilde{S}| \cdot \frac{1}{\sqrt{|S|}}\right)^2 + o(1)\\
        &= \frac{2|S \setminus \Tilde{S}|^2}{|S|^2} + o(1)\\
        &\le \frac{2 d_H(S,\Tilde{S})^2}{|S|^2} + o(1)\\
        &\le (1+o(1))\frac{o(k^2)}{k^2} + o(1)\\
        &\le o(1).
    \end{align*}

    For $i \in S$, we may compute
    \begin{align*}
        v_i \left(\sum_{j \in S} v_j\right)^*
        &= \frac{1}{\sqrt{|S|}}\exp\left(-\eye \pi \frac{\sigma_S(i)}{|S|}\right) \cdot \left(\frac{1}{\sqrt{|S|}} \sum_{j=1}^{|S|} \exp\left(-\eye \pi \frac{j}{|S|}\right)\right)^*\\
        &= \frac{1}{|S|}\exp\left(-\eye \pi \frac{\sigma_S(i)}{|S|}\right) \exp\left(\eye \pi \frac{1}{|S|}\right) \frac{2}{1 - \exp\left(\eye \pi \frac{1}{|S|}\right)}\\
        &= \frac{1}{|S|}\exp\left(-\eye \pi \frac{\sigma_S(i)}{|S|} + \eye \pi \frac{1}{2|S|}\right)  \frac{2 \left(\exp\left(\eye \pi \frac{1}{2|S|}\right) - \exp\left(-\eye \pi \frac{1}{|S|}\right)\right)}{2 - \exp\left(\eye \pi \frac{1}{|S|}\right) - \exp\left(-\eye \pi \frac{1}{|S|}\right)}\\
        &= \frac{1}{|S|} \cdot \frac{1}{\left|\sinh\left(\eye \pi \frac{1}{2|S|}\right)\right|} \cdot \exp\left(\eye \pi \left(\frac{1}{2} - \frac{2\sigma_S(i) - 1}{2|S|}\right)\right)
    \end{align*}

    Note that the scaling parameter is on the order of
    \[
        \frac{1}{|S|} \cdot \frac{1}{\left|\sinh\left(\eye \pi \frac{1}{2|S|}\right)\right|} \sim \frac{1}{|S|} \cdot \frac{2|S|}{\pi} = \frac{2}{\pi}.
    \]

    For some $\varepsilon > 0$, let us define \[G_{\varepsilon} = \left\{i \in S \cap \Tilde{S}: \left|\Tilde{v}_i \left(\sum_{j \in \Tilde{S} }\Tilde{v}_j\right)^* - v_i\left(\sum_{j\in S}v_j\right)^*\right|^2 < \varepsilon\right\}.\]

    Observe that $|S \cap \Tilde{S}| - |G_{\varepsilon}| \le o\left(k/\varepsilon\right)$, since
    \begin{align*}
        o(1) &\ge \frac{1}{|\Tilde{S}|} \sum_{i \in \Tilde{S}} \left|\Tilde{v}_i \left(\sum_{j \in \Tilde{S} }\Tilde{v}_j\right)^* - v_i\left(\sum_{j\in S}v_j\right)^*\right|^2\\
        &\ge \frac{1}{|\Tilde{S}|}\sum_{i \in S\cap \Tilde{S} \setminus G_{\varepsilon}} \left|\Tilde{v}_i \left(\sum_{j \in \Tilde{S} }\Tilde{v}_j\right)^* - v_i\left(\sum_{j\in S}v_j\right)^*\right|^2\\
        &\ge (1-o(1))\cdot \frac{1}{k} \cdot \left(|S \cap \Tilde{S}| - |G_{\varepsilon}|\right) \varepsilon.
    \end{align*}
    This means that w.h.p.~$|G_{\varepsilon}| \ge |S \cap \Tilde{S}| - o(k/\varepsilon) \ge |S| - d_H(S, \Tilde{S}) - o(k/\varepsilon) \ge (1 - o(1) - o(1/\varepsilon))k$.

    Now, for every $i \in G_{\varepsilon}$, let us write
    \[\Tilde{v}_i \left(\sum_{j\in \Tilde{S}} \Tilde{v}_j\right)^* = \Tilde{r}_i \cdot \exp(\eye \Tilde{\theta}_i) \]
    where $\Tilde{r}_i \in \RR_{\ge 0}$ and $\Tilde{\theta}_i \in [-\pi, \pi)$. Since we also have
    \[v_i \left(\sum_{j\in S} v_j\right)^* = r \cdot \exp(\eye \theta_i) \]
    for $i\in S$, where $r = \frac{1}{|S|} \cdot \frac{1}{\left|\sinh\left(\eye \pi \frac{1}{2|S|}\right)\right|} \sim  \frac{2}{\pi}$ and $\theta_i = \pi\left(\frac{1}{2} - \frac{2\sigma_S(i) - 1}{2|S|}\right)$, we can conclude that for $i\in G_{\varepsilon},$
    \[ |\Tilde{\theta}_i - \theta_i| \le \arcsin\left(\frac{\sqrt{\varepsilon}}{r}\right) \le \frac{\sqrt{\varepsilon}}{r}. \]
    Thus, if for a pair of $i,j \in G_{\varepsilon}$, $\theta_i$ and $\theta_j$ are well separated, then $\Tilde{\theta}_i$ and $\Tilde{\theta}_j$ will still have the correct relative order, and $\Tilde{\sigma}$ ranks this pair the same way as $\sigma$ does. It is easy to check that the number of pairs of $\{i,j\} \in \binom{G_{\varepsilon}}{2}$ for which $|\theta_i - \theta_j| > 2\cdot \frac{\sqrt{\varepsilon}}{r}$ is at least
    \begin{align*}
        \frac{1}{2}|G_{\varepsilon}| \cdot \left(|G_{\varepsilon}| - O(\sqrt{\varepsilon} \cdot k)\right) \ge (1 - o(1) - o(1/\varepsilon) - O(\sqrt{\varepsilon}))\frac{k^2}{2}
    \end{align*}
    which is at least $(1 - o(1))\frac{k^2}{2}$ by setting $\varepsilon$ to be an appropriate vanishing function depending on the $o(1)$ term. Thus, the permutation $\Tilde{\sigma}$, created according to $\Tilde{\theta}_i$ for $i \in \Tilde{S}$, ranks at least $(1 - o(1))\frac{k^2}{2}$ pairs correctly among $\binom{S \cap \Tilde{S}}{2}$, and consequently we conclude that w.h.p.~
    \begin{align*}
        d_{KT}(\sigma, \Tilde{\sigma}) &\le \left|\binom{S \cap \Tilde{S}}{2}\right| - (1 - o(1))\frac{k^2}{2}\\
        &\le (1+o(1))\frac{k^2}{2} - (1 - o(1))\frac{k^2}{2}\\
        &\le o(k^2),
    \end{align*}
    completing the proof.

\subsubsection{Lower Bound} \label{sec:log-density-comp-rec-lower}

In this section, we will prove that for degree  $D = n^{o(1)}$ polynomials, weak support recovery is impossible in the log-density regime once $\beta < \alpha + \frac{1}{2}\gamma + \frac{1}{2}$.
Specifically, we will show that, in the notation of Section~\ref{sec:low-deg}, the low-degree correlation with this choice of $D$ is $\Corr_{\leq D}(\sP)^2 \ll \frac{k}{n}$, from which this claim follows by the definitions there.

Recall, under our definition of the PRS model's distribution $\sP$, the observed matrix $Y \in \{0, \pm 1\}^{n \times n}$ is generated when $Y \sim \sP$ in the following way:
\begin{enumerate}
    \item First, sample a random vector $\theta \in \{0,1\}^{n}$ such that $\theta_i \stackrel{\text{iid}}{\sim} \text{Bern}(k/n)$.
    \item Next, sample a permutation $\pi \in S_n$ uniformly at random.
    \item Then, generate each entry of $Y_{i,j}$ for $i < j$ independently as
    \begin{align*}
        Y_{i,j} &= \begin{cases}
            0 & \quad \text{ with probability } 1 - p\\
            T_{i,j} & \quad \text{ with probability } p
        \end{cases}
    \end{align*}
    where $T_{i,j} \sim \text{Rad}(1/2 + q \cdot \theta_i\theta_j \pi(i,j))$ independently. The lower diagonal entries are set to $Y_{j,i} = -Y_{i,j}$.
\end{enumerate}
In particular, this model can be viewed as a general binary observation model \cite{schramm2022computational} with censorship, where the underlying tournament matrix $T$ is drawn from a specific binary observation model, with each entry of $T$ only revealed with probability $p$.

We claim that we can equivalently sample $Y$ using the following alternative procedure:
\begin{enumerate}
    \item Sample the random vector $\theta \in \{0,1\}^{n}$ such that $\theta_i \stackrel{\text{iid}}{\sim} \text{Bern}(k/n)$ and the permutation $\pi \in S_n$ as before.
    \item For $i < j$, sample $P_{i,j} \stackrel{\text{iid}}{\sim} \text{Bern}(p)$, $Z_{i,j} \stackrel{\text{iid}}{\sim} \text{Rad}(1/2)$, and $X_{i,j} \sim \text{Rad}(1/2 + q\theta_i\theta_j\pi(i,j))$ independently. For $i < j$, we set $P_{i,j} = P_{j,i}$, $Z_{i,j} = -Z_{j,i}$, and $X_{i,j} = -X_{j,i}$.
    \item Then, generate \[Y_{i,j} = P_{i,j} \cdot \left[(1 - \theta_i\theta_j)Z_{i,j} + \theta_i\theta_j X_{i,j}\right],\]
    or, in matrix form $Y = P \circ \left[\left(J - \theta\theta^\top\right) \circ Z + \left(\theta \theta^\top\right) \circ X \right]$.
\end{enumerate}

Let $f \in \mathbb{R}[Y]_{\le D}$. By Proposition~\ref{prop:h-orthonormal}, there is a unique expansion
\begin{align}
    f(Y) = \sum_{A, B} \hat{f}_{A, B} \cdot h_{A, B}(Y)
\end{align}
in the basis of orthogonal polynomials for $\sQ$ defined in Definition~\ref{def:h-poly}. Recall they are defined as
\[ h_{A, B}(Y) \colonequals \frac{1}{p^{|A|/2}} Y^A \frac{1}{(p(1-p))^{|B|/2}} (Y^{\circ 2}-pJ)^B. \]
For convenience, let us denote $b_{A, B} \colonequals \frac{1}{p^{|A|/2}} \cdot \frac{1}{(p(1-p))^{|B|/2}}$ and $b_A \colonequals b_{A, \emptyset}$, so that we can write $h_{A,B}(Y) = b_{A,B}\cdot Y^A (Y^{\circ 2} - pJ)^B$ and  $h_{A}(Y) = b_{A}\cdot Y^A$.

Recall that, following the discussion in Section~\ref{sec:low-deg}, we are interested in proving an upper bound low-degree correlation
\begin{align}
    \Corr_{\le D}(\sP) \colonequals \sup_{f \in \mathbb{R}[Y]_{\le D}} \frac{\Ex_{(\theta, Y) \sim \sP} \left[f(Y)\theta_1\right]}{\sqrt{\Ex_{Y \sim \sP} f(Y)^2}}.
\end{align}
We follow the strategy in \cite{schramm2022computational} and simplify the denominator appearing in $\Corr_{\le D}$ by applying Jensen's inequality:
\begin{align*}
    \Ex_{Y \sim \sP}[f(Y)^2]
    &\ge \Ex_{(Z,P) \sim \sP} \left[\Ex_{(X, \theta) \sim \sP} \left[f\left(P \circ \left( \left(J - \theta \theta^\top\right) \circ Z +\theta \theta^\top \circ X\right)\right)\right]^2\right]\\
    &\equalscolon \Ex_{(Z,P) \sim \sP} \left[g(Z, P)^2\right],
\end{align*}
where we define $g(Z, P) \colonequals \Ex_{(X, \theta) \sim \sP} f\left(P \circ \left( \left(J - \theta \theta^\top\right) \circ Z +\theta \theta^\top \circ X\right)\right)$.
{\allowdisplaybreaks
We may explicitly compute the coefficents of $g(Z,P) = \sum_{A,B} \hat{g}_{A,B} \cdot h_{A,B}(P \circ Z)$ using the expansion of $f$, as follows.
\begin{proposition}
    \label{prop:g-expansion}
    In the above setting,
    \[
    g(Z,P) = \sum_{A', B'} \hat{g}_{A', B'} \cdot h_{A', B'}(P \circ Z),
    \]
    where $\hat{g} = M \hat{f}$, with $M_{(A^{\prime}, B^{\prime}), (A, B)} = 0$ unless $A^{\prime} \subseteq A$, $B \subseteq B^{\prime}$, and $B^{\prime} - B \subseteq A - A^{\prime}$, and when these conditions are satisfied, then
    \[ M_{(A', B'), (A, B)} = \frac{b_{A, B}}{b_{A', B'}}\cdot  p^{|A - A' + B - B'|}\cdot  \mathbb{E}_{X} \left[X^{A - A'}\right] \mathbb{E}_{\theta} \left[\left(J - \theta \theta^\top\right)^{A'}\left(\theta \theta^\top\right)^{ A - A'}\right]. \]
\end{proposition}
\noindent
We give the proof---a straightforward but tedious calculation---in Appendix~\ref{app:pf:prop:g-expansion}.

Notice that the indices can be arranged so that $M$ is upper triangular, and moreover the diagonal entries $M_{(A, B), (A, B)} = \Ex_{\theta} \left[\left(J - \theta \theta^\top\right)^{A}\right] \ne 0$. Thus, $M$ is invertible.

Noticing that $\Ex_{Z,P}[g(Z,P)^2] = \sum_{A,B} \left(\hat{g}_{A,B}\right)^2 = \|\hat{g}\|^2$ due to the identical calculation as in Proposition~\ref{prop:h-orthonormal} and using that $\hat{g} = M \hat{f}$, we may further upper bound $\Corr_{\le D}$ by
\begin{align}
    \text{Corr}_{\le D}(\sP) &= \sup_{f} \frac{\Ex_{\theta, Y}[f(Y) \theta_1]}{\sqrt{\Ex_{Y}[f(Y)^2]}}\\
    &\le \sup_{f} \frac{\Ex_{\theta, Y}[f(Y)  \theta_1]}{\sqrt{\Ex_{Z,P}[g(Z,P)^2]}}
    \intertext{by the Jensen's inequality trick described earlier,}
    &= \sup_{\substack{\hat{f}\\ \hat{g} = M \hat{f}}} \frac{\sum_{A,B} \hat{f}_{A,B}\cdot \Ex_{\theta, Y}[h_{A,B}(Y) \theta_1]}{\sqrt{\|\hat{g}\|^2}}\\
    &= \sup_{\substack{\hat{f}\\ \hat{g} = M \hat{f}}} \frac{\langle c, \hat{f} \rangle }{\sqrt{\|\hat{g}\|^2}}
    \intertext{by defining a vector $c$ with entries given by $c_{A, B} \colonequals \Ex_{\theta, Y}[h_{A, B}(Y) \theta_1]$,}
    &= \sup_{\hat{g}} \frac{\langle c, M^{-1}\hat{g} \rangle }{\sqrt{\|\hat{g}\|^2}}\\
    &= \|c^\top M^{-1}\|\\
    &\equalscolon \|w\|, \label{ineq:corr-bound}
\end{align}
where in the last step we define $w^\top \colonequals c^\top M^{-1}$. We may compute the coefficients $c_{A,B}$ as
\begin{align*}
    c_{A, B} &= \Ex[h_{A, B}(Y) \theta_1]\\
    &= b_{A, B}\cdot \Ex \left[P^A \left[\left(J - \theta\theta^\top\right) \circ Z + \left(\theta\theta^\top\right) \circ X\right]^A (P - p)^{B} \cdot \theta_1\right]\\
    &= b_{A, B} \cdot  \Ex_{P}[P^A (P - p)^B ] \Ex_{Z,X,\theta} \left[\left(\sum_{A' \subseteq A} X^{A'} Z^{A - A'}\left(\theta\theta^\top\right)^{A'}\left(J - \theta\theta^\top\right)^{A - A'}\right)\cdot \theta_1\right]\\
    &= \frac{1}{p^{|A|/2}} \cdot \frac{1}{(p(1-p))^{|B|/2}}\cdot p^{|A|}\One\{B = \emptyset\} \cdot \Ex_{X}\left[X^{A}\right] \Ex_{\theta}\left[ \left(\theta\theta^\top\right)^{A}\theta_1\right]\\
    &= \One\{B = \emptyset\} p^{|A|/2}\Ex_{X}\left[X^{A}\right] \Ex_{\theta}\left[ \left(\theta\theta^\top\right)^{A}\theta_1\right].
\end{align*}
Observe that the vector $c$ is only supported on entries $c_{A, B}$ with $B = \emptyset$.

Solving for $w$ from $w^\top M = c^\top$, we solve the system
\[ \sum_{A',B'} w_{A',B'} M_{(A, B'), (A, B)} = c_{A, B}. \]
This gives
\begin{align*}
    &\One\{B = \emptyset\} \cdot p^{|A|/2}\Ex_{X}\left[X^{A}\right] \Ex_{\theta}\left[ \left(\theta\theta^\top\right)^{A}\theta_1\right]\\
    &= \sum_{A' \subseteq A} \sum_{\substack{B' \supseteq B \\ B' - B \subseteq A - A'}} w_{A', B'}\cdot \frac{b_{A, B}}{b_{A', B'}}\cdot  p^{|A - A' + B - B'|} \Ex_{X} \left[X^{A - A'}\right] \Ex_{\theta} \left[\left(J - \theta \theta^\top\right)^{A'}\left(\theta \theta^\top\right)^{ A - A'}\right].
\end{align*}
Thus, we arrive at the following recursive formula for $w$:
\begin{align*}
    &w_{A, B} \cdot \Ex_{\theta} \left[\left(J - \theta \theta^\top\right)^{A}\right]\\
    &= \One\{B = \emptyset\} p^{|A|/2}\Ex_{X}\left[X^{A}\right] \Ex_{\theta}\left[ \left(\theta\theta^\top\right)^{A}\theta_1\right]\\
    &\quad- \sum_{\substack{A' \subseteq A\\
    B' \supseteq B\\ B' - B \subseteq A - A' \\
    (A', B') \ne (A, B)}} w_{A', B'}   \cdot\frac{b_{A, B}}{b_{A', B'}}\cdot  p^{|A - A' + B - B'|}\cdot  \Ex_{X} \left[X^{A - A'}\right] \Ex_{\theta} \left[\left(J - \theta \theta^\top\right)^{A'}\left(\theta \theta^\top\right)^{ A - A'}\right]. \numberthis \label{eq:w-recursion}
\end{align*}

We establish the following initial facts about the coefficients resulting from solving this recursion, which control which $(A, B)$ contribute to the norm.
\begin{proposition}
    \label{prop:w-zero-B}
    $w_{A, B} = 0$ if $B \ne \emptyset$.
\end{proposition}

\begin{proposition}
    \label{prop:w-zero-cc}
    Suppose $A$ has any connected component containing an edge but not containing vertex $1$. Then $w_{A, \emptyset} = 0$.
\end{proposition}
\noindent
We defer the proofs to Appendices~\ref{app:pf:prop:w-zero-B} and~\ref{app:pf:prop:w-zero-cc}, respectively.

Thus we need only consider $w_{A, \emptyset}$ for $A$ a connected graph (aside from isolated vertices) containing vertex 1.
We next bound these terms.

\begin{proposition}\label{prop:w-bound}
    We have
    \begin{align*}
        w_{\emptyset, \emptyset} &= \frac{k}{n}, \\
        |w_{A, \emptyset}| &\le \left(1 - \frac{k}{n}\right)^{-2|A|^2} (|A| + 1)^{|A|} \left(4p q^2\right)^{|A|/2} \left(\frac{k}{n}\right)^{|V(A)|},
    \end{align*}
    for $A$ with $|A| \ge 1$.
\end{proposition}

\begin{proof}
The case $A = \emptyset$ follows immediately from \eqref{eq:w-recursion}.
It remains to consider the case when $A \ne \emptyset$ and $B = \emptyset$. Again using \eqref{eq:w-recursion}, we have
\begin{align*}
    |w_{A, \emptyset}|
    &\le \left|\frac{1}{\left|\Ex_{\theta}\left[\left(J-\theta\theta^\top\right)^{A}\right]\right|}\right|\cdot \Bigg|p^{|A|/2} \Ex_{X}\left[X^A\right] \Ex_{\theta}\left[\left(\theta\theta^\top\right)^{A} \theta_1\right] \\
    &\quad - \sum_{A' \subsetneq A} w_{A', \emptyset}    \frac{b_{A, \emptyset}}{b_{A', \emptyset}}\cdot  p^{|A - A'|}\cdot  \Ex_{X} \left[X^{A - A'}\right] \Ex_{\theta} \left[\left(J - \theta \theta^\top\right)^{A'}\left(\theta \theta^\top\right)^{ A - A'}\right]\Bigg|\\
    &\le \left|\frac{1}{\left|\Ex_{\theta}\left[\left(J-\theta\theta^\top\right)^{A}\right]\right|}\right|\cdot \Bigg(\left|p^{|A|/2} \Ex_{X}\left[X^A\right] \Ex_{\theta}\left[\left(\theta\theta^\top\right)^{A} \theta_1\right]\right| \\
    &\quad + \sum_{A' \subsetneq A} |w_{A', \emptyset}| \cdot \left|\frac{b_{A, \emptyset}}{b_{A', \emptyset}}\cdot  p^{|A - A'|}\cdot  \Ex_{X} \left[X^{A - A'}\right] \Ex_{\theta} \left[\left(J - \theta \theta^\top\right)^{A'}\left(\theta \theta^\top\right)^{ A - A'}\right]\right|\Bigg)
\end{align*}
by triangle inequality. We further use the following inequalities which can be verified as in Proposition~\ref{prop:adj-moments}.
\begin{align*}
    \left|\Ex\left[X^A\right]\right| &\le (2q)^{|A|}\\
    \left|\Ex_{\theta}\left[\left(J-\theta\theta^\top\right)^{A}\right]\right| &\ge \left(1 - \frac{k}{n}\right)^{|V(A)|}\\
    \left|\Ex_{\theta}\left[\left(\theta\theta^\top\right)^{A}\theta_1\right]\right| &= \left(\frac{k}{n}\right)^{|V(A) \cup \{1\}|}\\
    \left|\Ex_{\theta}\left[\left(J-\theta\theta^\top\right)^{A'}\left(\theta\theta^\top\right)^{A - A'}\right]\right| &\le \left(\frac{k}{n}\right)^{|V(A-A')|}
\end{align*}
to get
\begin{align*}
    |w_{A, \emptyset}|
    &\le \left|\frac{1}{\left|\Ex_{\theta}\left[\left(J-\theta\theta^\top\right)^{A}\right]\right|}\right|\cdot \Bigg(\left|p^{|A|/2} \Ex_{X}\left[X^A\right] \Ex_{\theta}\left[\left(\theta\theta^\top\right)^{A} \theta_1\right]\right| \\
    &\quad + \sum_{A' \subsetneq A} |w_{A', \emptyset}| \cdot \left|\frac{b_{A, \emptyset}}{b_{A', \emptyset}}\cdot  p^{|A - A'|}\cdot  \Ex_{X} \left[X^{A - A'}\right] \Ex_{\theta} \left[\left(J - \theta \theta^\top\right)^{A'}\left(\theta \theta^\top\right)^{ A - A'}\right]\right|\Bigg)\\
    &\le \left(1 - \frac{k}{n}\right)^{-|V(A)|} \cdot \Bigg(p^{|A|/2} (2q)^{|A|} \left(\frac{k}{n}\right)^{|V(A) \cup \{1\}|} \\
    &\quad + \sum_{A' \subsetneq A} |w_{A', \emptyset}| \cdot p^{|A - A'|/2}(2q)^{|A - A'|}\left(\frac{k}{n}\right)^{|V(A - A')|}\Bigg)\\
    &\le \left(1 - \frac{k}{n}\right)^{-2|A|^2} (|A| + 1)^{|A|} \left(4p q^2\right)^{|A|/2} \left(\frac{k}{n}\right)^{|V(A)|},
\end{align*}
where the last inequality follows from the same steps as in the proof of \cite[Lemma 3.9]{schramm2022computational}.
\end{proof}

Finally, we need the following estimate of the number of connected $A$ that contain the vertex $1$:
\begin{proposition}[{\cite[Lemma 3.5]{schramm2022computational}}] \label{prop:connected-A-bound}
    For integers $d\ge 1$ and $0 \le h \le d$, the number of $A \subseteq \binom{[n]}{2}$ such that (i) $A$ is connected, (ii) $1 \in V(A)$, (iii) $|A| = d$ and $|V(A)| = d+1-h$, is at most $(dn)^d \left(\frac{d}{n}\right)^d$.
\end{proposition}

We are now ready to finish the proof by plugging in these estimates to our earlier formula \eqref{ineq:corr-bound}:
{\allowdisplaybreaks
\begin{align*}
    &\text{Corr}_{\le D}(\sP)^2 \\ &\le \|w\|^2\\
    &= \sum_{A, B} w_{A, B}^2\\
    &= \sum_{A} w_{A, \emptyset}^2
    \intertext{by Proposition~\ref{prop:w-zero-B},}
    &\le \frac{k^2}{n^2} + \sum_{d=1}^D \sum_{h=0}^d \sum_{\substack{A \text{ connected}:\\
    1 \in V(A), |A| = d,\\
    |V(A)| = d+1-h}} \left(1 - \frac{k}{n}\right)^{-4d^2} (d+1)^{2d} \left(4pq^2\right)^d \left(\frac{k}{n}\right)^{2(d+1-h)}
    \intertext{by Proposition~\ref{prop:w-zero-cc} and Proposition~\ref{prop:w-bound}, and that every connected $A$ satisfies $|V(A)| \le |A| + 1$,}
    &\le \frac{k^2}{n^2} + \sum_{d=1}^D \sum_{h=0}^d (dn)^d \left(\frac{d}{n}\right)^h \left(1 - \frac{k}{n}\right)^{-4d^2} (d+1)^{2d} \left(4pq^2\right)^d \left(\frac{k}{n}\right)^{2(d+1-h)}
    \intertext{using the bound on the number of $A$ in Proposition~\ref{prop:connected-A-bound},}
    &\le \frac{k^2}{n^2} + \frac{k^2}{n^2}\sum_{d=1}^D \sum_{h=0}^d \left(d \cdot \frac{n}{k^2}\right)^h \left(4d(d+1)^2\left(1 - \frac{k}{n}\right)^{-4d^2} \cdot pq^2 \frac{k^2}{n}\right)^d\\
    &\le \frac{k^2}{n^2} + \frac{k^2}{n^2}\sum_{d=1}^D \sum_{h=0}^d \left(D \cdot \frac{n}{k^2}\right)^h \left(4D(D+1)^2\left(1 - \frac{k}{n}\right)^{-4D^2} \cdot pq^2 \frac{k^2}{n}\right)^d\\
    &= \frac{k^2}{n^2}\sum_{h=0}^D \left(4D^2(D+1)^2\left(1 - \frac{k}{n}\right)^{-4D^2}\cdot pq^2\right)^h \\
    &\hspace{1cm} \sum_{d=h}^D    \left(4D(D+1)^2\left(1 - \frac{k}{n}\right)^{-4D^2}\cdot p q^2  \frac{k^2}{n}\right)^{d-h}\\
    &\le \frac{k^2}{n^2} \cdot \frac{1}{\left(1 - 4D^2(D+1)^2\left(1 - \frac{k}{n}\right)^{-4D^2}\cdot  pq^2\right)\left(1 - 4D(D+1)^2 \left(1 - \frac{k}{n}\right)^{-4D^2}\cdot p q^2  \frac{k^2}{n}\right)}. \numberthis \label{ineq:corr-bound-1}
\end{align*}

Provided that $$q \ll \min\left\{\frac{\left(1 - \frac{k}{n}\right)^{2D^2}}{D(D+1)}\cdot \frac{1}{\sqrt{p}}, \frac{\left(1 - \frac{k}{n}\right)^{2D^2}}{\sqrt{D}(D+1)} \cdot \frac{\sqrt{n}}{k\sqrt{p}}\right\},$$ we have $\text{Corr}_{\le D}(\sP)^2 \le (1 + o(1))\frac{k^2}{n^2} \ll \frac{k}{n}$, since in the log-density setting we always have $k \ll n$.
In particular, in the log-density setting we find that this condition holds whenever $\beta < \alpha + \frac{1}{2}\gamma + \frac{1}{2}$, completing the proof.

}

\begin{remark}\label{rem:alternative-corr-bound}
    By a slightly different strategy of decomposing $Y \sim \sP$ into ``signal'' part plus the ``noise'' part, we can also show that
    \begin{align}
        \Corr_{\le D}(\sP)^2 \le \frac{k^2}{n^2} \cdot \frac{1}{\left(1 - 4D^2(D+1)^2\cdot  \frac{pq^2}{(1-2q)^2}\right)\left(1 - 4D(D+1)^2 \cdot \frac{pq^2}{(1-2q)^2}  \cdot\frac{k^2}{n}\right)}, \label{ineq:corr-bound-2}
    \end{align}
    as long as $q < \frac{1}{2}$. We give a proof of this alternative bound in Appendix~\ref{app:rem:alternative-corr-bound}.

    We note that \eqref{ineq:corr-bound-1} gives a non-trivial bound for any $q \in (0, \frac{1}{2}]$ so long as $k = O(n/D^2)$, while \eqref{ineq:corr-bound-2} has a slightly better degree dependency without the $(1 - \frac{k}{n})^{-4D^2}$ term at the cost of a deteriorating bound as $q$ approaches $\frac{1}{2}$. These slight differences, however, are not visible at the level of granularity of the log-density setting.
\end{remark}
}

\subsection{Statistical Recovery: Proof of Theorem~\ref{thm:log-density-stat-rec}}

\subsubsection{Upper Bound}

We consider a brute-force search algorithm that is similar to the maximum likelihood estimator.
Let $\mathcal{S}_k = \{\Tilde{S} \subseteq [n]: |\Tilde{S}| = k\}$. On input $Y \in \{0, \pm 1\}^{\binom{[n]}{2}}$, our algorithm outputs the maximizer
\begin{align*}
    \argmax_{\Tilde{S} \in \mathcal{S}_k, \Tilde{\pi}} \sum_{\{i,j\} \in \binom{\Tilde{S}}{2} } Y_{i,j} \Tilde{\pi}(i,j).
\end{align*}
We will show that if $q \gtrsim \frac{1}{\sqrt{k p} } \cdot \text{poly}\log(n)$ (in fact, $q \gg \frac{1}{\sqrt{k p}} \log n$), then the maximizer $(\Tilde{S}, \Tilde{\pi})$ achieves strong recovery of $(S, \pi)$.
In the log-density setting, the condition becomes $\beta > 2\alpha + \gamma$, as claimed.

In this proof, let us adopt the term \emph{with polynomially high probability (w.p.h.p.)} for events occurring with probability at least $1 - \frac{1}{n^c}$ for some $c > 0$.

We denote $\text{val}(\Tilde{S}, \Tilde{\pi}) \colonequals \sum_{\{i,j\} \in \binom{\Tilde{S}}{2} } Y_{i,j} \Tilde{\pi}(i,j)$, the objective function above.
Define $\mathcal{B}_t = \mathcal{B}_t(S, \pi)$ to be the set of $(\Tilde{S}, \Tilde{\pi})$ such that $\Tilde{S} \in \mathcal{S}_k$, and either $d_{\Ham}(S, \Tilde{S}) \ge t$, or $d_{\Ham}(S, \Tilde{S}) < t$ and $d_{\KT}(\pi, \Tilde{\pi}) \ge kt$.
Intuitively, this is the set of ``bad'' estimates of the planted community and permutation with some quantitative amount $t$ of error.
Clearly, $|B_t| \le |\mathcal{S}_k| \cdot k! \le (nk)^k$.

    Let $S^* = \arg\min_{\Tilde{S} \in \mathcal{S}_k} |\Tilde{S} - S|$. We note that the distribution of the cardinality $|S|$ of $S$ has the binomial law $\text{Bin}(n, k/n)$. By Chernoff bound, we conclude that w.p.h.p.~$|S| \in [k - \Tilde{O}(\sqrt{k}) , k + \Tilde{O}(\sqrt{k})]$. Thus, w.p.h.p.~$|S^* - S| = \Tilde{O}(\sqrt{k})$ and $|S - S^*| = \Tilde{O}(\sqrt{k})$. Let $\pi^*$ be a permutation on $S^*$ that agrees with $\pi$ on the intersection $S \cap S^*$. Then, conditioning on $S^*$ and $\pi^*$, $\text{val}(S^*, \pi^*) = \sum_{\{i,j\} \in \binom{S^*}{2} } Y_{i,j} \pi^*(i,j)$ stochastically dominates
    \begin{align}
        \sum_{\{i,j\} \in \binom{S^* \cap S}{2}} P_{i,j} A_{i,j} + \sum_{\{i,j\} \in \binom{S^*}{2} - \binom{S}{2} } P_{i,j}B_{i,j},
    \end{align}
    where $P_{i,j} \stackrel{\text{iid}}{\sim} \text{Bern}(p)$, $A_{i,j} \stackrel{\text{iid}}{\sim} \text{Rad}(1/2+q)$, and $B_{i,j} \stackrel{\text{iid}}{\sim} \text{Rad}(1/2 - q)$. Thus,
    \begin{align*}
        \Ex[\text{val}(S^*, \pi^*)] &\ge \binom{k - \Tilde{O}(\sqrt{k})}{2} \cdot p \cdot 2q - \left(\binom{k}{2} - \binom{k - \Tilde{O}(\sqrt{k})}{2}\right) \cdot p \cdot 2q\\
        &\ge k^2 p q - \Tilde{O}\left(k^{\frac{3}{2}}pq\right),
    \end{align*}
    and by Bernstein inequality, $\text{val}(S^*, \pi^*)$ w.p.h.p.~satisfies $|\mathrm{val}(S^*, \pi^*) - \EE \mathrm{val}(S^*, \pi^*)| \leq \Tilde{O}(k\sqrt{p})$.
    Thus we have shown that there exists at least one feasible point in our maximization, $(S^*, \pi^*)$, which w.p.h.p.\ achieves a large value
    \begin{align*}
        \mathrm{val}(S^*, \pi^*) \geq k^2 p q - \Tilde{O}\left(k^{\frac{3}{2}}pq + k\sqrt{p}\right). \label{ineq:S*-obj}
    \end{align*}

    We next show that, again w.p.h.p., every point in $\mathcal{B}_t$ has smaller value.
    If we can do this, then $(\Tilde{S}, \Tilde{\pi}) \notin \mathcal{B}_t$ w.p.h.p., and if this holds for $t$ small enough then we will obtain our result.

    For every $(\Tilde{S}, \Tilde{\pi}) \in \mathcal{B}_t$, consider $\text{val}(\Tilde{S}, \Tilde{\pi})$. By definition of $\mathcal{B}_t$, $d_{\Ham}(S, \Tilde{S}) \ge t$ or $d_{\KT}(\pi, \Tilde{\pi}) \ge kt$, and $\text{val}(\Tilde{S}, \Tilde{\pi})$ is stochastically dominated by
    \begin{align}
        \sum_{\substack{\{i,j\} \in \binom{\Tilde{S} \cap S}{2}:\\ \Tilde{\pi}(i,j)=\pi(i,j) }} P_{i,j} A_{i,j} + \sum_{\substack{\{i,j\} \in \binom{\Tilde{S} \cap S}{2}:\\ \Tilde{\pi}(i,j)\ne \pi(i,j) }} P_{i,j} B_{i,j} + \sum_{\{i,j\} \in \binom{\Tilde{S}}{2} - \binom{S}{2} } P_{i,j}C_{i,j},
    \end{align}
    where $P_{i,j} \stackrel{\text{iid}}{\sim} \text{Bern}(p)$, $A_{i,j} \stackrel{\text{iid}}{\sim} \text{Rad}(1/2+q)$, $B_{i,j} \stackrel{\text{iid}}{\sim} \text{Rad}(1/2-q)$, and $C_{i,j} \stackrel{\text{iid}}{\sim} \text{Rad}(1/2)$. In particular, either the first sum has at most $\binom{(k + |S| - t)/2}{2}$ terms (since $d_{\Ham}(S, \Tilde{S}) \ge t$ and $2|\Tilde{S} \cap S| = |\Tilde{S}| + |S| - d_{\Ham}(S, \Tilde{S})$ ), or the second sum has at least $kt$ terms (since $d_{\KT}(\pi, \Tilde{\pi}) \ge kt$).

    Thus,
    \begin{align}
        \Ex[\text{val}(\Tilde{S}, \Tilde{\pi})] &\le \min\left\{\binom{(k + |S| - t)/2}{2}\cdot p \cdot 2q,  \left(\binom{k}{2} - kt\right)\cdot p \cdot 2q - kt \cdot p \cdot 2q\right\}\\
        &\le k^2pq + \Tilde{O}\left(k^{\frac{3}{2}}pq\right) - \Omega(kt p q),
    \end{align}
    where we use that $|S| \le k + \tilde{O}(\sqrt{k})$ w.p.h.p. Moreover, for each $(\Tilde{S}, \Tilde{\pi}) \in \mathcal{B}_t$, again by Bernstein inequality, we have for any $\lambda > 0$
    \begin{align*}
        \mathbb{P}\left[\text{val}(\Tilde{S}, \Tilde{\pi}) - \Ex\left[\text{val}(\Tilde{S}, \Tilde{\pi})\right] \ge \lambda \right] \le \exp\left(-\frac{\frac{1}{2}\lambda^2}{\binom{k}{2}p + \frac{2}{3}\lambda}\right),
    \end{align*}
    since for every $\{i,j\} \in \binom{\Tilde{S}}{2}$, the random variable $X_{i,j} = Y_{i,j}\Tilde{\pi}(i,j)$ satisfies $|X_{i,j} - \Ex[X_{i,j}]| \le 2$ and $\Var(X_{i,j}) \le p$. We will then choose $\lambda > 0$ large enough so that we may apply a union bound over all $(\Tilde{S}, \Tilde{\pi}) \in \mathcal{B}_t$ to conclude that, w.h.p.,
    \[\max_{(\Tilde{S}, \Tilde{\pi})\in \mathcal{B}_t} \text{val}(\Tilde{S}, \Tilde{\pi})  \le \max_{(\Tilde{S}, \Tilde{\pi})\in \mathcal{B}_t} \Ex\left[\text{val}(\Tilde{S}, \Tilde{\pi})\right] + \lambda.\]
    Since $|\mathcal{B}_t| \le (nk)^k$, it is enough to choose $\lambda > 0$ so that
    \begin{align*}
        \exp\left(-\frac{\frac{1}{2}\lambda^2}{\binom{k}{2}p + \frac{2}{3}\lambda}\right) (nk)^k = o(1).
    \end{align*}
    Solving for $\lambda$, we need
    \[ -\frac{\frac{1}{2}\lambda^2}{\binom{k}{2}p + \frac{2}{3}\lambda} + k\log(nk) \ll 0, \]
    for which it suffices to choose
    \[ \lambda = \Tilde{\Theta}\left(\max\left\{k^{\frac{3}{2}} \sqrt{p}, k\right\}\right). \]

    Thus, by a union bound, we conclude that w.h.p.
    \begin{align*}
        \max_{(\Tilde{S}, \Tilde{\pi})\in \mathcal{B}_t} \text{val}(\Tilde{S}, \Tilde{\pi})  &\le \max_{(\Tilde{S}, \Tilde{\pi})\in \mathcal{B}_t} \Ex\left[\text{val}(\Tilde{S}, \Tilde{\pi})\right] + \Tilde{O}\left(\max\left\{k^{\frac{3}{2}} \sqrt{p}, k\right\}\right)\\
        &\le k^2pq + \Tilde{O}\left(k^{\frac{3}{2}}pq\right) - \Omega(kt p q) + \Tilde{O}\left(\max\left\{k^{\frac{3}{2}} \sqrt{p}, k\right\}\right). \numberthis \label{ineq:Bt-obj}
    \end{align*}

    Recall from \eqref{ineq:S*-obj} that w.h.p., $(S^*, \pi^*)$ achieves
    \begin{align*}
        \text{val}(S^*, \pi^*) \ge k^2 p q - \Tilde{O}\left(k^{\frac{3}{2}}pq\right) - \Tilde{O}(k\sqrt{p}).
    \end{align*}

    Recall also our strategy: if for specific $p, k, q$ and $t = o(k)$, we have that w.h.p.~$\text{val}(S^*, \pi^*) > \max_{(\Tilde{S}, \Tilde{\pi})\in \mathcal{B}_t} \text{val}(\Tilde{S}, \Tilde{\pi})$, then we know that the maximizer is not in $\mathcal{B}_t$ and may use this to show that the maximizer achieves strong recovery of $(S, \pi)$. From \eqref{ineq:S*-obj} and \eqref{ineq:Bt-obj}, it is enough to have
    \begin{align*}
        ktpq \gg \Tilde{O}\left(k^{\frac{3}{2}}\sqrt{p} + k + k^{\frac{3}{2}}pq + k\sqrt{p}\right),
    \end{align*}
    which holds for some $t = o(k)$ if $q = \Tilde{\Omega}(\frac{1}{\sqrt{kp}})$.

\subsubsection{Lower Bound}
\label{sec:log-density-stat-rec-lower}

We will show that if $q = o(\frac{1}{\sqrt{kp}})$, then weak support recovery is impossible.
Recall that under $\mathcal{P}$, a directed graph is generated by first drawing a random set $S$, then drawing a permutation $\pi \in \Sym([n])$, and finally generating the directed edges according to $S$ and $\pi$ (or equivalently $S$ and $\pi_S$, which is the restriction of $\pi$ on $S$).

    We have for any estimator $A: \{0, \pm 1\}^{n \times n} \to \{0,1\}^n$ of the indicator vector of $S$, that
    \[ \mathbb{E}[d_{\Ham}(A(Y), S)]\\
        = \sum_{i=1}^n \mathbb{P}[A(Y)_i \ne \theta_i], \]
    where $\theta \in \{0,1\}^n$ is the indicator vector of $S$. We will prove a lower bound for the first term in the sum above, which will hold for the rest of the terms. We may expand the first term into
    \[ \mathbb{P}[A(Y)_1 \ne \theta_1]
        = \sum_{i=0}^{n-1} \mathbb{P}\left[|S - \{1\}| = i\right] \Ex_{\substack{S' \subseteq [n] \setminus \{1\} \\
        |S'| = i}} \Ex_{\pi' \in \Sym([n])} \mathbb{P}[A(Y)_1 \ne \theta_1 \vert S - \{1\} = S', \pi = \pi']. \]
    Fix $S' \subseteq [n] \setminus \{1\}$ such that $|S'| = i$, and $\pi' \in \Sym([n])$. We have
    \begin{align*}
        &\mathbb{P}[A(Y)_1 \ne \theta_1 \vert S - \{1\} = S', \pi = \pi']\\
        &= \mathbb{P}[A(Y)_1 = 1, \theta_1 = 0 \vert S - \{1\} = S', \pi = \pi'] + \mathbb{P}[A(Y)_1 = 0, \theta_1 = 1 \vert S - \{1\} = S', \pi = \pi'].
    \end{align*}
    Let us define $\mathcal{H}_0 \colonequals \mathcal{P}_{S = S', \pi = \pi'}$ and $\mathcal{H}_1 \colonequals \mathcal{P}_{S = \{1\} \cup S', \pi = \pi'}$, by fixing the planted set to be either $S'$ or $\{1\} \cup S'$, and fixing the underlying permutation to be $\pi'$. Then we may rewrite
    \begin{align*}
        &\mathbb{P}[A(Y)_1 \ne \theta_1 \vert S - \{1\} = S', \pi = \pi']\\
        &= \mathbb{P}[A(Y)_1 = 1, \theta_1 = 0 \vert S - \{1\} = S', \pi = \pi'] + \mathbb{P}[A(Y)_1 = 0, \theta_1 = 1 \vert S - \{1\} = S', \pi = \pi']\\
        &= \mathbb{P}[\theta_1 = 0]\cdot \Px_{Y \sim \mathcal{H}_0}[A(Y)_1 = 1] + \mathbb{P}[\theta_1 = 1]\cdot \Px_{Y \sim \mathcal{H}_1}[A(Y)_1 = 0]\\
        &= \left(1 - \frac{k}{n} \right)\cdot\Px_{Y \sim \mathcal{H}_0}[A(Y)_1 = 0] + \left(\frac{k}{n} \right)\cdot \Px_{Y \sim \mathcal{H}_1}[A(Y)_1 = 1]\\
        &\ge  \frac{k}{n} \cdot \left(\Px_{Y \sim \mathcal{H}_0}[A(Y)_1 = 1] + \Px_{Y \sim \mathcal{H}_1}[A(Y)_1 = 0]\right),
        \intertext{which relates this quantity to the sum of Type~I and Type~II errors for hypothesis testing using $Y \mapsto A(Y)_1$ between $\mathcal{H}_0$ and $\mathcal{H}_1$. As a consequence of the Neyman-Pearson Lemma, this quantity is further bounded by}
        &\geq \frac{k}{n} \cdot \left(1 - d_{\text{TV}}(\mathcal{H}_0, \mathcal{H}_1)\right).
    \end{align*}

    We will now bound the TV distance between $\mathcal{H}_0$ and $\mathcal{H}_1$ using the KL divergence. Since $\mathcal{H}_0$ and $\mathcal{H}_1$ are product distributions over the entries of $Y \in \{0, \pm 1\}^{n \times n}$, by tensorization of KL divergence we have
    \begin{align*}
        d_{\text{KL}}\left(\mathcal{H}_1 \dbar \mathcal{H}_0\right) &= \sum_{i<j} d_{\text{KL}}\left(\mathcal{L}_{\mathcal{H}_1}(Y_{i,j})  \dbar \mathcal{L}_{\mathcal{H}_0}(Y_{i,j})\right)\\
        &\le |S'|\cdot d_{\text{KL}}\left( \text{SparseRad}(p, 1/2 + q)  \dbar \text{SparseRad}(p, 1/2) \right)
        \intertext{since the law of $Y_{i,j}$ differs under $\mathcal{H}_1$ and $\mathcal{H}_0$ only when $i=1$ and $j \in S'$, in which situation $Y_{i,j} \sim \text{SparseRad}(p, 1/2+q)$ or $\text{SparseRad}(p, 1/2-q)$ in $\mathcal{H}_1$ and $Y_{i,j} \sim \text{SparseRad}(p, 1/2)$ in $\mathcal{H}_0$. We may then continue with an explicit calculation of these divergences,}
        &= i \cdot \left(\left(\frac{1}{2} + q\right)p\cdot\log\left(1 + 2q\right) + \left(\frac{1}{2} - q\right)p\cdot \log\left(1 - 2q\right)\right) \\
        &\le i \cdot 4pq^2.
    \end{align*}

    By Pinsker's inequality,
    \begin{align*}
        \mathbb{P}[A(Y)_1 \ne \theta_1 \vert S - \{1\} = S', \pi = \pi']
        &\ge \frac{k}{n} \cdot \left(1 - d_{\text{TV}}(\mathcal{H}_0, \mathcal{H}_1)\right)\\
        &\ge \frac{k}{n} \cdot \left(1 - \sqrt{\frac{1}{2}d_{\text{KL}}(\mathcal{H}_1 \| \mathcal{H}_0)}\right)\\
        &\ge \frac{k}{n} \cdot \left(1 - \sqrt{2\cdot i pq^2}\right).
    \end{align*}

    Finally, we note that the distribution of $|S - \{1\}|$ follows $\text{Bin}(n-1, k/n)$, and by Chernoff bound $i = |S - \{1\}| \le 2k$ w.h.p. Therefore,
    \begin{align*}
        \mathbb{P}[A(Y)_1 \ne \theta_1]
        &= \sum_{i=0}^{n-1} \mathbb{P}\left[|S - \{1\}| = i\right] \Ex_{\substack{S' \subseteq [n] - \{1\}:\\
        |S'| = i}} \Ex_{\pi' \in S_n} \mathbb{P}[A(Y)_1 \ne \theta_1 \vert S - \{1\} = S', \pi = \pi']\\
        &\ge (1 - o(1))\cdot \frac{k}{n} \cdot \left(1 - \sqrt{2\cdot 2k pq^2}\right),
    \end{align*}
    which is at least $(1 - o(1))k/n$ when $q \ll \frac{1}{\sqrt{kp}}$. Thus, when $q \ll \frac{1}{\sqrt{kp}}$, the expected Hamming distance from $S$ achieved by any algorithm $A: \{0, \pm 1\}^{n \times n} \to \{0,1\}^n$ is at least
    \begin{align*}
        \mathbb{E}[d_H(A(Y), S)] = \sum_{i=1}^n \mathbb{P}[A(Y)_i \ne \theta_i] \ge (1-o(1))k.
    \end{align*}

    \begin{remark}
        \label{rem:strong-recovery-impossible}
        As an aside, we may also treat the boundary case $q = O(\frac{1}{\sqrt{kp}})$ with the same calculations, in which case it follows from the Bretagnolle-Huber inequality (in place of Pinsker's inequality) that the expected Hamming distance achieved by any algorithm is $\Omega(k)$.
    \end{remark}

\section{Proofs for Extreme Parameter Scalings}

\subsection{Planted Global Ranking}

\subsubsection{Detection: Proof of Theorem~\ref{thm:detection-th}}

Some of the results of this Theorem follow from our proofs in the log-density setting, which sometimes do not actually require that the parameters $\alpha, \beta, \gamma$ be bounded away from 0 and 1.
In particular, that thresholding the degree~2 polynomial in $Y$ given in \eqref{eq:deg-2-test} achieves strong detection when $q = \omega(n^{-3/4})$ follows from the calculations in Section~\ref{sec:log-density-comp-det-upper}, which only requires the condition \eqref{eq:log-density-comp-det-condition}.
Likewise, that weak detection is statistically impossible when $q = o(n^{-3/4})$ follows from the argument in Section~\ref{sec:log-density-stat-det-lower}.

The other two results are slightly more delicate than what we have argued in the log-density setting and require revisiting our arguments.
Recall what they say:
\begin{enumerate}
    \item If $q = O(n^{-3/4})$, then strong detection is statistically impossible.
    \item For a constant $c > 0$, if $q \geq c \cdot n^{-3/4}$, then there exists a polynomial-time algorithm that achieves weak detection.
\end{enumerate}

For the first point, note that repeating the calculations given in Section~\ref{sec:log-density-stat-det-lower} with $k = k' = n$ shows that, if $q = O(n^{-3/4})$, then the $\chi^2$ divergence of the two models is bounded as $\chi^2(\sP \dbar \sQ) = O(1)$.
When the $\chi^2$ divergence is bounded then $\sP$ is \emph{contiguous} to $\sQ$, which in particular implies that strong detection is impossible.
See, e.g., See, e.g., \cite[Lemma 1.13]{kunisky2019notes} for details of the simple argument for this.

For the second point, recall that in Section~\ref{sec:log-density-comp-det-upper}, we argued, if $q = \omega(n^{-3/4})$, then the polynomial $f(Y)$ defined in \eqref{eq:deg-2-test} satisfies the \emph{strong separation} condition
    \begin{align*}
        \max\{\sqrt{\Var_{\mathcal{P} }[f(Y)]}, \sqrt{\Var_{\mathcal{Q}}[f(Y)]} \} = o\left(\mathbb{E}_{\mathcal{P} }[f(Y)] - \mathbb{E}_{\mathcal{Q} }[f(Y)] \right).
    \end{align*}
It then followed by Chebyshev's inequality that thresholding the value of $f(Y)$ achieves strong detection.
In this case, by the same calculations, for sufficiently large $c$, we will have the quantitative weaker version of the above that
\begin{align*}
        (2\sqrt{2} + \varepsilon)\cdot \max\{\sqrt{\Var_{\mathcal{P} }[f(Y)]}, \sqrt{\Var_{\mathcal{Q}}[f(Y)]} \} \le \mathbb{E}_{\mathcal{P} }[f(Y)] - \mathbb{E}_{\mathcal{Q} }[f(Y)]
    \end{align*}
for some constant $\varepsilon > 0$.
In that case, thresholding the value of $f(Y)$ achieves weak detection between $\mathcal{P}$ and $\mathcal{Q}$.

\subsubsection{Suboptimality of Spectral Detection: Proof of Theorem~\ref{thm:spectral}}

We now show that a natural spectral algorithm for detection performs much worse than the algorithms described above.
Recall that we work with $\eye Y$, which is a Hermitian matrix and thus has real eigenvalues.
Those real eigenvalues are related to the singular values of $Y$: the singular values of $Y$ are the absolute values of the eigenvalues of $\eye Y$.
They are also related to the eigenvalues of $Y$: those eigenvalues were purely imaginary and came in conjugate pairs, and if $\eye \lambda$ and $-\eye \lambda$ are eigenvalues of $Y$ then $\lambda$ and $-\lambda$ will be eigenvalues of $\eye Y$.

We will consider a detection algorithm that computes and thresholds $\lambda_{\max}(\eye Y)$ (which is the same as the spectral radius or operator norm of either $Y$ or $\eye Y$ by the above observation on the symmetry of the eigenvalues).
To carry out this analysis, we must understand this largest eigenvalue when $Y \sim \sQ$ (Part 1 of Theorem~\ref{thm:spectral}) and when $Y \sim \sP$ (Parts 2 and 3 of Theorem~\ref{thm:spectral}).

The former case is straightforward, since then $Y$ is just a Hermitian Wigner matrix of i.i.d.\ (albeit complex) entries.
\begin{proof}[Proof of Part 1 of Theorem~\ref{thm:spectral}]
Recall that we want to show that, when $Y \sim \sQ$, then
    \begin{equation}
        \frac{1}{\sqrt{n}} \lambda_{\max}(\eye Y) \,\,\to\,\, 2,
    \end{equation}
    where the convergence is in probability.
    The entries of $\eye Y$ above the diagonal are i.i.d.\ centered complex random variables whose modulus is always equal to 1.
    The result then follows by standard analysis of Wigner matrices; see, e.g., Sections 2.1.6 and 2.2 of \cite{AGZ-2010-RandomMatrices}.
\end{proof}

The latter case $Y \sim \sP$ is more complicated.
Morally, it is similar to a \emph{spiked matrix model}, a low-rank additive perturbation of a Wigner matrix.
We will eventually appeal to the analysis of \cite{CDMF-2009-DeformedWigner} of general such models for perturbations of constant rank.
Yet, we will see that in this case the perturbation is actually not quite low-rank but rather has full rank with rapidly decaying eigenvalues, which must be treated more carefully.

This perturbation will correspond to the expectation of $Y \sim \sP$, which we now compute along with its eigenvalues.
Note without loss of generality that we may restrict all of our discussion to the case where the hidden permutation $\pi$ is the identity; we will condition on this event without further discussion going forward.
To be explicit about this, we will replace $\sP$ with $\sP_{\mathrm{id}}$ when we mention it.

Let $A \in \mathbb{C}^{n \times n}$ be defined as
    \begin{align}
        A_{i,j} = \begin{cases}
            \eye & \quad \text{ if } i < j, \\
            -\eye & \quad \text{ if } i > j, \\
            0 & \quad \text{ if } i = j.
        \end{cases}
    \end{align}
In words, $A$ is a Hermitian matrix with $\eye$ in the upper diagonal entries, $-\eye$ in the lower diagonal entries, and $0$ on the diagonal.

By the same reasoning as in Section~\ref{sec:log-density-comp-rec}, we have
\[ \Ex_{Y \sim \sP_{\mathrm{id}}}[\eye Y] = 2 q A, \]
and moreover Corollary~\ref{cor:expectation-spectral} describes the eigenvalues of this expectation:
        \begin{align*}
            \lambda_i(A) = \frac{1}{\tan\left(\frac{2i-1}{2n}\pi\right)} \,\,\text{ for }\,\, 1\le i\le n
        \end{align*}
        with the corresponding eigenvectors $v^{(i)} \in \mathbb{C}^n$ whose entries are given by
        \begin{align*}
            (v^{(i)})_j = \exp\left(-\eye \pi \frac{(2i-1)j}{n}\right) \,\,\text{ for }\,\, 1\le j\le n.
        \end{align*}

We may now give the analysis of the spectral algorithm on the planted model.
The interpretation of this result is that, for $q = c \cdot n^{-1/2}$, for $c$ large enough there is at least one outlier eigenvalue greater than the typical largest eigenvalue under the null model, while for $c$ smaller the largest eigenvalue is the same as that of the null model.

\begin{proof}[Proof of Parts 2 and 3 of Theorem~\ref{thm:spectral}]
    Recall the statement of these results: when $Y \sim \sP$ and $q = c \cdot n^{-1/2}$ for a constant $c > 0$, we want to show that:
    \begin{enumerate}
        \item If $c \leq \pi / 4$, then
        \begin{equation}
            \frac{1}{\sqrt{n}} \lambda_{\max}(\eye Y) \to 2
        \end{equation}
        in probability.
        \item If $c > \pi / 4$, then
        \begin{equation}
            \frac{1}{\sqrt{n}} \lambda_{\max}(\eye Y) \geq 2 + f(c)
        \end{equation}
        for some $f(c) > 0$ with high probability.
    \end{enumerate}
    We may also assume without loss of generality that instead $Y \sim \sP_{\mathrm{id}}$.
    Let us write
    \begin{equation}
        \eye Y = \EE \eye Y + (\eye Y - \EE \eye Y) = 2q A + (\eye Y - \EE \eye Y).
    \end{equation}
    We have that $W \colonequals \eye Y - \EE \eye Y$ is a Hermitian Wigner matrix whose entries above the diagonal are i.i.d.\ with the law of $\eye (X - 2q) = \eye (X - \EE X)$ for $X \sim \mathrm{Rad}(\frac{1}{2} + q)$.
    In particular, these entries are bounded, centered, and have complex variance $\EE|\eye (X - 2q)|^2 = \Var(X) = 1 - 4q^2 = 1 - O(1 / n)$ with our scaling of $q$.

    We emphasize a small nuance: that $\eye Y - \EE \eye Y$ is exactly a Wigner matrix depends on our having assumed that the hidden permutation $\pi$ is the identity, since otherwise the entries above the diagonal would not be identically distributed.
    Yet, since conjugating a matrix by a permutation does not change the spectrum, the assumption that $\pi$ is any fixed permutation is without loss of generality; the choice of the identity permutation is just uniquely amenable to fitting into the existing theory around Wigner random matrices and spiked matrix models.

    Fix a large $k \in \NN$ not depending on $n$, to be chosen later.
    Let $\hat{v}_i \colonequals v_i / \|v_i\|$ for $v_i$ the eigenvectors of $A$ given above, so that $A = \sum_{i = 1}^n \lambda_i \hat{v}_i\hat{v}_i^{*}$.
    Let us write $A = A_1 + A_2$, where
    \begin{align}
        A_1 &\colonequals \sum_{i = 1}^k \lambda_i \hat{v}_i\hat{v}_i^{*} + \sum_{i = n - k + 1}^n \lambda_i \hat{v}_i \hat{v_i}^*, \\
        A_2 &\colonequals \sum_{i = k + 1}^{n - k} \lambda_i \hat{v}_i \hat{v_i}^*.
    \end{align}
    We have, since $\tan(x) \geq x / 2$ for sufficiently small $x$, that, for sufficiently large $n$,
    \begin{equation}
        \|A_2\| = \lambda_{k + 1} = \frac{1}{\tan(\frac{2k + 1}{2n} \pi)} \leq \frac{4n}{(2k + 1)\pi} = O\left(\frac{n}{k}\right).
    \end{equation}
    Also, the eigenvalues of $A_1$ satisfy, as $n \to \infty$, the convergences
    \begin{align}
        \frac{\lambda_a}{n} &= \frac{1}{n \tan(\frac{2a - 1}{2n}\pi)} \,\,\to\,\, \frac{2}{(2a - 1)\pi}, \\
        \frac{\lambda_{n - a}}{n} &= -\frac{\lambda_a}{n} \,\,\to\,\, -\frac{2}{(2a - 1)\pi}
    \end{align}
    for any $a$ fixed as $n \to \infty$.
    Thus we may define
    \begin{equation}
        A_0 \colonequals \sum_{i = 1}^k \frac{2n}{(2i - 1)\pi} \hat{v}_i\hat{v}_i^{*} - \sum_{i = 1}^k \frac{2n}{(2i - 1)\pi} \hat{v}_{n - i + 1} \hat{v}_{n - i + 1}^*,
    \end{equation}
    and this will satisfy $\|A_1 - A_0\| = o(n)$ as $n \to \infty$.

    We have
    \begin{equation}
        \label{eq:iT-decomp}
        \frac{1}{\sqrt{n}}\eye Y = \left(\frac{2c}{n} A_0 + \frac{1}{\sqrt{n}}W\right) + \frac{2c}{n}(A_1 - A_0) + \frac{2c}{n} A_2.
    \end{equation}
    The rank of $A_0$ is $2k$, a constant as $n \to \infty$, the non-zero eigenvalues of $\frac{2c}{n}A_0$ are independent of $n$, and $W$ is a Wigner matrix with entrywise complex variance converging to 1.
    Thus, we may apply Theorem 2.1 of \cite{CDMF-2009-DeformedWigner} to the first term of \eqref{eq:iT-decomp}, which is a finite-rank additive perturbation of a Wigner matrix.
    That result implies that the largest eigenvalue of the first term converges to 2 in probability if no eigenvalue of $\frac{2c}{n}A_0$ is greater than 1, and converges to some $2 + f(c, k)$ with $f(c, k) > 0$ in probability if some eigenvalue of $\frac{2c}{n}A_0$ is greater than 1.
    We have $\lambda_{\max}(\frac{2c}{n}A_0) = \frac{4}{\pi}c$, so the latter condition is equivalently $c > \frac{\pi}{4}$.
    Moreover, since $\lambda_{\max}(\frac{2c}{n}A_0)$ depends only on $c$ and not on $k$, the details of the result of \cite{CDMF-2009-DeformedWigner} imply that $f(c, k) = f(c)$ depends only on $c$ and not on $k$ as well.
    Let us define this value as
    \begin{equation}
        \Lambda = \Lambda(c) \colonequals \begin{cases} 2 & \text{if } c \leq \frac{\pi}{4}, \\ 2 + f(c) & \text{if } c > \frac{\pi}{4}. \end{cases}
    \end{equation}

    It remains to show that the other two terms of \eqref{eq:iT-decomp} do not change this behavior substantially.
    We have $\|\frac{2c}{n}(A_1 - A_0)\| = o(1)$ and $\|\frac{2c}{n}A_2\| = O(\frac{1}{k})$, so by taking $k$ and $n$ sufficiently large we may make $\|\frac{2c}{n}(A_1 - A_0)\| + \|\frac{2c}{n}A_2\|$ smaller than any given $\eta > 0$ not depending on $n$.
    By the Weyl eigenvalue inequality, we then have $|\lambda_{\max}(\frac{1}{\sqrt{n}} \eye Y) - \Lambda| \leq 2\eta$ with high probability.
    Since this is true for any $\eta > 0$, we also have that $\lambda_{\max}(\frac{1}{\sqrt{n}} \eye Y) \to \Lambda$ in probability, and the result follows.
\end{proof}

\subsubsection{Recovery: Proof of Theorem~\ref{thm:recovery-th}}
\label{sec:pf:thm:recovery-th}

\paragraph{Upper Bounds}
For the two upper bound results, we analyze the Ranking By Wins algorithm described in Definition~\ref{def:rank-by-wins}.
Recall what we want to show:
\begin{enumerate}
    \item If $q = \omega(n^{-1/2})$, then Ranking By Wins achieves strong recovery.
    \item If $q = \Theta(n^{-1/2})$, then Ranking By Wins achieves weak recovery.
\end{enumerate}

Recall that, in the Ranking By Wins algorithm, a permutation estimator $\hat{\pi}$ is created by ranking according to the scores $s_i = \sum_{k \in [n]} Y_{i,k}$.
Without loss of generality, we may assume that the hidden permutation $\pi$ is the identity.
Note that $\hat{\pi}$ is only well-defined up to the tie-breaking rule.
We will get rid of the technicality of the tie-breaking rules by employing a pessimistic view that the algorithm breaks ties in the worst possible way, i.e., that for every $i < j$, if $s_i = s_j$, then the algorithm ranks the pair $i,j$ incorrectly.
Thus, this upper bounds the Kendall tau distance $d_\KT(\pi, \hat{\pi})$ (achieved with by the Ranking By Wins algorithm with any tie-breaking rule) by a function $f(Y)$ defined as
\begin{align*}
    f(Y) \colonequals \sum_{i < j} \One\{s_i \le s_j\}.
\end{align*}
Our strategy is then to bound the expectation of $f(Y)$ and show that it concentrates around this expectation.

\begin{proof}[Proof of Parts 1 and 3 of Theorem~\ref{thm:recovery-th}]
    While the random variables $\One\{s_i > s_j\}$ are not independent, they are only weakly dependent and form a read-$(2n)$ family. By Theorem \ref{thm:tail-read-k}, we get
    \begin{align*}
        \mathbb{P}\left[f(Y) \ge \mathbb{E}[f(Y)] + t\right] &\le \exp\left(-\frac{2t^2 }{2n \binom{n}{2}}\right),
    \end{align*}
    for any $t\ge 0$. Thus, we see that, with high probability,
    \begin{align}
        f(Y) \le \mathbb{E}[f(Y)]  + t(n) \label{ineq:recovery-tail-bd}
    \end{align}
    for any function $t(n) = \omega(n^{3/2})$.
    Now it remains to upper bound $\mathbb{E}[f(Y)]$. We have
    \begin{align*}
        \mathbb{E}[f(Y)]&= \sum_{i < j} \mathbb{P}[s_i \le s_j].
    \end{align*}
    For every pair $i < j$, we have
    \begin{align*}
        s_i - s_j &= \sum_{k \in [n]} Y_{i,k} - \sum_{k \in [n]} Y_{j,k}\\
        &\stackrel{\text{(d)}}{=} \sum_{t = 1}^{n-i+j-3} X_t + \sum_{t=1}^{n+i-j-1} Y_t + Z,
    \end{align*}
    where $X_t \stackrel{\text{iid}}{\sim} \text{Rad}\left(\frac{1}{2} + q\right)$, $Y_t \stackrel{\text{iid}}{\sim} \text{Rad}\left(\frac{1}{2} - q\right)$, and $Z \sim 2\cdot \text{Rad}\left(\frac{1}{2} + q\right)$.

    Recall that we have assumed $q \le \frac{1}{4}$.
    With this assumption, we have
    \begin{align*}
        \Var(X_t) = \Var(Y_t) = \frac{1}{4}\Var(Z) &= 1 - 4q^2 \ge \frac{3}{4},\\
        \mathbb{E}[|X_t - \mathbb{E}[X_t]|^3] = \mathbb{E}[|Y_t - \mathbb{E}[Y_t]|^3] = \frac{1}{8}\mathbb{E}[|Z - \mathbb{E}[Z]|^3] &=1 - 16q^4,\\
        \mathbb{E}\left[\sum_{t = 1}^{n-i+j-3} X_t + \sum_{t=1}^{n+i-j-1} Y_t + Z\right] &= 4(j-i)q,\\
        \Var\left(\sum_{t = 1}^{n-i+j-3} X_t + \sum_{t=1}^{n+i-j-1} Y_t + Z\right) &= 2(1-4q^2)n.
    \end{align*}
    Then, we apply Theorem \ref{thm:berry-esseen} and obtain that
    \begin{align*}
        &\mathbb{P}\left[s_i \le s_j\right]\\
        &= \mathbb{P}\left[\sum_{t = 1}^{n-i+j-3} X_t + \sum_{t=1}^{n+i-j-1} Y_t + Z \le 0\right]\\
        &= \mathbb{P}\left[\frac{\sum_{t = 1}^{n-i+j-3} X_t + \sum_{t=1}^{n+i-j-1} Y_t + Z - 4(j-i)q}{\sqrt{2(1-4q^2)n}} \le -\frac{4(j-i)q}{\sqrt{2(1-4q^2)n}}\right] \\
        &\le \Phi\left(- \frac{4(j-i)q}{\sqrt{2(1-4q^2)n}}\right) + C \cdot \frac{1}{\sqrt{n}}, \numberthis \label{ineq:berry-esseen-bd}
    \end{align*}
    for some absolute constant $C > 0$.

    Plugging this bound back to $\mathbb{E}[f(Y)]$, we obtain
    \begin{align*}
    \mathbb{E}[f(Y)] &\le \sum_{i < j} \left(\Phi\left(- \frac{4(j-i)q}{\sqrt{2(1-4q^2)n}}\right) + C \cdot\frac{1}{\sqrt{n}}\right)\\
    &= \sum_{i<j} \Phi\left(- \frac{4(j-i)q}{\sqrt{2(1-4q^2)n}}\right) + O\left(n^{\frac{3}{2}}\right)\\
    &= \sum_{d=1}^{n-1}(n-d)\Phi\left(- \frac{4dq}{\sqrt{2(1-4q^2)n}}\right) + O\left(n^{\frac{3}{2}}\right). \numberthis \label{ineq:recovery-expectation-ub}
\end{align*}
By Lemma \ref{lem:concavity-Phi-expr}, we may bound the sum above as
\begin{align*}
    &\sum_{d=1}^{n-1}(n-d)\Phi\left(- \frac{4dq}{\sqrt{2(1-4q^2)n}}\right)\\
    &= n\sum_{d=1}^{n-1}\left(1 - \frac{d}{n}\right)\Phi\left(-\frac{4q\sqrt{n}\cdot \frac{d}{n}}{\sqrt{2(1-4q^2)}}\right) \numberthis \label{eq:riemann-sum}
    \intertext{now, we use that $(1-y)\Phi
\left(-\frac{4q\sqrt{n}\cdot y}{\sqrt{2(1-4q^2)}}\right)$ is concave for $y\in [0,1]$, and we apply Jensen's inequality to get}
    &\le \binom{n}{2}\Phi\left(-\frac{2q\sqrt{n}}{\sqrt{2(1-4q^2)}}\right). \numberthis \label{ineq:sum-Phi-bd}
\end{align*}
Plugging it back to \eqref{ineq:recovery-expectation-ub}, we obtain
\begin{align}
    \mathbb{E}[f(Y)] \le  \binom{n}{2}\Phi\left(-\frac{2q\sqrt{n}}{\sqrt{2(1-4q^2)}}\right) + O\left(n^{\frac{3}{2}}\right).
\end{align}

Together with \eqref{ineq:recovery-tail-bd}, for any $t(n) = \omega(n^{3/2})$, the permutation $\hat{\pi}$ output by the Ranking By Wins algorithm with high probability achieves a Kendall tau distance from the hidden permutation of at most
\begin{align}
    d_\KT(\pi, \hat{\pi}) \le f(Y) \le \binom{n}{2}\Phi\left(-\frac{2q\sqrt{n}}{\sqrt{2(1-4q^2)}}\right) + t(n).
\end{align}

Qualitatively, this means that for any constant $c> 0$ and $q = c \cdot n^{-1/2}$, there exists some constant $\delta = \delta(c) > 0$ such that the Ranking By Wins algorithm outputs a permutation that w.h.p.~achieves a Kendall tau distance of at most $\left(\frac{1}{2} - \delta \right) \binom{n}{2}$ from the hidden permutation, thus achieving weak recovery.
Moreover, as $c \to \infty$, this value satisfies $\delta = \delta(c) \to \frac{1}{2}$, whereby for any $q = \omega(n^{-1/2})$, the Ranking By Wins algorithm achieves strong recovery of the hidden permutation.
\end{proof}

Let us remark on some more precise statements about the strong recovery regime.
Quantitatively, we may bound the cdf of the standard normal distribution by Mill's inequality \cite{gordon1941values} and get
\begin{align*}
    \binom{n}{2}\Phi\left(-\frac{2q\sqrt{n}}{\sqrt{2(1-4q^2)}}\right) \le \binom{n}{2}   \frac{C}{\frac{2q \sqrt{n}}{\sqrt{2(1-4q^2)}}} \cdot e^{- \frac{q^2 n}{1-4q^2}} \le \binom{n}{2} \frac{C}{q\sqrt{n}} \cdot e^{-q^2 n},
\end{align*}
for some constant $C > 0$. Thus, the Ranking By Wins algorithm with high probability achieves a Kendall tau distance from the hidden permutation of $O(\frac{n^{3/2}}{q}\cdot e^{-q^2 n}) + t(n)$ for any $t(n) = \omega(n^{3/2})$.
So, our positive result for the Ranking By Wins algorithm can prove at best an upper bound scaling slightly faster than $n^{3/2}$ on the Kendall tau error.
It seems likely that this could be improved, since we are pessimistically assuming that the errors from the Berry-Esseen bound used accumulate additively in \eqref{ineq:berry-esseen-bd}, while probably in reality they themselves enjoy some cancellations.

In the weak recovery regime, based on the intermediate Riemann sum result in \eqref{eq:riemann-sum}, one may also show the more precise asymptotic that, for $q = c \cdot n^{-1/2}$, we have
\begin{equation}
    \frac{1}{n^2}\EE[f(Y)] = \int_0^1 (1 - y) \Phi\left(-2\sqrt{2} \, c y\right)dy + o(1),
\end{equation}
which, with our remaining calculations, describes the precise amount of error incurred by the Ranking By Wins algorithm to leading order, with high probability.

\paragraph{Lower Bounds}
Now we turn to the impossibility results in Theorem~\ref{thm:recovery-th}.
\begin{proof}[Proof of Parts 2 and 4 of Theorem~\ref{thm:recovery-th}]
    Consider any recovery algorithm, given by a function $A: \{\pm 1\}^{\binom{n}{2}} \to \Sym([n])$.
    Here all expectations and probabilities will be over $Y \in \sP$.
    In expectation, $A$ achieves
    \begin{align*}
        \mathbb{E}[d_{\KT}(A(Y), \pi)] = \sum_{i<j} \mathbb{E}\left[\One\{A(Y)(i,j)\ne \pi(i,j)\}\right] = \sum_{i<j}\mathbb{P}\left[A(Y)(i,j)\ne \pi(i,j)\right]. \numberthis \label{eq:kendall-tau-expectation}
    \end{align*}
    Thus, if for every $i < j$ we can show that $\mathbb{P}[A(Y)(i,j)\ne \pi(i,j)]$ is bounded away from $0$ (or even close to $\frac{1}{2}$) for any function $A$, we will get an information-theoretic lower bound for the recovery problem.

    Let us rewrite
    \begin{equation}
        \Px_{(Y, \pi) \sim \mathcal{P}}\left[A(Y)(i,j)\ne \pi(i,j)\right] = \frac{1}{n!} \sum_{\pi \in \Sym([n])} \Px_{Y \sim \mathcal{P}_{\pi}}\left[A(Y)(i,j)\ne \pi(i,j)\right],
    \end{equation}
    where $\mathcal{P}_{\pi}$ denotes the distribution $\sP$ conditional on the hidden permutation being $\pi$. For every $\pi \in \Sym([n])$, let $\pi^{\{i,j\}} \in \Sym([n])$ denote the permutation obtained from $\pi$ by swapping $\pi(i)$ and $\pi(j)$. Then,
    \begin{align*}
        &\frac{1}{n!} \sum_{\pi \in \Sym([n])} \Px_{Y \sim \mathcal{P}_{\pi}}\left[A(Y)(i,j)\ne \pi(i,j)\right]\\
        &= \frac{1}{2\cdot n!} \sum_{\pi \in \Sym([n])} \left(\Px_{Y \sim \mathcal{P}_{\pi}}\left[A(Y)(i,j)\ne \pi(i,j)\right] + \Px_{Y \sim \mathcal{P}_{\pi^{\{i,j\}}}}\left[A(Y)(i,j)\ne \pi^{\{i,j\}}(i,j)\right]\right). \numberthis \label{eq:sum-of-testing-errors}
    \end{align*}
    Our plan is then to show that each summand is small.

    For every fixed $\pi \in \Sym([n])$, we note that the sum
    \[\Px_{Y \sim \mathcal{P}_{\pi}}\left[A(Y)(i,j)\ne \pi(i,j)\right] + \Px_{Y \sim \mathcal{P}_{\pi^{\{i,j\}}}}\left[A(Y)(i,j)\ne \pi^{\{i,j\}}(i,j)\right]\]
    may be viewed as the sum of Type-I and Type-II errors of the test statistic $Y \mapsto A(Y)(i,j) \in \{ \pm 1\}$ for distinguishing $\mathcal{P}_{\pi}$ and $\mathcal{P}_{\pi^{\{i,j\}}}$.
    As a consequence of the Neyman-Pearson Lemma, we have the lower bound
    \begin{align*}
        &\mathbb{P}_{Y \sim \mathcal{P}_{\pi}}\left[A(Y)(i,j)\ne \pi(i,j)\right] + \mathbb{P}_{Y \sim \mathcal{P}_{\pi^{\{i,j\}}}}\left[A(Y)(i,j)\ne \pi^{\{i,j\}}(i,j)\right] \ge 1 - d_{\text{TV}}\left(\mathcal{P}_{\pi}, \mathcal{P}_{\pi^{\{i,j\}}}\right).
    \end{align*}
    It remains to understand $d_{\text{TV}}(\mathcal{P}_{\pi}, \mathcal{P}_{\pi^{\{i,j\}}})$.

    We will bound the total variation distance by the KL divergence. Note that both distributions $\mathcal{P}_{\pi}$ and $\mathcal{P}_{\pi^{\{i,j\}}}$ are the product distributions of independent entries of the upper diagonal of the tournament matrix. Using the tensorization of KL divergence, we have
    \begin{align*}
        d_{\text{KL}}\left(\mathcal{P}_{\pi}, \mathcal{P}_{\pi^{\{i,j\}}}\right)
        &= \sum_{a<b} d_{\text{KL}}\left(\mathcal{P}_{\pi}(Y_{a,b}), \mathcal{P}_{\pi^{\{i,j\}}}(Y_{a,b})\right).
    \end{align*}
    We note that if $a,b \in [n] \setminus \{i,j\}$, then the associated distributions of $Y_{a,b}$ are identical for $\mathcal{P}_{\pi}$ and $\mathcal{P}_{\pi^{\{i,j\}}}$. Thus, the number of pairs of $(a,b)$ for which $d_{\text{KL}}\left(\mathcal{P}_{\pi}(Y_{a,b}), \mathcal{P}_{\pi^{\{i,j\}}}(Y_{a,b})\right)$ is nonzero is at most $2n$. Lastly, among those pairs of $(a,b)$ that have different distributions of $Y_{a,b}$ under $\mathcal{P}_{\pi}$ and $\mathcal{P}_{\pi^{\{i,j\}}}$, each has KL divergence
    \begin{align*}
        d_{\text{KL}}\left(\mathcal{P}_{\pi}(Y_{a,b}), \mathcal{P}_{\pi^{\{i,j\}}}(Y_{a,b})\right)
        &= d_{\text{KL}}\left(\text{Rad}\left(\frac{1}{2} + q\right), \text{Rad}\left(\frac{1}{2} - q\right)\right),
    \end{align*}
    where we use the symmetry $d_{\text{KL}}\left(\text{Rad}\left(\frac{1}{2} + q\right), \text{Rad}\left(\frac{1}{2} - q\right)\right) = d_{\text{KL}}\left(\text{Rad}\left(\frac{1}{2} - q\right), \text{Rad}\left(\frac{1}{2} + q\right)\right)$. It is an easy calculus computation to show that
    \begin{align*}
        \quad d_{\text{KL}}\left(\text{Rad}\left(\frac{1}{2} + q\right), \text{Rad}\left(\frac{1}{2} - q\right)\right)
        &= \left(\frac{1}{2} + q\right) \log \frac{\frac{1}{2} + q}{\frac{1}{2} - q} + \left(\frac{1}{2} - q\right) \log \frac{\frac{1}{2} - q}{\frac{1}{2} + q}
        \intertext{and since $\log(1 + x) \le x$, we have}
        &\le \left(\frac{1}{2} + q\right) \frac{2q}{\frac{1}{2}-q} - \left(\frac{1}{2} - q\right) \frac{2q}{\frac{1}{2}+q}\\
        &= \frac{4q^2}{\frac{1}{4} - q^2}.
    \end{align*}
    We conclude that we have the bound
    \begin{align*}
        d_{\text{KL}}\left(\mathcal{P}_{\pi}, \mathcal{P}_{\pi^{\{i,j\}}}\right) \le 2n \cdot \frac{4q^2}{\frac{1}{4} - q^2} = \frac{8 q^2 n}{\frac{1}{4} - q^2}.
    \end{align*}

    {\allowdisplaybreaks
    Finally, by the Pinsker and Bretagnolle-Huber inequalities (see \cite[Equation 2.25]{Tsybakov_2009} and \cite{canonne2022short}), we have
    \begin{align*}
        &\mathbb{P}_{Y \sim \mathcal{P}_{\pi}}\left[A(Y)(i,j)\ne \pi(i,j)\right] + \mathbb{P}_{Y \sim \mathcal{P}_{\pi^{\{i,j\}}}}\left[A(Y)(i,j)\ne \pi^{\{i,j\}}(i,j)\right]\\
        &\ge 1 - d_{\text{TV}}\left(\mathcal{P}_{\pi}, \mathcal{P}_{\pi^{\{i,j\}}}\right)\\
        &\ge 1 - \min\left\{\sqrt{2d_{\text{KL}}\left(\mathcal{P}_{\pi}, \mathcal{P}_{\pi^{\{i,j\}}}\right)}, \sqrt{1 - \exp(-d_{\text{KL}}\left(\mathcal{P}_{\pi}, \mathcal{P}_{\pi^{\{i,j\}}}\right))}\right\}\\
        &\ge 1 - \min\left\{\sqrt{\frac{16 q^2 n}{\frac{1}{4} - q^2}}, \sqrt{1 - \exp\left(-\frac{8 q^2 n}{\frac{1}{4} - q^2}\right)}\right\}\\
        &\ge \max\left\{1 - \frac{4q\sqrt{n}}{\sqrt{\frac{1}{4} - q^2}}, \frac{1}{2} \exp\left(-\frac{8 q^2 n}{\frac{1}{4} - q^2}\right) \right\}. \numberthis \label{ineq:KL-lb}
    \end{align*}
    Plugging \eqref{ineq:KL-lb} back into \eqref{eq:sum-of-testing-errors},
    \begin{align*}
        &\mathbb{P}_{(Y, \pi) \sim \mathcal{P}}
\left[A(Y)(i,j)\ne \pi(i,j)\right]\\
        &= \frac{1}{2n!} \sum_{\pi \in \Sym([n])} \left(\mathbb{P}_{Y \sim \mathcal{P}_{\pi}}\left[A(Y)(i,j)\ne \pi(i,j)\right] + \mathbb{P}_{Y \sim \mathcal{P}_{\pi^{\{i,j\}}}}\left[A(Y)(i,j)\ne \pi^{\{i,j\}}(i,j)\right]\right)\\
        &\ge \frac{1}{2}\max\left\{1 - \frac{4q\sqrt{n}}{\sqrt{\frac{1}{4} - q^2}}, \frac{1}{2} \exp\left(-\frac{8 q^2 n}{\frac{1}{4} - q^2}\right)\right\}.
    \end{align*}
    Substituting it into \eqref{eq:kendall-tau-expectation}, we then obtain
    \begin{align*}
        \mathbb{E}[d_\KT(A(Y), \pi)]
        &\ge \frac{1}{2}\binom{n}{2}\max\left\{1 - \frac{4q\sqrt{n}}{\sqrt{\frac{1}{4} - q^2}}, \frac{1}{2} \exp\left(-\frac{8 q^2 n}{\frac{1}{4} - q^2}\right)\right\}.
    \end{align*}}

    We therefore conclude that if $q = O(n^{-1/2})$, it is impossible to achieve strong recovery, and if $q = o(n^{-1/2})$, it is impossible to achieve weak recovery.
\end{proof}

\subsubsection{Alignment Maximization: Proof of Theorem~\ref{thm:alignment-approx}}

We now show that the Ranking By Wins algorithm also approximately maximizes the alignment objective.

Before proving Theorem~\ref{thm:alignment-approx}, we state a high probability bound on the maximum alignment objective when $Y$ is drawn from $\mathcal{P}$, to which we will compare the alignment objective achieved by Ranking By Wins.

\begin{proposition}\label{prop:opt-high-prob-bound}
    For a tournament $Y$ drawn from $\mathcal{P}$, the optimum alignment objective with high probability satisfies
    \begin{align*}
        2q \binom{n}{2} - \Tilde{O}(n) \le \max_{\hat{\pi} \in S_n} \, \aln(\hat{\pi}, Y) \le 2q \binom{n}{2} + O\left(n^{\frac{3}{2}}\right).
    \end{align*}
\end{proposition}

The proof of Proposition \ref{prop:opt-high-prob-bound} can be adapted from \cite{de1983maximum}, but we include here the main argument for the sake of completeness.

\begin{proof}[Proof of Proposition \ref{prop:opt-high-prob-bound}]
    Let us write $\mathrm{OPT} \colonequals \max_{\hat{\pi} \in S_n} \, \aln(\hat{\pi}, Y)$.
    First, we prove the high probability lower bound on $\text{OPT}$. In fact, by Chernoff bound, we may show that the hidden permutation $\pi$ would achieve this lower bound w.h.p., as
    \begin{align*}
        \mathbb{P}\left[\sum_{i<j} \pi(i,j)Y_{i,j} - \mathbb{E}\left[\sum_{i<j} \pi(i,j)Y_{i,j}\right] \le -\lambda\right] \le \exp\left(-\frac{\lambda^2}{4\binom{n}{2}}\right)
    \end{align*}
    since for each $i< j$, $\pi(i,j)Y_{i,j}$ is independently distributed as $\text{Rad}\left(\frac{1}{2} + q\right)$. Further note that the expectation is $2q\binom{n}{2}$. Plugging in $\lambda = C\cdot n\log n$, we get that w.h.p.~for $Y$ drawn from $\mathcal{P}$,
    \[\text{OPT} \ge \sum_{i < j}\pi(i,j)Y_{i,j} \ge 2q\binom{n}{2} - C\cdot n\log n.\]

    Next we turn to a high probability upper bound on $\text{OPT}$.
    Without loss of generality, we assume that the hidden permutation is the identity. For convenience, we furthermore assume that $n = 2^k$ is a power of $2$. In the end we will easily see how to extend the analysis to arbitrary values of $n$. Following the argument in \cite{de1983maximum}, for any permutation $\hat{\pi} \in S_n$, we partition the set of $\binom{n}{2}$ directed edges of a tournament on $n$ vertices according to $\hat{\pi}$ into
    \begin{align*}
        \bigsqcup_{\ell=1}^k B_{\ell},
    \end{align*}
    where $B_{\ell}$ consists of the pairs of indices that belong to
    \begin{align*}
        B_{\ell} &= \Bigg\{(x,y) \in [n]^2: \text{there exists } 0\le i \le 2^{\ell-1}-1 \text{ such that }2i \frac{n}{2^l} + 1\le \hat{\pi}(x) \le (2i+1) \frac{n}{2^{\ell}},\\
        &\hspace{3.1cm} (2i+1) \frac{n}{2^{\ell}} + 1\le \hat{\pi}(y) \le (2i+2) \frac{n}{2^{\ell}}\Bigg\}.
    \end{align*}

    Then, we may write
    \begin{align*}
        \sum_{i<j} \hat{\pi}(i,j) Y_{i,j} = \sum_{\ell=1}^k \sum_{(x,y) \in B_{\ell}} Y_{x,y},
    \end{align*}
    and
    \begin{align*}
        \max_{\hat{\pi}} \sum_{i<j} \hat{\pi}(i,j) Y_{i,j} \le \sum_{\ell=1}^k \max_{B_{\ell}} \sum_{(x,y) \in B_{\ell}} Y_{x,y}.
    \end{align*}

    We then follow the argument in \cite{de1983maximum} and use a union bound over all $B_{\ell}$ for each $\ell \in [k]$. For each fixed $B_{\ell}$, we observe that the distribution of $\sum_{(x,y)\in B_{\ell} } Y_{x,y}$ is stochastically dominated by that of a sum of $\frac{n}{2^{\ell}}$ independent Rademacher random variables $\text{Rad}\left(\frac{1}{2} + q\right)$. Recycling the calculations from \cite{de1983maximum}, we can show that w.h.p.~for $Y$ drawn from $\mathcal{P}$,
    \begin{align*}
        \max_{\hat{\pi}} \sum_{i<j} \hat{\pi}(i,j) Y_{i,j} \le \sum_{\ell=1}^k \max_{B_{\ell}} \sum_{(x,y) \in B_{\ell}} Y_{x,y} \le 2q \binom{n}{2} + C\cdot n^{\frac{3}{2}},
    \end{align*}
    for some constant $C > 0$. Finally, for the values of $n$ that are not a power of $2$, we use a similar decomposition of $\binom{[n]}{2}$ into $B_{\ell}$, but we no longer decompose $[n]$ into $2^l$ subsets of equal sizes. Despite this technicality, the same idea works here \emph{mutatis mutandis}.
\end{proof}

We remark that there is also a worst-case lower bound of $\Omega(n^{3/2})$ on the optimum alignment objective for an arbitrary tournament on $n$ vertices, which was shown in \cite{spencer1971optimal}. Moreover, it is shown that one can construct in polynomial time a solution that always achieves this lower bound asymptotically in \cite{poljak1988tournament}.

\begin{theorem}[\cite{spencer1971optimal, poljak1988tournament}]\label{thm:opt-worst-case-lb}
    There exists a constant $C > 0$, such that for any tournament $Y$ on $n$ vertices for all sufficiently large $n$,
    \begin{align}
        \max_{\hat{\pi} \in S_n} \, \aln(\hat{\pi}, Y) \ge C \cdot n^{\frac{3}{2}},
    \end{align}
    and there exists a polynomial-time algorithm that outputs, given a tournament $Y$, a $\hat{\pi} \in S_n$ having $\aln(\hat{\pi}, Y) = \Omega(n^{3/2})$.
\end{theorem}

\begin{remark}
    Theorem \ref{thm:opt-worst-case-lb} above together with Proposition \ref{prop:opt-high-prob-bound} implies that when $q = O(n^{-1/2})$, then the polynomial-time algorithm mentioned in Theorem \ref{thm:opt-worst-case-lb} w.h.p.~achieves a constant factor approximation to the maximum alignment on input $Y$ drawn from $\mathcal{P}$.
    But, our Theorem~\ref{thm:alignment-approx} shows a near optimal approximation for larger $q$, in which case the guarantee of Theorem~\ref{thm:opt-worst-case-lb} is sub-optimal.
\end{remark}

In what follows, we will thus focus on the regime when $q = \omega(n^{-1/2})$. Our main result (Theorem~\ref{thm:alignment-approx}) will be that, when $q = \omega(n^{-1/2})$, the Ranking By Wins algorithm w.h.p.~achieves a $(1-o(1))$ approximation to the maximum alignment on input $Y$ drawn from $\mathcal{P}$.
Note that, as before, without loss of generality we may assume that the hidden permutation in the planted model is the identity.

\begin{proof}[Proof of Theorem \ref{thm:alignment-approx}]
Recall the alignment objective of a tournament $Y$ evaluated at a permutation $\hat{\pi}$ is equal to
\begin{align*}
    \aln(\hat{\pi}, Y) = \sum_{i < j} \hat{\pi}(i,j) Y_{i,j}.
\end{align*}
The permutation $\hat{\pi}$ output by Ranking By Wins algorithm satisfies
\begin{align*}
    \hat{\pi}(i,j) = \begin{cases}
        +1 & \text{ if } s_i > s_j, \\
        -1 & \text{ if } s_i < s_j,
    \end{cases}
\end{align*}
and when $s_i = s_j$, the Ranking By Wins algorithm break ties between $i$ and $j$ according to some tie-breaking rule. Regardless of the tie-breaking rule, we may always lower bound the alignment objective achieved by the Ranking By Wins algorithm with
{\allowdisplaybreaks
\begin{align*}
    &\sum_{i < j} \hat{\pi}(i,j) Y_{i,j}\\
    &\ge \sum_{i < j} \left(\One\{s_i > s_j\} - \One\{s_i < s_j\}\right)Y_{i,j} - \sum_{i<j} \One\{s_i = s_j\}\\
    &= \sum_{i < j} \left(\One\{s_i > s_j\} - \One\{s_i < s_j\}\right)(2\cdot b_{i,j} - 1) - \sum_{i<j} \One\{s_i = s_j\}\\
    &= 2\sum_{i<j} \One\{s_i > s_j\} \cdot b_{i,j} + \sum_{i<j} \One\{s_i < s_j\} - 2\sum_{i<j} \One\{s_i < s_j\} \cdot b_{i,j} - \sum_{i<j} \One\{s_i \ge s_j\}\\
    &\equalscolon f(Y),
\end{align*}
where we introduce random variables $b_{i,j} = \frac{Y_{i,j}+1}{2}$, so that $b_{i,j}$ are independently distributed as $\text{Ber}\left(\frac{1}{2} + q\right)$ for $i < j$, and use $f(Y)$ to denote this lower bound. It is easy to see that $\{\One\{s_i > s_j\} \cdot b_{i,j}: 1\le i<j \le n\}$, $\{\One\{s_i < s_j\}: 1\le i<j \le n\}$, $\{\One\{s_i < s_j\} \cdot b_{i,j}: 1\le i<j \le n\}$, $\{\One\{s_i \ge s_j\}: 1\le i<j \le n\}$ each form a read-$(2n)$ family. Thus, by Theorem \ref{thm:tail-read-k}, we may get that w.h.p.,}
\begin{align}
    &\quad f(Y) \ge \mathbb{E}\left[f(Y)\right] - t(n), \label{ineq:alignment-tail-bd}
\end{align}
for any function $t(n) = \omega(n^{3/2})$.
Now it remains to lower bound the expectation $\mathbb{E}\left[f(Y)\right]$. We have
\begin{align*}
    &\mathbb{E}[f(Y)]\\
    &= \mathbb{E}\left[2\sum_{i<j} \One\{s_i > s_j\} \cdot b_{i,j} + \sum_{i<j} \One\{s_i < s_j\} - 2\sum_{i<j} \One\{s_i < s_j\} \cdot b_{i,j} - \sum_{i<j} \One\{s_i \ge s_j\}\right]\\
    &= 2\sum_{i<j}\mathbb{P}[s_i > s_j, b_{i,j} = 1] + \sum_
{i<j} \mathbb{P}[s_i < s_j] - 2\sum_{i<j}\mathbb{P}[s_i < s_j, b_{i,j} = 1] - \sum_{i<j} \mathbb{P}[s_i \ge s_j]. \numberthis \label{eq:alignment-expectation}
\end{align*}
Notice that
\begin{align*}
    \mathbb{P}[s_i > s_j, b_{i,j} = 1] &= \mathbb{P}[s_i^{(j)} > s_j^{(i)} - 2] \cdot \mathbb{P}[b_{i,j} = 1]\\
    &= \left(\frac{1}{2} + q\right)\mathbb{P}[s_i^{(j)} > s_j^{(i)} - 2], \numberthis \label{eq:Pij'-1}\\
    \mathbb{P}[s_i < s_j, b_{i,j} = 1] &= \mathbb{P}[s_i^{(j)} < s_j^{(i)} - 2]\cdot \mathbb{P}[b_{i,j} = 1]\\
    &= \left(\frac{1}{2} + q\right)\mathbb{P}[s_i^{(j)} < s_j^{(i)} - 2], \numberthis \label{eq:Pij'-2}
\end{align*}
where we define $s_i^{(j)} = \sum_{k\in [n]: k\ne j} Y_{i,k}$ and similarly define $s_j^{(i)}$. Moreover, since $s_i^{(j)} < s_j^{(i)} - 2$ implies $s_i < s_j$, and $s_i > s_j$ implies $s_i^{(j)} > s_j^{(i)} - 2$, we have
\begin{align}
    \mathbb{P}[s_i^{(j)} < s_j^{(i)} - 2] &\le \mathbb{P}[s_i < s_j],\\
    \mathbb{P}[s_i > s_j] &\le \mathbb{P}[s_i^{(j)} > s_j^{(i)} - 2].
\end{align}
Thus, together with \eqref{eq:alignment-expectation}, \eqref{eq:Pij'-1} and \eqref{eq:Pij'-2}, we get
\begin{align*}
    &\mathbb{E}[f(Y)]\\
    &= 2\sum_{i<j}\mathbb{P}[s_i > s_j, b_{i,j} = 1] + \sum_{i<j} \mathbb{P}[s_i < s_j] - 2\sum_{i<j}\mathbb{P}[s_i < s_j, b_{i,j} = 1] - \sum_{i<j} \mathbb{P}[s_i \ge s_j]\\
    &= \sum_{i<j}\left(\left(1 + 2q\right)\mathbb{P}[s_i^{(j)} > s_j^{(i)} - 2] - \mathbb{P}[s_i \ge s_j]\right) - \sum_{i<j}\left(\left(1+2q\right)\mathbb{P}[s_i^{(j)} < s_j^{(i)} - 2] -\mathbb{P}[s_i < s_j] \right)\\
    &\ge \sum_{i<j}\left(\left(1 + 2q\right)\mathbb{P}[s_i > s_j] - \mathbb{P}[s_i \ge s_j]\right) - \sum_{i<j}\left(\left(1+2q\right)\mathbb{P}[s_i < s_j] -\mathbb{P}[s_i < s_j] \right)\\
    &= 2q \sum_{i<j}\mathbb{P}[s_i > s_j] - 2q\sum_{i<j}\mathbb{P}[s_i < s_j] - \sum_{i<j}\mathbb{P}[s_i = s_j].
\end{align*}

Then, we may apply Theorem \ref{thm:berry-esseen} to obtain
\begin{align*}
        \mathbb{P}\left[s_i > s_j\right] &\ge 1 - \Phi\left(- \frac{4(j-i)q}{\sqrt{2(1-4q^2)n}}\right) - C \cdot \frac{1}{\sqrt{n}}\\
        \mathbb{P}\left[s_i < s_j\right] &\le \Phi\left(- \frac{4(j-i)q}{\sqrt{2(1-4q^2)n}}\right) + C \cdot \frac{1}{\sqrt{n}}\\
        \mathbb{P}\left[s_i = s_j\right] &\le C \cdot \frac{1}{\sqrt{n}},
    \end{align*}
    for some constant $C > 0$ similarly as in \eqref{ineq:berry-esseen-bd} (with a different $C$ from that in \eqref{ineq:berry-esseen-bd}). We then further lower bound $\mathbb{E}[f(Y)]$ by
    \begin{align*}
        \mathbb{E}[f(Y)]
        &\ge 2q \sum_{i<j}\mathbb{P}[s_i > s_j] - 2q\sum_{i<j}\mathbb{P}[s_i < s_j] - \sum_{i<j}\mathbb{P}[s_i = s_j]\\
        &\ge 2q\binom{n}{2} - 4q \sum_{i<j} \Phi\left(- \frac{4(j-i)q}{\sqrt{2(1-4q^2)n}}\right) - O\left(n^{\frac{3}{2}}\right).
    \end{align*}

    Finally, we reuse the calculations in \eqref{ineq:sum-Phi-bd} and get{\allowdisplaybreaks
    \begin{align*}
        \mathbb{E}[f(Y)]
        &\ge 2q\binom{n}{2} - 4q \sum_{i<j} \Phi\left(- \frac{4(j-i)q}{\sqrt{2(1-4q^2)n}}\right) - O\left(n^{\frac{3}{2}}\right)\\
        &\ge 2q\binom{n}{2} - 4q \binom{n}{2}\Phi\left(-\frac{2q\sqrt{n}}{\sqrt{2(1-4q^2)}}\right) - O\left(n^{\frac{3}{2}}\right)\\
        &= 2q\binom{n}{2}\left(1 - 2\cdot \Phi\left(-\frac{2q\sqrt{n}}{\sqrt{2(1-4q^2)}}\right)\right) - O\left(n^{\frac{3}{2}}\right).
    \end{align*}
    Notice that in particular, if $q = \omega(n^{-1/2})$, we have
    }
    \begin{align}
        \mathbb{E}[f(Y)] \ge (1 - o(1))\cdot  2q\binom{n}{2}. \label{ineq:alignment-expectation-lb}
    \end{align}

    Combining the fluctuation bound \eqref{ineq:alignment-tail-bd} and the expectation lower bound \eqref{ineq:alignment-expectation-lb}, we obtain that, for $q = \omega(n^{-1/2})$, w.h.p.~the Ranking By Wins algorithm achieves an alignment objective of at least
    \begin{align*}
        \sum_{i<j} \hat{\pi}(i,j) Y_{i,j} \ge f(Y) \ge (1 - o(1)) \cdot 2q \binom{n}{2}.
    \end{align*}
    On the other hand, by Proposition \ref{prop:opt-high-prob-bound}, for $q = \omega(n^{-1/2})$, w.h.p.~the maximum alignment objective for a tournament $Y$ drawn from $\mathcal{P}$ is at most
    \begin{align*}
        \text{OPT} \le 2q\binom{n}{2} + O\left(n^{\frac{3}{2}}\right) = (1 + o(1)) \cdot 2q\binom{n}{2}.
    \end{align*}
    Thus, we conclude that for $q = \omega(n^{-1/2})$, w.h.p.~the Ranking By Wins algorithm achieves a $(1 - o(1))$ approximation of the maximum alignment on input $Y$ drawn from $\mathcal{P}$.
\end{proof}

\subsection{Planted Ordered Clique}

\subsubsection{Exact Recovery: Proof of Theorem~\ref{thm:planted-ordered-clique-rec}} \label{sec:planted-ordered-clique-rec}

First, we describe a slightly modified spectral algorithm similar to that considered in Theorem~\ref{thm:log-density-comp-rec}, which recovers a reasonably large planted ordered clique of size $C\sqrt{n}$ for a specific constant $C > 0$.
\begin{figure}[H]
    \centering
    \begin{framed}
        \begin{enumerate}
            \item Compute the top eigenvector $\Tilde{v}$ of $\eye Y$ normalized so that $\|\Tilde{v}\|^2 = 1$.
            \item Define $\Tilde{S}_1 \subseteq [n]$ as
            \[\Tilde{S}_1 \colonequals \left\{i \in [n]: |\Tilde{v}_i|^2 \ge \frac{1}{2k}\right\}.\]
            This is our rough estimate of $S$, the support of the planted community.
            \item Let $\Tilde{x} \colonequals \sum_{i \in \Tilde{S}} \Tilde{v}_i$. For every $i \in \Tilde{S}_1$, let $\Tilde{s}_i \colonequals \Tilde{v}_i \Tilde{x}^\dagger$ and write them in the polar form $\Tilde{s}_i = \Tilde{r}_i \cdot \exp(\eye \Tilde{\theta}_i)$ where $\Tilde{r}_i \in \RR_{\ge 0}$ and $\Tilde{\theta}_i \in [-\pi, \pi)$. Split $\Tilde{S}_1$ into two sets $\Tilde{S}_1 = \Tilde{L} \sqcup \Tilde{R}$ in the following way: set $\Tilde{L}$ to be half of the elements in $\Tilde{S}_1$ with the largest $\Tilde{\theta}_i$, and set $\Tilde{R}$ to be the other half of the elements in $\Tilde{S}_1$ with the smallest $\Tilde{\theta}_i$.
            \item For $j \in [n]$, let $\deg_{\textrm{in}}(j, \Tilde{L}) \colonequals \#\{(i,j): i \in \Tilde{L}\}$ denote the number of directed edges from $\Tilde{L}$ to $j$, and $\deg_{\textrm{out}}(j, \Tilde{R}) \colonequals \#\{(j, i): i \in \Tilde{R}\}$ denote the number of directed edges from $j$ to $\Tilde{R}$. Define $\Tilde{S}_2 \subseteq [n]$ as
            \[\Tilde{S}_2 \colonequals \left\{j \in [n]: \deg_{\textrm{in}}(j, \Tilde{L}) \ge \frac{3}{8}k \text{ or } \deg_{\textrm{out}}(j, \Tilde{R}) \ge \frac{3}{8}k\right\}.\]
            This is our refined estimate of $S$.
            \item Finally, create a permutation $\Tilde{\pi}$ on $\Tilde{S}_2$, if the induced tournament on $\Tilde{S}_2$ is acyclic (which is easy to verify) in which case it corresponds to a unique permutation on $\Tilde{S}_2$. Otherwise, the algorithm stops and declares failure.
            \item Output $(\Tilde{S}_2, \Tilde{\pi})$.
        \end{enumerate}
    \end{framed}
    \vspace{-1em}
    \caption{A description of the spectral algorithm analyzed in Theorem~\ref{thm:planted-ordered-clique-rec}.}
    \label{fig:spectral-planted-ordered-clique-rec}
\end{figure}

Now we will show that the spectral algorithm stated in Figure~\ref{fig:spectral-planted-ordered-clique-rec} above will exactly recover $(S, \pi)$ w.h.p.~once $k \ge C\sqrt{n}$ for a large constant $C$. In fact, we only need to show the set $\Tilde{S}_2$ constructed by the algorithm above is exactly $S$ w.h.p., since once $S$ is obtained, the permutation on $S$ can be read off easily in the planted ordered clique setting ($p = 1$ and $q = \frac{1}{2}$).

By Chernoff bound, w.h.p.~$|S| = (1 + o(1))k$. We split $S$ into two sets $S = L \sqcup R$ such that $L$ contains half of the elements of $S$ that are ranked in the first half by $\pi$, and $R$ contains the other half of $S$ that are ranked in the second half by $\pi$. Notice that $|L| = \left(\frac{1}{2} + o(1)\right)k$ and $|R| = \left(\frac{1}{2} + o(1)\right)k$.

Fix a small constant $\varepsilon > 0$. There exists a large constant $C = C(\varepsilon) > 0$ such that a similar argument goes through as in the proof of Proposition~\ref{prop:eigen-perturb} and the proof of the spectral algorithm in Theorem~\ref{thm:log-density-comp-rec}. In particular, replacing the $o(1)$ dependency in the proofs with $\varepsilon$ appropriately, we can show that w.h.p.~$d_\Ham(L, \Tilde{L}) \le \varepsilon k$ and $d_\Ham(R, \Tilde{R}) \le \varepsilon k$. Then, it is an easy exercise of concentration inequalities to show that $\Tilde{S}_2$ defined by the spectral algorithm w.h.p.~recovers $S$ exactly.

On the one hand, for every $j \not\in S$, we see that $\deg_{\textrm{in}}(j, L)$ is independently distributed as $\textrm{Bin}(|L|, \frac{1}{2})$, and $\deg_{\textrm{out}}(j, R)$ is independently distributed as $\textrm{Bin}(|R|, \frac{1}{2})$. By Chernoff bound, w.h.p.~we have that for every $j \not\in S$, $\deg_{\textrm{in}}(j, L) \le \frac{5}{8}|L|$ and $\deg_{\textrm{out}}(j, R) \le \frac{5}{8}|R|$. Here $\frac{5}{8}$ is just some arbitrary constant greater than $\frac{1}{2}$. Thus, we have for every $j \not\in S$,
\begin{align}
    \deg_{\textrm{in}}(j, \Tilde{L}) &\le \deg_{\textrm{in}}(j, L) + d_\Ham(L, \Tilde{L}) \le \frac{5}{8}|L| + \varepsilon k < \frac{3}{8}k,\\
    \deg_{\textrm{out}}(j, \Tilde{R}) &\le \deg_{\textrm{out}}(j, R) + d_\Ham(R, \Tilde{R}) \le \frac{5}{8}|R| + \varepsilon k < \frac{3}{8}k.
\end{align}

On the other hand, for every $j \in S$, either $j \in L$ or $j \in R$. If $j \in L$, then $\deg_{\textrm{out}}(j, R) = |R|$, and if $j \in R$, then $\deg_{\textrm{in}}(j, L) = |L|$. Thus, for every $j \in S$, one the following holds:
\begin{align}
    \deg_{\textrm{in}}(j, \Tilde{L}) &\ge \deg_{\textrm{in}}(j, L) - d_\Ham(L, \Tilde{L}) \ge |L| - \varepsilon k > \frac{3}{8}k, \quad \text{or}\\
    \deg_{\textrm{out}}(j, \Tilde{R}) &\ge \deg_{\textrm{out}}(j, R) - d_\Ham(R, \Tilde{R}) \ge |R| - \varepsilon k > \frac{3}{8}k.
\end{align}

From the above argument, we see that by definition of $\Tilde{S}_2$, w.h.p.~we have $\Tilde{S}_2 = S$, and thus exact recovery is possible if $k \ge C\sqrt{n}$ for a large constant $C > 0$.

\subsubsection{Enhanced Exact Recovery: Proof of Corollary~\ref{cor:planted-ordered-clique-rec}}

Finally, we address how to reduce the constant in front of $\sqrt{n}$. We follow a similar idea as in \cite{alon1998finding}: the algorithm simply starts by guessing a few vertices, and then runs the spectral algorithm described in the section above on a smaller tournament.

To be more concrete, let $b \in \NN$ be a nonnegative integer. The algorithm iterates over all subsets $B \in \binom{[n]}{b}$ of size $b$. For each $B \in \binom{[n]}{b}$, define $V_B \subseteq [n]$ as
\begin{align}
    V_B \colonequals \left\{i \in [n] \setminus B: \deg_{\textrm{in}}(i, B) = b\right\},
\end{align}
i.e., the set of vertices $i \in [n]$ such that all the directed edges are oriented from $B$ to $i$. The algorithm then invokes the spectral algorithm used in the proof of Theorem~\ref{thm:planted-ordered-clique-rec} on the tournament induced by $V_B$, and tries to recover a set $\Tilde{S}_B$. The algorithm then checks if the tournament induced on $B \sqcup \Tilde{S}_B$ is acyclic. If so, it identifies the permutation $\Tilde{\pi}_B$ on $B \sqcup \Tilde{S}_B$. Among all the sets $B$ for which $B \sqcup \Tilde{S}_B$ is acyclic, the algorithm outputs $(B \sqcup \Tilde{S}_B, \Tilde{\pi}_B)$ for the one with the maximum size of $|B \sqcup \Tilde{S}_B|$.

We will sketch the proof why such an algorithm works and can reduce the constant in front of $\sqrt{n}$. We will argue that for some choice of $B \in \binom{[n]}{2}$, the spectral algorithm on $V_B$ w.h.p.~recovers a set $\Tilde{S}_B$ such that $B \sqcup \Tilde{S}_B = S$. In fact, this suffices for showing that w.h.p.~$S$ is the final output, since the following proposition states that w.h.p.~the planted $S$ is the unique maximum acyclic set of vertices in the random tournament, the proof of which is deferred to Appendix~\ref{app:pf:prop:unique-acyclic}.

\begin{proposition}\label{prop:unique-acyclic}
    Fix $p = 1$ and $q = \frac{1}{2}$. For any constant $c > 0$, if $k = k(n) \ge c\sqrt{n}$, then w.h.p.~the planted community $S$ is the unique maximum acyclic subset in the random tournament drawn from $\sP$.
\end{proposition}

Now consider the set $B \subseteq S$ given by the $b$ elements that are ranked in the front by $\pi$. We will show that if $k \ge (1 + \varepsilon) \cdot \frac{C}{2^{b/2}}\sqrt{n}$ for any $\varepsilon > 0$, then w.h.p.~$B \sqcup \Tilde{S}_B = S$. Note that $|V_B|$ is distributed as $|S| - b + \textrm{Bin}(n - |S|, \frac{1}{2^b})$, and thus w.h.p.~we have $|V_B| = (1 + o(1)) \cdot \left[\frac{n}{2^b} + \left(1 - \frac{1}{2^b}\right) |S|\right] = (1 + o(1)) \cdot \frac{n}{2^b}$. The spectral algorithm invoked on $V_B$ then attempts to exactly recover the planted ordered clique $S \setminus B$ in $V_B$, which w.h.p.~succeeds once $|S \setminus B| \ge C \cdot \sqrt{|V_B|} = (1 + o(1)) \frac{C}{2^{b/2}}\sqrt{n}$ by Theorem~\ref{thm:planted-ordered-clique-rec}.

Since the argument works for any $b \in \NN$, there exists a polynomial time algorithm that achieves exact recovery if $k \ge \num{1.01} \cdot \frac{C}{2^{b/2}}\sqrt{n}$, at the cost of running the spectral algorithm described in Figure~\ref{fig:spectral-planted-ordered-clique-rec} $O(n^b)$ times. This concludes the proof.

\appendix

\section{Omitted Proofs}

\subsection{Proof of Proposition~\ref{prop:weak-sep}}
\label{app:weak-sep}

Without loss of generality, center and scale $f$ so that $\EE_{\sP}[f] = \mu$, $\EE_{\sQ}[f] = -\mu$, $\Var_{\sP}[f] \le 1$, and $\Var_{\sQ}[f] \le 1$, for some $\mu = \mu_n \ge 0$. By assumption we can take $\mu \ge \epsilon$ for all sufficiently large $n$, where $\epsilon > 0$ is a constant independent of $n$. We can also assume $\mu \le \sqrt{3}$, or else Chebyshev's inequality shows that the choice $t = 0$ works, namely $\PP_{Y \sim \sP}[f(Y) \le 0] \le 1/3$ and $\PP_{Y \sim \sQ}[f(Y) \ge 0] \le 1/3$.

Recall, for any (integrable) random variable $X$, the formula
\[ \EE[X] = \int_0^\infty \PP[X \ge u] \, du - \int_{-\infty}^0 \PP[X < u] \, du = \int_0^\infty \left(\PP[X \ge u] - \PP[X < -u] \right) du. \]
For a constant $C = C(\epsilon) > \sqrt{3}$ to be chosen later,
\begin{equation}\label{eq:EPf}
\Ex_{\sP}[f] = \int_0^C \left(\Px_{Y \sim \sP}[f(Y) \ge u] - \Px_{Y \sim \sP}[f(Y) < -u]\right) du + \Delta_{\sP}
\end{equation}
where, using Chebyshev,
\begin{align*}
|\Delta_{\sP}| :=& \, \left|\int_C^\infty \left(\Px_{Y \sim \sP}[f(Y) \ge u] - \Px_{Y \sim \sP}[f(Y) < -u]\right) du\right| \\
\le& \, \int_C^\infty \Px_{Y \sim \sP}(|f(Y)| \ge u) \, du \\
\le& \, \int_C^\infty (u-\sqrt{3})^{-2} \, du \\
=& \; (C-\sqrt{3})^{-1}.
\end{align*}
Choose $C$ large enough so that $|\Delta_{\sP}| \le \epsilon/2$. The same calculations hold with $\sQ$ in place of $\sP$. Since $\EE_{\sP}[f] - \EE_{\sQ}[f] \ge 2\epsilon$ and $|\Delta_{\sP}|, |\Delta_{\sQ}| \le \epsilon/2$, we have from~\eqref{eq:EPf},
\[ \int_0^C \left(\Px_{Y \sim \sP}[f(Y) \ge u] - \Px_{Y \sim \sQ}[f(Y) \ge u] - \Px_{Y \sim \sP}[f(Y) < -u] + \Px_{Y \sim \sQ}[f(Y) < -u]\right) du \ge \epsilon. \]
There must exist some $u$ for which the integrand exceeds its average value, which gives us $s \in [0,C]$ such that
\[ \Px_{Y \sim \sP}[f(Y) \ge s] - \Px_{Y \sim \sQ}[f(Y) \ge s] - \Px_{Y \sim \sP}[f(Y) < -s] + \Px_{Y \sim \sQ}[f(Y) < -s] \ge \epsilon/C. \]
This means one of the two quantities $\PP_{Y \sim \sP}[f(Y) \ge s] - \PP_{Y \sim \sQ}[f(Y) \ge s]$ or $\PP_{Y \sim \sQ}[f(Y) < -s] - \PP_{Y \sim \sP}[f(Y) < -s]$ is at least $\epsilon/(2C)$, and in either case we get weak detection by choosing the threshold $t = s$ or $t = -s$, respectively.

\subsection{Proof of Proposition~\ref{prop:mgf-inv}}
\label{app:pf:prop:mgf-inv}

Note that, by the Knuth shuffle construction of random permutations, when $\pi \sim \Unif(\Sym([h]))$, then $$\inv(\pi) \stackrel{\text{(d)}}{=} X_1 + \cdots + X_{h - 1}$$ where $X_i \sim \Unif(\{0, \dots, i\})$ are independent.
Thus we may compute and use Hoeffding's inequality on this representation, introducing a parameter $0 < r \leq \binom{h}{2}$:
\begin{align*}
    &\Ex_{\pi \sim \text{Unif}(\Sym([h]))} (1 + x)^{\binom{h}{2} - 2\inv(\pi) } \\
    &\le \Ex_{\pi \sim \text{Unif}(\Sym([h]))} \exp\left(x\left(\binom{h}{2} - 2\inv(\pi)\right)\right) \\
    &\le \exp(xr) + \int_{\exp(xr)}^{\exp(x\binom{h}{2})} \mathbb{P}\left[\exp\left(x\left(\binom{h}{2} - 2\inv(\pi)\right)\right) \ge t\right]dt\\
    &= \exp(xr) + \int_{\exp(xr)}^{\exp(x\binom{h}{2})} \mathbb{P}\left[\inv(\pi) \le \frac{1}{2}\binom{h}{2} - \frac{\log t}{2x}\right]dt\\
    &\le \exp(x r) + \int_{\exp(xr)}^{\exp(x \binom{h}{2}) } \exp\left(-2 \frac{\left(\frac{\log t}{2x}\right)^2}{h^3}\right)dt \\
    &= \exp(x r) + \frac{\sqrt{\pi} \exp\left(\frac{1}{2}x^2h^3\right)\left[\text{erf}\left(\frac{\frac{x\binom{h}{2}}{x^2 h^3} - 1}{\sqrt{\frac{2}{x^2h^3}}}\right)  -\text{erf}\left(\frac{\frac{x r}{x^2 h^3} - 1}{\sqrt{\frac{2}{x^2h^3}}}\right)\right] }{\sqrt{\frac{2}{x^2h^3}}}
    \intertext{since the indefinite integral is $\int \exp(-(\log t)^2/a)\,dt= \frac{\sqrt{\pi a}\cdot\exp(a/4) \text{erf}\left(\frac{2\log(x)-a}{2\sqrt{a}}\right)}{2} + C$, and finally we may bound}
    &\le \exp(x^2 h^3 / 2) \left(1 + 2\sqrt{x^2 h^3 / 2} \right),
\end{align*}
by setting $r = xh^3 / 2$, provided that $r \leq \binom{h}{2}$.

Note that it is possible that $r > \binom{h}{2}$, but in this case we may just bound the original expectation above by $\exp(xr) = \exp(x^2h^3 / 2)$.
Thus, in all cases the bound above is valid, completing the proof.

\subsection{Proof of Proposition~\ref{prop:binom3}}
\label{app:pf:prop:binom3}

Computing directly, we have
{\allowdisplaybreaks
\begin{align*}
    &\mathbb{E}_M\left[\exp(x^2M^3)\right]\\
    &\le \exp\left(64x^2 \frac{k^6}{n^3}\right) + \int_{\exp\left(64x^2 \frac{k^6}{n^3}\right)}^{\exp(x^2k^3)} \mathbb{P}\left[\exp(x^2M^3) \ge t\right] dt\\
    &= \exp\left(64x^2 \frac{k^6}{n^3}\right) + \int_{\exp\left(64x^2 \frac{k^6}{n^3}\right)}^{\exp(x^2k^3)} \mathbb{P}\left[M \ge \left(\frac{\log t}{x^2}\right)^{\frac{1}{3}}\right] dt\\
    &= \exp\left(64x^2 \frac{k^6}{n^3}\right) + \int_{\exp\left(64x^2 \frac{k^6}{n^3}\right)}^{\exp(x^2k^3)} \mathbb{P}\left[M - \mathbb{E}[M] \ge \left(\frac{\log t}{x^2}\right)^{\frac{1}{3}} - \mathbb{E}[M]\right] dt\\
    &\le \exp\left(64x^2 \frac{k^6}{n^3}\right) + \int_{\exp\left(64x^2 \frac{k^6}{n^3}\right)}^{\exp(x^2k^3)} \exp\left(- \frac{\left(\left(\frac{\log t}{x^2}\right)^{\frac{1}{3}} - \mathbb{E}[M]\right)^2}{2\mathbb{E}[M] +  \left(\left(\frac{\log t}{x^2}\right)^{\frac{1}{3}} - \mathbb{E}[M]\right)}\right) dt
    \intertext{by Chernoff bound,}
    &\le \exp\left(64x^2 \frac{k^6}{n^3}\right) + \int_{\exp\left(64x^2 \frac{k^6}{n^3}\right)}^{\exp(x^2k^3)} \exp\left(- C \cdot \left(\left(\frac{\log t}{x^2}\right)^{\frac{1}{3}} - \mathbb{E}[M]\right)\right) dt
    \intertext{since for $t \ge \exp\left(64x^2 \frac{k^6}{n^3}\right)$, $\left(\frac{\log t}{x^2}\right)^{\frac{1}{3}} - \mathbb{E}[M] \ge 4 \frac{k^2}{n} - \frac{k^2}{n-k} \ge 2\frac{k^2}{n} \ge \mathbb{E}[M]$,}
    &\le \exp\left(64x^2 \frac{k^6}{n^3}\right) + \int_{\exp\left(64x^2 \frac{k^6}{n^3}\right)}^{\exp(x^2k^3)} \exp\left(- C' \cdot \left(\frac{\log t}{x^2}\right)^{\frac{1}{3}}\right) dt
    \intertext{since for $t \ge \exp\left(64x^2 \frac{k^6}{n^3}\right)$, $\left(\frac{\log t}{x^2}\right)^{\frac{1}{3}} \ge 4\frac{k^2}{n} \ge 2\mathbb{E}[M]$,}
    &= \exp\left(64x^2 \frac{k^6}{n^3}\right) + \int_{64x^2 \frac{k^6}{n^3}}^{x^2k^3} \exp\left(- C' \cdot \left(\frac{u}{x^2}\right)^{\frac{1}{3}}\right) \exp(u) du
    \intertext{by setting $u = \log t$,}
    &\le \exp\left(64x^2 \frac{k^6}{n^3}\right) + \int_{64x^2 \frac{k^6}{n^3}}^{x^2k^3} \exp\left(- C'' \cdot \left(\frac{u}{x^2}\right)^{\frac{1}{3}}\right) du
    \intertext{since for $u \le x^2 k^3$, $C' \left(\frac{u}{x^2}\right)^{\frac{1}{3}} \ge 2 u$ when $pq^2 k = o(1)$,}
    &\le \exp\left(64x^2 \frac{k^6}{n^3}\right) + \exp\left(-4C'' \frac{k^2}{n}\right)3\Bigg[\left(\frac{(x^2)^{\frac{1}{3}}}{C''}\right)\left(64x^2 \frac{k^6}{n^3}\right)^{\frac{2}{3}}\\
    &\quad + 2\left(\frac{(x^2)^{\frac{1}{3}}}{C''}\right)^2\left(64x^2 \frac{k^6}{n^3}\right)^{\frac{1}{3}} + 2\left(\frac{(x^2)^{\frac{1}{3}}}{C''}\right)^3\Bigg]\intertext{since the indefinite integral is $\int\exp(-u^{1/3}/b)\, du = -3(bu^{2/3} + 2b^2u^{1/3} + 2b^3)\exp(-u^{1/3}/b) + C$}
    &\le 1 + o(1) + \exp\left(-4C'' \frac{k^2}{n}\right)\left(C''' p^2q^4 + o(1)\right)
    \intertext{since $p^2q^4 \frac{k^6}{n^3} = o(1)$,}
    &\le 1 + o(1) + \exp\left(-4C'' \frac{k^2}{n}\right)\left( \frac{c}{k^2}\right)
    \intertext{when $pq^2 k = o(1)$, where $c > 0$ is an arbitarily small constant,}
    &\le 1 + o(1) + c\cdot \exp\left(-\frac{4C''}{n}\right)
    \intertext{since the relevant term is nondecreasing as $k$ decreases, and $k \ge 1$}
    &\le 1 + o(1) + c, \label{bd:term1}
\end{align*}
which is arbitrarily close to $1$ since $c > 0$ is arbitrarily small, under the assumptions $x^2 \frac{k^6}{n^3} = o(1), x k = o(1), x^2 \frac{k^2}{n} = o(1)$.

\subsection{Proof of Proposition~\ref{prop:eig-A}}
\label{app:pf:prop:eig-A}

        We explicitly multiply the putative eigenvectors by $A$ to verify the claim:
        {\allowdisplaybreaks
        \begin{align*}
            (A_{\ell}v^{(i)})_j
            &= \sum_{t=1}^l (A_{\ell})_{j,t} v^{(i)}_t\\
            &= -\sum_{t=1}^{j-1} \eye \cdot \exp\left(-\eye \pi \frac{(2i-1)t}{\ell}\right)  + \sum_{t=j+1}^l \eye \cdot \exp\left(-\eye \pi \frac{(2i-1)t}{\ell}\right) \\
            &= \eye \left(-\sum_{t=1}^{j-1} \exp\left(-\eye \pi \frac{(2i-1)t}{\ell}\right)  + \sum_{t=j+1}^l \exp\left(-\eye \pi \frac{(2i-1)t}{\ell}\right)\right)\\
            &= \eye\left(\sum_{t=1}^{j-1} \exp\left(-\eye \pi \frac{(2i-1)t}{\ell} - \eye\pi (2i -  1)\right)  + \sum_{t=j+1}^l \exp\left(-\eye \pi \frac{(2i-1)t}{\ell}\right)\right)\\
            &= \eye\sum_{t= j+1}^{\ell+j-1} \exp\left(-\eye \pi \frac{(2i-1)t}{\ell}\right) \\
            &= \eye\cdot \exp\left(-\eye \pi \frac{(2i-1)j}{\ell}\right) \cdot \sum_{t = 1}^{\ell-1} \exp\left(-\eye \pi \frac{(2i-1)t}{\ell}\right)\\
            &= \eye\cdot \exp\left(-\eye \pi \frac{(2i-1)j}{\ell}\right) \cdot \exp\left(-\eye \pi \frac{2i-1}{\ell}\right) \cdot \frac{1 - \exp\left(-\eye \pi \frac{(2i-1)(l - 1)}{\ell}\right)}{1 - \exp\left(-\eye \pi \frac{2i - 1}{\ell}\right)}\\
            &= \exp\left(-\eye \pi \frac{(2i-1)j}{\ell}\right) \cdot \eye \cdot \frac{1 + \exp\left(-\eye \pi \frac{2i-1}{\ell}\right)}{1 - \exp\left(-\eye \pi \frac{2i-1}{\ell}\right)}\\
            &= \exp\left(-\eye \pi \frac{(2i-1)j}{\ell}\right) \cdot \eye \cdot \frac{\exp\left(\eye \pi \frac{2i-1}{2l}\right) + \exp\left(-\eye \pi \frac{2i-1}{2l}\right)}{\exp\left(\eye \pi \frac{2i-1}{2l}\right) - \exp\left(-\eye \pi \frac{2i-1}{2l}\right)}\\
            &= \exp\left(-\eye \pi \frac{(2i-1)j}{\ell}\right) \cdot \frac{1}{\tan\left(\frac{2i-1}{2l}\pi\right)}\\
            &= \frac{1}{\tan\left(\frac{2i-1}{2l}\pi\right)} \cdot v^{(i)}_j,
        \end{align*}
        so $v^{(i)}$ is an eigenvector corresponding to eigenvalue $\lambda_i = 1/\tan(\frac{2i - 1}{2\ell}\pi)$. Moreover, since the eigenvalues $\lambda_i$ are distinct, the eigenvectors $v^{(i)}$ form an orthogonal basis of $\CC^{\ell}$, completing the proof.
        }

\subsection{Proof of Proposition~\ref{prop:g-expansion}}
\label{app:pf:prop:g-expansion}

We compute directly using the orthogonal polynomial expansion:
\begin{align*}
    &\quad g(Z,P)\\
    &= \sum_{A, B} \hat{f}_{A, B} \Ex_{X, \theta} \vast[ b_{A, B} \left(\prod_{\{i,j\}\in A} P_{i,j}((1-\theta_i \theta_j)Z_{i,j} + \theta_i \theta_j X_{i,j})\right)\\
    &\hspace{3.5cm} \left(\prod_{\{i,j\}\in B}\left(P_{i,j}^2((1-\theta_i \theta_j)Z_{i,j} + \theta_i \theta_j X_{i,j})^2 - p\right)\right) \vast]\\
    &=\sum_{A, B} \hat{f}_{A, B} \Ex_{X, \theta} \left[ b_{A, B} \left(P^A \sum_{A' \subseteq A} Z^{A'}X^{A - A'} \left(J-\theta\theta^\top\right)^{A'}\left(\theta\theta^\top\right)^{A - A'}\right)\left(P - p\right)^B \right]\\
    &= \sum_{A, B} \hat{f}_{A, B} b_{A, B} \cdot P^A (P-p)^{B}\sum_{A' \subseteq A} Z^{A'}  \Ex_{X} \left[X^{A - A'}\right] \Ex_{\theta} \left[\left(J-\theta\theta^\top\right)^{A'}\left(\theta\theta^\top\right)^{A - A'}\right] \\
    &= \sum_{A, B} \hat{f}_{A, B} b_{A, B} \cdot (P-p)^{B} \sum_{A' \subseteq A} Z^{A'} P^{A'} \sum_{\delta \subseteq A - A'} (P - p)^{\delta} p^{|A - A' - \delta|} \\
    &\hspace{7cm}\Ex_{X} \left[X^{A - A'}\right] \Ex_{\theta} \left[\left(J-\theta\theta^\top\right)^{A'}\left(\theta\theta^\top\right)^{A - A'}\right] \\[1em]
    &= \sum_{A, B} \hat{f}_{A, B} b_{A, B} \cdot  \sum_{A' \subseteq A} Z^{A'} P^{A'} \sum_{\substack{B':\\
    B \subseteq B'\\
    B' - B \subseteq A - A'}} (P - p)^{B'} p^{|A - A' + B - B'|} \\
    &\hspace{7cm} \Ex_{X} \left[X^{A - A'}\right] \Ex_{\theta} \left[\left(J-\theta\theta^\top\right)^{A'}\left(\theta\theta^\top\right)^{A - A'}\right] \\[1em]
    &= \sum_{A, B} \hat{f}_{A, B} b_{A, B} \cdot  \sum_{A' \subseteq A} \sum_{\substack{B':\\
    B \subseteq B'\\
    B' - B \subseteq A - A'}} h_{A', B'}(P \circ Z)\cdot \frac{1}{b_{A', B'}} p^{|A - A' + B - B'|} \\
    &\hspace{7cm}\Ex_{X} \left[X^{A - A'}\right] \Ex_{\theta} \left[\left(J-\theta\theta^\top\right)^{A'}\left(\theta\theta^\top\right)^{A - A'}\right]\\[1em]
    &= \sum_{A', B'}  h_{A', B'}(P \circ Z) \cdot \sum_{\substack{A, B: \\
    A' \subseteq A\\
    B \subseteq B'\\
    B' - B \subseteq A - A'}} \frac{b_{A, B}}{b_{A', B'}}\cdot  p^{|A - A' + B - B'|} \\
    &\hspace{7cm}\Ex_{X} \left[X^{A - A'}\right] \Ex_{\theta} \left[\left(J - \theta \theta^\top\right)^{A'}\left(\theta \theta^\top\right)^{ A - A'}\right]\cdot \hat{f}_{A, B}.
\end{align*}

\subsection{Proof of Proposition~\ref{prop:w-zero-B}}
\label{app:pf:prop:w-zero-B}

    Suppose $B \ne \emptyset$. We will prove by induction that $w_{A, B} = 0$. We prove by induction on the lexicographic order of the pair $(|A| + |B|, |A|)$. Now suppose $|w_{A', B'}| = 0$ for any pair of $(A', B')$ with $B' \ne \emptyset$ such that $(|A'| + |B'|, |A'|) \prec (|A| + |B|, |A|)$. Then, \eqref{eq:w-recursion} simplifies to
\begin{align*}
    &\quad w_{A, B} \cdot \Ex_{\theta} \left[\left(J - \theta \theta^\top\right)^{A}\right]\\
    &= \One\{B = \emptyset\} p^{|A|/2}\Ex_{X}\left[X^{A}\right] \Ex_{\theta}\left[ \left(\theta\theta^\top\right)^{A}\theta_1\right]\\
    &\quad- \sum_{\substack{A' \subseteq A\\
    B' \supseteq B\\
    (A', B') \ne (A, B)}} w_{A', B'} \cdot \One\{B' - B \subseteq A - A'\}   \frac{b_{A, B}}{b_{A', B'}}\cdot  p^{|A - A' + B - B'|} \\
    &\hspace{3cm}\Ex_{X} \left[X^{A - A'}\right] \Ex_{\theta} \left[\left(J - \theta \theta^\top\right)^{A'}\left(\theta \theta^\top\right)^{ A - A'}\right]\\[1em]
    &= - \sum_{\substack{A' \subseteq A\\
    B' \supseteq B\\
    B' - B \subseteq A - A'\\
    (A', B') \ne (A, B)}} w_{A', B'} \cdot\frac{b_{A, B}}{b_{A', B'}}\cdot  p^{|A - A' + B - B'|}  \Ex_{X} \left[X^{A - A'}\right] \Ex_{\theta} \left[\left(J - \theta \theta^\top\right)^{A'}\left(\theta \theta^\top\right)^{ A - A'}\right]\\
    &= 0,
\end{align*}
since for any pair $(A', B')$ that satisfies $A' \subseteq A, B' \supseteq B, B' - B \subseteq A - A'$, and $(A', B') \ne (A, B)$, we have $(|A'|+|B'|, |A'|) \prec (|A| + |B|, |A|)$. This proves the claim that $w_{A, B} = 0$ if $B \ne \emptyset$, as $\Ex_{\theta} \left[\left(J - \theta \theta^\top\right)^{A}\right]\ne 0$.

\subsection{Proof of Proposition~\ref{prop:w-zero-cc}}
\label{app:pf:prop:w-zero-cc}

    We prove by induction on $|A|$. Suppose $A$ has a connected component not containing vertex $1$, and suppose we have proved the statement for every $A'$ with fewer edges than $A$.

    Let $\delta$ be the connected component of $A$ containing vertex $1$, and note that $\delta$ could potentially be empty.

    By induction, we have
{\allowdisplaybreaks
    \begin{align*}
        &w_{A, \emptyset}\cdot  \Ex_{\theta} \left[\left(J - \theta \theta^\top\right)^{A}\right]\\
        &= p^{|A|/2}\Ex_{X}\left[X^{A}\right] \Ex_{\theta}\left[ \left(\theta\theta^\top\right)^{A}\theta_1\right]\\
        &\quad- \sum_{A' \subsetneq A} w_{A', \emptyset} \cdot \frac{b_{A, \emptyset}}{b_{A', \emptyset}}\cdot  p^{|A - A'|}\cdot  \Ex_{X} \left[X^{A - A'}\right] \Ex_{\theta} \left[\left(J - \theta \theta^\top\right)^{A'}\left(\theta \theta^\top\right)^{ A - A'}\right]
        \intertext{using the recursive formula \eqref{eq:w-recursion} and Proposition~\ref{prop:w-zero-B},}
        &= p^{|A|/2}\Ex_{X}\left[X^{A}\right] \Ex_{\theta}\left[ \left(\theta\theta^\top\right)^{A}\theta_1\right]\\
        &\quad- \sum_{\delta' \subseteq \delta} w_{\delta', \emptyset} \cdot \frac{b_{A, \emptyset}}{b_{\delta', \emptyset}}\cdot  p^{|A - \delta'|}\cdot  \Ex_{X} \left[X^{A - \delta'}\right] \Ex_{\theta} \left[\left(J - \theta \theta^\top\right)^{\delta'}\left(\theta \theta^\top\right)^{ A - \delta'}\right]
        \intertext{by our inductive hypothesis, as every $A' \subsetneq A$ that is not a subgraph of the component $\delta$ contains a connected component not containing vertex $1$ and has fewer edges than $A$,}
        &= p^{|A|/2}\Ex_{X}\left[X^{A}\right] \Ex_{\theta}\left[ \left(\theta\theta^\top\right)^{A}\theta_1\right]\\
        &\quad - \Bigg(p^{|\delta|/2}\Ex_{X}\left[X^{\delta}\right] \Ex_{\theta}\left[ \left(\theta\theta^\top\right)^{\delta}\theta_1\right]\\
        &\quad - \sum_{\delta' \subsetneq \delta} w_{\delta', \emptyset} \cdot \frac{b_{\delta, \emptyset}}{b_{\delta', \emptyset}}\cdot  p^{|\delta - \delta'|}\cdot  \Ex_{X} \left[X^{\delta - \delta'}\right] \Ex_{\theta} \left[\left(J - \theta \theta^\top\right)^{\delta'}\left(\theta \theta^\top\right)^{ \delta - \delta'}\right]\Bigg)\\
        &\quad  \cdot \frac{1}{ \Ex_{\theta} \left[\left(J - \theta \theta^\top\right)^{\delta}\right]}\cdot \frac{b_{A, \emptyset}}{b_{\delta, \emptyset}}\cdot  p^{|A - \delta|}\cdot  \Ex_{X} \left[X^{A - \delta}\right] \Ex_{\theta} \left[\left(J - \theta \theta^\top\right)^{\delta}\left(\theta \theta^\top\right)^{ A - \delta}\right]\\
        &\quad- \sum_{\delta' \subsetneq \delta} w_{\delta', \emptyset} \cdot \frac{b_{A, \emptyset}}{b_{\delta', \emptyset}}\cdot  p^{|A - \delta'|}\cdot  \Ex_{X} \left[X^{A - \delta'}\right] \Ex_{\theta} \left[\left(J - \theta \theta^\top\right)^{\delta'}\left(\theta \theta^\top\right)^{ A - \delta'}\right]
        \intertext{by separating out the term corresponding to $\delta$ from the summation over $\delta' \subseteq \delta$ and applying \eqref{eq:w-recursion} to the $\delta$ term,}
        &= p^{|A|/2}\Ex_{X}\left[X^{A}\right] \Ex_{\theta}\left[ \left(\theta\theta^\top\right)^{A}\theta_1\right]  - p^{|\delta|/2}\cdot \frac{b_{A, \emptyset}}{b_{\delta, \emptyset}}\cdot  p^{|A - \delta|} \cdot \Ex_{X}\left[X^{\delta}\right]\Ex_{X} \left[X^{A - \delta}\right]\\
        &\quad \cdot  \Ex_{\theta}\left[ \left(\theta\theta^\top\right)^{\delta}\theta_1\right]\cdot \frac{1}{ \Ex_{\theta} \left[\left(J - \theta \theta^\top\right)^{\delta}\right]}\cdot  \Ex_{\theta} \left[\left(J - \theta \theta^\top\right)^{\delta}\left(\theta \theta^\top\right)^{ A - \delta}\right]\\
        &\quad - \sum_{\delta' \subsetneq \delta} w_{\delta', \emptyset} \cdot \Bigg(\frac{b_{A, \emptyset}}{b_{\delta', \emptyset}}\cdot  p^{|A - \delta'|}\cdot  \Ex_{X} \left[X^{A - \delta'}\right] \Ex_{\theta} \left[\left(J - \theta \theta^\top\right)^{\delta'}\left(\theta \theta^\top\right)^{ A - \delta'}\right]\\
        &\quad - \frac{b_{\delta, \emptyset}}{b_{\delta', \emptyset}} \cdot \frac{b_{A, \emptyset}}{b_{\delta, \emptyset}}\cdot  p^{|\delta - \delta'|}\cdot  p^{|A - \delta|}\cdot  \Ex_{X} \left[X^{\delta - \delta'}\right]\Ex_{X} \left[X^{A - \delta}\right] \\
        &\quad \cdot \Ex_{\theta} \left[\left(J - \theta \theta^\top\right)^{\delta'}\left(\theta \theta^\top\right)^{ \delta - \delta'}\right]\cdot \frac{1}{ \Ex_{\theta} \left[\left(J - \theta \theta^\top\right)^{\delta}\right]}\cdot   \Ex_{\theta} \left[\left(J - \theta \theta^\top\right)^{\delta}\left(\theta \theta^\top\right)^{ A - \delta}\right]\Bigg)\\
        &= 0,
    \end{align*}
    since by component-wise independence (as we have argued in Proposition~\ref{prop:component-independence}), for two vertex disjoint $A_1, A_2$, we have \[\Ex_{X}[X^{A_1} X^{A_2}] = \Ex_{X}[X^{A_1}]\Ex_{X}[X^{A_2}],\] and more generally for any functions $f_1, f_2$ that take input from disjoint sets of variables, \[\Ex_{\theta}\left[f_1\left(\{\theta_i\}_{i \in V(A_1)}\right) f_2\left(\{\theta_j\}_{j \in V(A_2)}\right) \right] = \Ex_{\theta}\left[f_1\left(\{\theta_i\}_{i \in V(A_1)}\right)\right] \Ex_{\theta}\left[f_2\left(\{\theta_j\}_{j \in V(A_2)}\right) \right].\]
    }

\subsection{Proof of Remark~\ref{rem:alternative-corr-bound}} \label{app:rem:alternative-corr-bound}

Recall that the observed matrix $Y \in \{0, \pm 1\}^{n \times n}$, when $Y \sim \sP$, is generated in the following way:
\begin{enumerate}
    \item First, sample a random vector $\theta \in \{0,1\}^{n}$ such that $\theta_i \stackrel{\text{iid}}{\sim} \text{Bern}(k/n)$.
    \item Next, sample a permutation $\pi \in S_n$ uniformly at random.
    \item Then, generate each entry of $Y_{i,j}$ for $i < j$ independently as
    \begin{align*}
        Y_{i,j} &= \begin{cases}
            0 & \quad \text{ with probability } 1 - p\\
            T_{i,j} & \quad \text{ with probability } p
        \end{cases}
    \end{align*}
    where $T_{i,j} \sim \text{Rad}(1/2 + q \cdot \theta_i\theta_j \pi(i,j))$ independently. The lower diagonal entries are set to $Y_{j,i} = -Y_{i,j}$.
\end{enumerate}

We may use a different strategy from the one used in Section~\ref{sec:log-density-comp-rec-lower}, and equivalently sample $Y$ using the following procedure:
\begin{itemize}
    \item Sample the random vector $\theta \in \{0,1\}^{n}$ such that $\theta_i \stackrel{\text{iid}}{\sim} \text{Bern}(k/n)$ and the permutation $\pi \in S_n$ as before.
    \item For $i < j$, sample $P_{i,j} \stackrel{\text{iid}}{\sim} \text{Bern}(p)$, $Z_{i,j} \stackrel{\text{iid}}{\sim} \text{Rad}(1/2)$, $\sigma_{i,j} \stackrel{\text{iid}}{\sim} \text{Bern}(2q)$, and $X_{i,j} \sim \text{Rad}(1/2 + \theta_i\theta_j\pi(i,j)/2)$ independently. For $i < j$, set $P_{i,j} = P_{j,i}$, $\sigma_{i,j} = \sigma_{j,i}$, $Z_{i,j} = - Z_{j,i}$, and $X_{i,j} = -X_{j,i}$.
    \item Then, generate \[Y_{i,j} = P_{i,j} \cdot \left[(1 - \sigma_{i,j})Z_{i,j} + \sigma_{i,j} X_{i,j}\right],\]
    or, in matrix form $Y = P \circ [(J - \sigma) \circ Z + \sigma \circ X]$.
\end{itemize}

Note that the $X$ variables are defined differently compared to those defined in Section~\ref{sec:log-density-comp-rec-lower}. As before, we want to upper bound low-degree correlation
\begin{align}
    \Corr_{\le D}(\sP) \colonequals \sup_{f \in \mathbb{R}[Y]_{\le D}} \frac{\Ex_{(\theta, Y) \sim \sP} \left[f(Y)\theta_1\right]}{\sqrt{\Ex_{Y \sim \sP} f(Y)^2}}.
\end{align}
by simplifying the demoninator using Jensen's inequality, but now we take the average with respect to $X, \sigma$ variables.

For any $f \in \mathbb{R}[Y]_{\le D}$,
\begin{align*}
    \Ex_{Y \sim \sP}[f(Y)^2] &\ge \Ex_{(Z, P)\sim \sP} \left[\Ex_{(X, \sigma) \sim \sP} \left[f\left(P \circ [(J - \sigma) \circ Z + \sigma\circ X]\right)\right]^2\right]\\
    &\equalscolon \Ex_{(Z, P) \sim \sP} \left[g(Z, P)^2\right],
\end{align*}
where we define $g(Z, P) \colonequals \Ex_{(X, \theta) \sim \sP} f\left(P \circ \left( \left(J - \sigma\right) \circ Z + \sigma \circ X\right)\right)$.

Similar to Proposition~\ref{prop:g-expansion}, we may explicitly compute the coefficients $\hat{g}$ in the expansion $g(Z,P) = \sum_{A,B} \hat{g}_{A,B} \cdot h_{A,B}(P \circ Z)$ using the expansion of $f(Y) = \sum_{A,B} \hat{f}_{A,B} \cdot h_{A,B}(Y)$:
{\allowdisplaybreaks
\begin{align*}
    g(Z, P) &= \sum_{A, B} \hat{f}_{A, B} \Ex_{X, \sigma} \vast[ b_{A, B} \left(\prod_{\{i,j\}\in A} P_{i,j}((1-\sigma_{i,j})Z_{i,j} + \sigma_{i,j}X_{i,j})\right)\\
    &\hspace{3.2cm} \cdot \left(\prod_{\{i,j\}\in B}\left(P_{i,j}^2((1-\sigma_{i,j})Z_{i,j} + \sigma_{i,j}X_{i,j})^2 - p\right)\right) \vast]\\
    &=\sum_{A, B} \hat{f}_{A, B} \Ex_{X, \sigma} \left[ b_{A, B} \left(P^A \sum_{A' \subseteq A} Z^{A'}X^{A - A'} (J-\sigma)^{A'}\sigma^{A - A'}\right)\left(P - p\right)^B \right]\\
    &= \sum_{A, B} \hat{f}_{A, B} b_{A, B} P^{A} (P-p)^{B}\sum_{A' \subseteq A} Z^{A'} \Ex_{X} \left[X^{A - A'}\right] \Ex_{\sigma} \left[(J-\sigma)^{A'}\sigma^{A - A'}\right]\\
    &= \sum_{A, B}\sum_{A' \subseteq A} Z^{A'} P^{A}  (P-p)^{B} \cdot b_{A, B}\cdot  \Ex_{X} \left[X^{A - A'}\right] \Ex_{\sigma} \left[(J - \sigma)^{A'}\sigma^{ A - A'}\right] \hat{f}_{A, B}\\
    &= \sum_{A, B}\sum_{A' \subseteq A} Z^{A'} P^{A'} \sum_{\delta \subseteq A - A'} (P-p)^{\delta + B} \cdot b_{A', \delta + B} \cdot \frac{b_{A, B}}{b_{A', \delta + B}} p^{|A - A'  - \delta|} \\
    &\hspace{4.5cm} \Ex_{X} \left[X^{A - A'}\right] \Ex_{\sigma} \left[(J - \sigma)^{A'}\sigma^{ A - A'}\right] \hat{f}_{A, B}\\
    &= \sum_{A, B}\sum_{\substack{A' \subseteq A\\ \delta \subseteq A - A' } } h_{A', \delta + B}(Z \circ P) \cdot \frac{b_{A, B}}{b_{A', \delta + B}} p^{|A - A' - \delta|} \Ex_{X} \left[X^{A - A'}\right] \Ex_{\sigma} \left[(J - \sigma)^{A'}\sigma^{ A - A'}\right] \hat{f}_{A, B}\\
    &= \sum_{A', B'} h_{A', B'}(Z \circ P) \sum_{\substack{A, B\\
    A \supseteq A'\\
    B \subseteq B'\\
    B' - B \subseteq A - A'}} \frac{b_{A, B}}{b_{A', B'}} p^{|A - A' + B - B'|}\Ex_{X} \left[X^{A - A'}\right] \Ex_{\sigma} \left[(J - \sigma)^{A'}\sigma^{ A - A'}\right] \hat{f}_{A, B}.
\end{align*}
}

This means that in the expansion
\begin{align*}
    g(Z,P) = \sum_{A', B'} \hat{g}_{A', B'} \cdot h_{A', B'}(Z \circ P),
\end{align*}
the coefficient vector $\hat{g}$ satisfies $\hat{g} = M \hat{f}$, with $M_{(A', B'), (A,B)} = 0$ unless $A' \subseteq A$, $B \subseteq B'$, and $B' - B \subseteq A - A'$, and when these conditions are satisfied, then
\begin{align*}
    M_{(A', B'), (A, B)} &= \frac{b_{A, B}}{b_{A', B'}}\cdot p^{|A - A' + B - B'|}\cdot \Ex_{X} \left[X^{A - A'}\right] \Ex_{\sigma} \left[(J - \sigma)^{A'}\sigma^{ A - A'}\right].
\end{align*}
We remark that this matrix $M$ is different from the one that appeared in Section~\ref{sec:log-density-comp-rec-lower}. Notice that the indices can be arranged so that $M$ is upper triangular, and moreover the diagonal entries $M_{(A, B), (A, B)} = \Ex_{\sigma}\left[(J - \sigma)^A\right] = (1 - 2 q)^{|A|} \ne 0$ when $q < \frac{1}{2}$. Thus, $M$ is invertible.

Carrying out the same steps as in Section~\ref{sec:log-density-comp-rec-lower}, we may upper bound $\Corr_{\le D}$ by
\begin{align*}
    \text{Corr}_{\le D}(\sP) &= \sup_{f} \frac{\Ex_{\theta, Y}[f(Y) \theta_1]}{\sqrt{\Ex_{Y}[f(Y)^2]}}\\
    &\le \|c^\top M^{-1}\|\\
    &\equalscolon \|w\|, \label{ineq:corr-bound}
\end{align*}
for the same vector $c$ defined in Section~\ref{sec:log-density-comp-rec-lower} with entries $c_{A, B} = \Ex[h_{A, B}(Y) \theta_1]$ and $w^\top \colonequals c^\top M^{-1}$. While $c$ is the same vector, we shall use a different expression for its entries using the new collections of variables $Z,P,\sigma, X$ defined here:
\begin{align*}
    c_{A, B} &= \Ex[h_{A, B}(Y) \theta_1]\\
    &= b_{A, B}\cdot \Ex \left[P^A \left[(J - \sigma) \circ Z + \sigma \circ X\right]^A (P - p)^{B} \cdot \theta_1\right]\\
    &= b_{A, B}\cdot
    \Ex_{P}[P^A (P - p)^B ] \Ex_{Z,X,\sigma, \theta} \left[\left(\sum_{A' \subseteq A} X^{A'} Z^{A - A'}\sigma^{A'}(J - \sigma)^{A - A'}\right)\cdot \theta_1\right]\\
    &= \frac{1}{p^{|A|/2}} \cdot \frac{1}{(p(1-p))^{|B|/2}}\cdot p^{|A|}\One\{B = \emptyset\}\cdot (2q)^{|A|}\Ex_{X, \theta}\left[X^{A} \theta_1\right]\\
    &= \One\{B = \emptyset\}  p^{|A|/2} (2q)^{|A|}\Ex_{X, \theta}\left[X^{A} \theta_1\right]
\end{align*}

As before, we solve for $w$ from $w^\top M = c^\top$ and arrive at the following recursive formula for $w$:
\begin{align*}
    &w_{A, B} \cdot (1-2q)^{|A|}\\
    &= \One\{B = \emptyset\}  p^{|A|/2} (2q)^{|A|}\Ex_{X, \theta}\left[X^{A} \theta_1\right]\\
    &\quad- \sum_{\substack{A' \subseteq A\\ B' \supseteq B\\B' - B \subseteq A - A'\\
    (A', B')\ne (A, B)}} w_{A', B'} \cdot \frac{b_{A, B}}{b_{A', B' }}\cdot p^{|A - A' + B - B'|} (1-2q)^{|A'|}(2q)^{|A - A'|} \Ex_X\left[X^{A - A'}\right]
\end{align*}

{\allowdisplaybreaks
Setting $\kappa_{A, B} = \frac{w_{A, B}(1 - 2q)^{|A|}}{b_{A, B} }$, we get
\begin{align*}
    &\kappa_{A, B}\\
    &= \One\{B = \emptyset\} \frac{p^{|A|/2}}{b_{A, B}}\left(2q\right)^{|A|} \Ex_{X, \theta}\left[X^{A} \theta_1\right]\\
    &\quad- \sum_{\substack{A' \subseteq A\\ B' \supseteq B\\B' - B \subseteq A - A'\\
    (A', B')\ne (A, B)}} \kappa_{A', B'} \cdot p^{|A - A' + B - B'|} \left(2q\right)^{|A - A'|}\Ex_X\left[X^{A - A'}\right]\\
    &= \One\{B = \emptyset\} \left(2pq\right)^{|A|} \Ex_{X, \theta}\left[X^{A} \theta_1\right]\\
    &\quad- \sum_{\substack{A' \subseteq A\\ B' \supseteq B\\B' - B \subseteq A - A'\\
    (A', B')\ne (A, B)}} \kappa_{A', B'} \cdot p^{|A - A' + B - B'|} \left(2q\right)^{|A - A'|}\Ex_X\left[X^{A - A'}\right]
\end{align*}
}

Similar to Proposition~\ref{prop:w-zero-B}, Proposition~\ref{prop:w-zero-cc}, we have the following
\begin{proposition}\label{prop:k-zero-B}
    $\kappa_{A, B} = 0$ if $B \ne \emptyset$.
\end{proposition}

\begin{proposition}\label{prop:k-zero-cc}
    Suppose $A$ has any connected component containing an edge but not containing vertex $1$. Then $\kappa_{A, \emptyset} = 0$.
\end{proposition}

The first proposition follows from the identical reasoning, and we give a straightforward proof of Proposition~\ref{prop:k-zero-cc}, using the same strategy as in Appendix~\ref{app:pf:prop:w-zero-cc}.

\begin{proof}
    We prove by induction on $|A|$. Suppose $A$ has a connected component not containing vertex $1$, and suppose we have proved the statement for every $A'$ with fewer edges than $A$.

    Let $\delta$ be the connected component of $A$ containing vertex $1$, and note that $\delta$ could potentially be empty.

    Using the recursive formula for $\kappa$ and Proposition~\ref{prop:k-zero-B}, we have
    {\allowdisplaybreaks
    \begin{align*}
        &\kappa_{A, \emptyset}\\
        &= \left(2pq\right)^{|A|} \Ex_{X,\theta}\left[X^{A} \theta_1\right]\\
        &\quad- \sum_{A' \subsetneq A} \kappa_{A', \emptyset}   \left(2pq\right)^{|A - A'|}\Ex_X\left[X^{A - A'}\right] \\
        &= \left(2pq\right)^{|A|} \Ex_{X,\theta}\left[X^{A} \theta_1\right]\\
        &\quad- \sum_{\delta' \subseteq \delta} \kappa_{\delta', \emptyset}   \left(2pq\right)^{|A - \delta'|}\Ex_X\left[X^{A - \delta'}\right]
        \intertext{since by our inductive hypothesis, $\kappa_{A', \emptyset} = 0$ for any $A' \subsetneq A$ that is not fully contained in the connected component $\delta$,}
        &=\left(2pq\right)^{|A|} \Ex_{X,\theta}\left[X^{A} \theta_1\right]\\
        &\quad- \kappa_{\delta, \emptyset}   \left(2pq\right)^{|A - \delta|}\Ex_X\left[X^{A - \delta}\right]\\
        &\quad- \sum_{\delta' \subsetneq \delta} \kappa_{\delta', \emptyset}   \left(2pq\right)^{|A - \delta'|}\Ex_X\left[X^{A - \delta'}\right]\\
        &= \left(2pq\right)^{|A|} \Ex_{X,\theta}\left[X^{A} \theta_1\right]\\
        &\quad-  \left(2pq\right)^{|A - \delta|}\Ex_X\left[X^{A - \delta}\right] \cdot \left(2pq\right)^{|\delta|} \Ex_{X,\theta}\left[X^{\delta} \theta_1\right]\\
        &\quad + \left(2pq\right)^{|A - \delta|}\Ex_X\left[X^{A - \delta}\right]\sum_{\delta' \subsetneq \delta} \kappa_{\delta', \emptyset}   \left(2pq\right)^{|\delta - \delta'|}\Ex_X\left[X^{\delta - \delta'}\right]\\
        &\quad- \sum_{\delta' \subsetneq \delta} \kappa_{\delta', \emptyset}   \left(2pq\right)^{|A - \delta'|}\Ex_X\left[X^{A - \delta'}\right]\\
        &= 0,
    \end{align*}
    since $\Ex_{X,\theta}[X^A \theta_1] = \Ex_{X,\theta}[X^{A - \delta}]\Ex_{X,\theta}[X^\delta \theta_1]$ and for any $\delta' \subsetneq \delta$, $\Ex_{X,\theta}[X^{A - \delta'}] = \Ex_{X,\theta}[X^{A - \delta}]\Ex_{X,\theta}[X^{\delta-\delta'}]$.
    }
\end{proof}

Next we bound the terms $\kappa_{A, \emptyset}$ for $A$ a connected graph (aside from isolated vertices) containing vertex 1.
\begin{proposition}\label{prop:k-bound}
    We have
    \begin{align*}
        \kappa_{\emptyset, \emptyset} &= \frac{k}{n},\\
        |\kappa_{A, \emptyset}| &\le \left(2pq\right)^{|A|} \left(\frac{k}{n}\right)^{|V(A)|}\left(|A| + 1\right)^{|A|},
    \end{align*}
    for $A$ with $|A| \ge 1$.
\end{proposition}

\begin{proof}
    The case when $A = \emptyset$ is straightforward. Let us now consider $A \ne \emptyset$.
\begin{align*}
    &|\kappa_{A, \emptyset}|\\
    &= \Bigg|\left(2pq\right)^{|A|} \Ex_{X, \theta}\left[X^{A} \theta_1\right]\\
    &\quad- \sum_{A' \subsetneq A} \kappa_{A', \emptyset} \cdot \left(2pq\right)^{|A - A'|}\Ex_X\left[X^{A - A'}\right]\Bigg|\\
    &\le \left(2pq\right)^{|A|} \cdot\left|\Ex_{X, \theta}\left[X^{A} \theta_1\right]\right|\\
    &\quad+ \sum_{A' \subsetneq A} |\kappa_{A', \emptyset}| \cdot  \left(2pq\right)^{|A - A'|}\cdot \left|\Ex_X\left[X^{A - A'}\right]\right|\\
    &\le  \left(2pq\right)^{|A|} \left(\frac{k}{n}\right)^{|V(A) \cup \{1\}|}\\
    &\quad + \sum_{A' \subsetneq A } \left(2pq\right)^{|A - A'|} \left(\frac{k}{n}\right)^{|V(A - A')|} |\kappa_{A', \emptyset}|.
\end{align*}
Again, using a proof by induction following the steps as in the proof of \cite[Lemma 3.9]{schramm2022computational}, we get
\begin{align*}
    |\kappa_{A, \emptyset}| \le \left(2pq\right)^{|A|} \left(\frac{k}{n}\right)^{|V(A)|}\left(|A| + 1\right)^{|A|}.
\end{align*}
\end{proof}

With the tools above in hand, we get back to upper bound $\Corr_{\le D}(\sP)$ as
{\allowdisplaybreaks
\begin{align*}
    &\Corr_{\le D}(\sP)^2\\
    &\le \|w\|^2\\
    &= \sum_{A, B} w_{A, B}^2
    \intertext{recall that $\kappa_{A,B} = \frac{w_{A,B}(1-2q)^{|A|}}{b_{A,B}}$ and $q < \frac{1}{2}$, so we can write $w_{A,B}$ in terms of $\kappa_{A,B}$ and get}
    &= \sum_{A, B} \frac{b_{A, B}^2}{(1-2q)^{2|A|}} \cdot \kappa_{A, B}^2\\
    &= \sum_{A} \frac{b_{A}^2}{(1-2q)^{2|A|}} \cdot\kappa_{A, \emptyset}^2
    \intertext{by Proposition~\ref{prop:k-zero-B},}
    &= \sum_{A} \left(\frac{1}{p(1-2q)^2}\right)^{|A|} \kappa_{A, \emptyset}^2\\
    &\le \frac{k^2}{n^2} + \sum_{d=1}^D \sum_{h=0}^d \sum_{\substack{A \text{ connected}:\\
    1 \in V(A), |A| = d,\\
    |V(A)| = d+1-h}} \left(\frac{1}{p(1-2q)^2}\right)^d \left(2pq\right)^{2d} \left(\frac{k}{n}\right)^{2(d+1-h)} (d + 1)^{2d}
    \intertext{by Proposition~\ref{prop:k-zero-cc} and Proposition~\ref{prop:k-bound}, and that every connected $A$ satisfies $|V(A)| \le |A| + 1$,}
    &\le \frac{k^2}{n^2} + \sum_{d=1}^D \sum_{h=0}^d (dn)^d \left(\frac{d}{n}\right)^h \left(\frac{1}{p(1-2q)^2}\right)^d \left(2pq\right)^{2d} \left(\frac{k}{n}\right)^{2(d+1-h)} (d + 1)^{2d}
    \intertext{using the bound on the number of $A$ in Proposition~\ref{prop:connected-A-bound},}
    &\le \frac{k^2}{n^2} + \frac{k^2}{n^2}\sum_{d=1}^D \sum_{h=0}^d \left(d \cdot \frac{n}{k^2}\right)^h \left(4d(d+1)^2 \cdot \frac{pq^2}{(1-2q)^2}\cdot \frac{k^2}{n}\right)^d\\
    &\le \frac{k^2}{n^2}\sum_{h=0}^D \left(4D^2(D+1)^2 \cdot\frac{pq^2}{(1-2q)^2}\right)^h \sum_{d=h}^D    \left(4D(D+1)^2 \cdot\frac{pq^2}{(1-2q)^2} \cdot \frac{k^2}{n}\right)^{d-h}\\
    &\le \frac{k^2}{n^2} \cdot \frac{1}{\left(1 - 4D^2(D+1)^2 \cdot\frac{pq^2}{(1-2q)^2}\right)\left(1 - 4D(D+1)^2\cdot \frac{pq^2}{(1-2q)^2} \cdot \frac{k^2}{n}\right)},
\end{align*}
as desired. This finishes the proof.

\subsection{Proof of Proposition~\ref{prop:unique-acyclic}} \label{app:pf:prop:unique-acyclic}

Assume for contradiction that $S$ is not the unique maximum acyclic subset. Then, there exists $S' \subseteq [n]$ that is also acyclic, $|S'| = |S|$, and $S' \ne S$. In particular, $T = S' \setminus S$ is also an acyclic subset, and moreover it does not involve any vertex in the planted community $S$. It is known that w.h.p.~the maximum acyclic subset in a random tournament on $n$ vertices has size at most $(2+o(1))\log_2(n)$~\cite{spencer2008size}, and thus $|T| \le (2+o(1))\log_2(n)$. We will show that this leads to a contradiction as the directed edges between $T$ and $S$, which are outside the ranked community $S$, are atypical of randomly oriented edges.

Similar to the analysis presented in Section~\ref{sec:planted-ordered-clique-rec}, let us split $S = L \sqcup R$ such that $L$ contains half of the elements of $S$ that are ranked in the first half by $\pi$, and $R$ contains the other half of $S$ that are ranked in the second half by $\pi$. W.h.p., we have $|L| = \left(\frac{1}{2} + o(1)\right)k$ and $|R| = \left(\frac{1}{2} + o(1)\right)k$.

Since $S' = T \sqcup (S \cap S')$ is acyclic, for every $j \in T$, either $j$ is ranked after the elements in $L \cap S'$, or $j$ is ranked before the elements in $R \cap S'$. Therefore, one the following holds:
\begin{align*}
    \deg_{\textrm{in}}(j, L \cap S') &= |L \cap S'|, \quad \text{or}\\
    \deg_{\textrm{out}}(j, R \cap S') &= |R \cap S'|.
\end{align*}
Notice that since $T = S' \setminus S$ has size at most $(2+o(1))\log_2(n)$, we have
\begin{align*}
    |L \cap S'| = |L| - |L \setminus S'| \ge |L| - |S \setminus S'| = |L| - |S' \setminus S| &\ge \left(\frac{1}{2} - o(1)\right)k - (2+o(1))\log_2(n),\\
    |R \cap S'| = |R| - |R \setminus S'| \ge |R| - |S \setminus S'| = |R| - |S' \setminus S| &\ge \left(\frac{1}{2} - o(1)\right)k - (2+o(1))\log_2(n).
\end{align*}
This means that for every $j \in T$, either
\begin{align*}
    \deg_{\textrm{in}}(j, L) \ge \deg_{\textrm{in}}(j, L \cap S') &\ge \left(\frac{1}{2} - o(1)\right)k - (2+o(1))\log_2(n), \quad \text{or}\\
    \deg_{\textrm{out}}(j, R) \ge \deg_{\textrm{out}}(j, R \cap S') &\ge \left(\frac{1}{2} - o(1)\right)k - (2+o(1))\log_2(n).
\end{align*}
However, by application of Chernoff bound for $\deg_{\textrm{in}}(j, L)$ and $\deg_{\textrm{out}}(j, R)$, and a union bound over all $j \not\in S$, we can show that w.h.p.~for every $j \not\in S$,
\begin{align*}
    \deg_{\textrm{in}}(j, L),\, \deg_{\textrm{out}}(j, R) \le \frac{3}{8}k,
\end{align*}
which contradicts the degree bound above for $j \in T$, completing the proof.

\addcontentsline{toc}{section}{References}
\bibliographystyle{alpha}
\bibliography{main}

\end{document}